\documentclass{article}

\usepackage[utf8]{inputenc}
\usepackage[margin=3cm]{geometry}
\usepackage[english]{babel}
\usepackage{verbatim}

\usepackage{enumerate}
\usepackage{graphicx}
\usepackage{appendix}
\usepackage{enumitem}
\setenumerate[1]{label={(\arabic*)}} 

\usepackage{lmodern} 
\usepackage{amsmath}
\usepackage{amssymb}
\usepackage{physics}
\usepackage{braket}

\usepackage[svgnames]{xcolor}
\usepackage{tikz}

\usepackage{hyphsubst}
\usepackage{mathrsfs}

\usepackage{amsthm}

\usepackage{algorithm}

\usepackage{algpseudocode}
\usepackage{setspace}
\usepackage{centernot}
\usepackage{listings}
\usepackage{cancel}
\usepackage{caption}

\usepackage{dsfont}
\usepackage{blindtext}
\usepackage{adjustbox}
\usepackage{textcomp}
\usepackage{ytableau}

\graphicspath{ {./images/} }
\usepackage[autostyle,italian=guillemets]{csquotes}
\usepackage[colorlinks=true,
            linkcolor=Purple,
            citecolor=red,
            urlcolor=black,
            hyperfootnotes=false]{hyperref}

\usepackage{titlesec}
\titleformat{\paragraph}
{\normalfont\normalsize\bfseries}{\theparagraph}{1em}{}
\titlespacing*{\paragraph}
{0pt}{3.25ex plus 1ex minus .2ex}{1.5ex plus .2ex}

\def\titlerunning#1{\gdef\titrun{#1}}
\makeatletter
\def\author#1{\gdef\autrun{\def\and{\unskip, }#1}\gdef\@author{#1}}
\def\address#1{{\def\and{\\\hspace*{18pt}}\renewcommand{\thefootnote}{}%
\footnote {#1}}%
\markboth{\autrun}{\titrun}}
\makeatother
\def\email#1{\hspace*{4pt}{\em e-mail}: #1}
\def\MSC#1{{\renewcommand{\thefootnote}{}%
\footnote{\emph{Mathematics Subject Classification (2020):} #1}}}
\def\keywords#1{\par\medskip
\noindent\textbf{Keywords:} #1}

\theoremstyle{definition}
\newtheorem{definition}{Definition}[section]
\newtheorem{proposition}[definition]{Proposition}
\newtheorem{example}[definition]{Example}
\newtheorem{remark}[definition]{Remark}

\newtheorem{theorem}[definition]{Theorem}
\newtheorem{corollary}[definition]{Corollary}
\newtheorem{lemma}[definition]{Lemma}
\newtheorem{conjecture}[definition]{Conjecture}
\newtheorem{construction}[definition]{Construction}
\lstdefinestyle{mystyle}{
    basicstyle=\ttfamily\footnotesize,
    breakatwhitespace=false,         
    breaklines=true,                 
    captionpos=b,                    
    keepspaces=true,                                  
    showspaces=false,                
    showstringspaces=false,
    showtabs=false,                  
    tabsize=3
}
\newcommand{\mc}[1]{\mathcal{#1}}
\renewcommand{\set}[1]{\mc{#1}_{sm,d,j}}
\newcommand{\setq}[4]{\mc{#1}_{#2,#3,#4}}
\DeclareMathOperator{\rk}{rk}

\DeclareMathOperator{\srk}{srk}

\newcommand*{\tran}{^{\mkern-1.5mu\mathsf{T}}}

\newcommand{\F}{\mathbb F}
\newcommand{\Fq}{\mathbb F_q}
\newcommand{\Fqm}{\mathbb F_{q^m}}

\newcommand{\num}[2]{\nu_{\min{}}(\mathcal{#1},#2)}

\titlerunning{}
\date{}
\title{On the Etzion-Silberstein conjecture for block Ferrers
diagrams}
\author{Marco Calderini \and Marta Messia \and Alessandro Neri}

\begin{document}

\maketitle

\address{
\,\,M. Calderini: Department of Mathematics, University of Trento, Trento, Italy; \email{marco.calderini@unitn.it}
\and 
M. Messia:  Department of Mathematics and Data Science, Vrije Universiteit Brussel Pleinlaan 2, 1050 Brussel, Belgium; \email{marta.messia@vub.be}
\and 
A. Neri: Department of Mathematics and Applications ``R. Caccioppoli'', University of Naples ``Federico II'', Naples, Italy; \email{alessandro.neri@unina.it}
}

\bigskip

\MSC{Primary 11T71; Secondary 94B05}
\begin{abstract}
 { Ferrers diagram rank-metric codes are rank-metric codes with prescribed support, and their dimension is bounded from above by the Etzion--Silberstein bound. In this paper, we study this problem for block Ferrers diagrams, namely Ferrers diagrams whose dots are grouped into square blocks of a fixed size. Motivated by the diagonal construction for MDS-constructible Ferrers diagrams, we introduce the notion of MSRD-constructibility, where MDS codes on diagonals are replaced by maximum sum-rank distance (MSRD) codes on block diagonals. We show that MSRD-constructible pairs yield optimal Ferrers diagram rank-metric codes over sufficiently large finite fields. We then relate MSRD-constructibility of a block Ferrers diagram to MDS-constructibility of its contraction, proving an equivalence when the distance is compatible with the block size and giving lifting criteria in the general case. As a consequence, we obtain MSRD-constructibility for strictly block-monotone and initially block-convex diagrams. Finally, we prove a reduction to block triangular diagrams and use it to obtain new arbitrary-field cases of the Etzion--Silberstein conjecture for MSRD-constructible block Ferrers diagrams.}
\end{abstract}
%\tableofcontents
\section{Introduction}

\subsubsection*{Context}
Rank-metric codes were introduced by Delsarte~\cite{bib:Delsarte} and later studied by Gabidulin~\cite{bib:Gabidulin}. Since then, they have become a central object in coding theory, partly because of their applications to network coding \cite{bib:silva}, crisscross error correction \cite{bib:roth},  and cryptography \cite{gabidulin1991ideals}. A particularly important class of rank-metric codes is given by maximum rank distance codes, or MRD codes, namely codes attaining the Singleton-like bound for the rank metric.

Ferrers diagram rank-metric codes arise naturally in the construction of constant-dimension subspace codes. In the multilevel construction proposed by Etzion and Silberstein~\cite{bib:etzion1}, one first fixes a Schubert cell in a Grassmannian. The pivot positions defining the Schubert cell determine a Ferrers diagram, and the matrices parameterizing the points in that cell are precisely matrices supported on this diagram. Thus, within each Schubert cell, the problem of constructing large subspace codes leads to the problem of constructing rank-metric codes with a prescribed Ferrers-diagram support. Given a Ferrers diagram $\mc F$, an $[\mc F,k,d]_q$ Ferrers diagram rank-metric code is a rank-metric code whose codewords are supported on $\mc F$. Etzion and Silberstein proved a Singleton-like upper bound for the dimension of such codes~\cite{bib:etzion1}. More precisely, if  the minimum rank distance is $d$, then the dimension of any $[\mc F,k,d]_q$ code is at most $\nu_{\min}(\mc F,d)$, a combinatorial value obtained from the diagram $\mc F$ by erasing $d-1$ between rows and columns. Thus, throughout the paper, $\nu_{\min}(\mc F,d)$ denotes precisely the Etzion--Silberstein upper bound. Etzion and Silberstein conjectured that this bound is always attainable: for every Ferrers diagram $\mc F$ of order $n$, every $1\le d\le n$, and every finite field $\Fq$, there exists an $[\mc F,\nu_{\min}(\mc F,d),d]_q$ maximum Ferrers diagram rank-metric code.

Despite substantial progress, the Etzion--Silberstein conjecture remains open in full generality. Several families of Ferrers diagrams are known to satisfy the conjecture; see, for instance,~\cite{bib:liu} for an overview. Two main approaches have proved especially useful. The first constructs optimal Ferrers diagram codes as subcodes of MRD codes. A systematic account of such MRD-subcode constructions can be found in~\cite{bib:liu2}. This approach completely settles the case $d=2$ and is closely related to reduction arguments that pass from known optimal diagrams to new ones via inclusions and puncturing. These reduction ideas already appear in~\cite{bib:antrobus} and have recently been studied further in~\cite{bib:beeloo}.

The second approach is based on placing MDS codes, in the Hamming metric, on the diagonals of the Ferrers diagram. This method leads to the notion of MDS-constructible pairs. Roughly speaking, a pair $(\mc F,d)$ is MDS-constructible if the Etzion--Silberstein bound $\nu_{\min}(\mc F,d)$ is attained by the dimension predicted by the diagonal MDS-code construction.
The idea was developed in the study of optimal Ferrers diagram codes in~\cite{bib:etzion} and was later refined in the work of Antrobus and Gluesing-Luerssen~\cite{bib:antrobus}. More recently, Neri and Stanojkovski proved that every MDS-constructible pair admits an optimal Ferrers diagram rank-metric code over every finite field~\cite{bib:neri}. Their proof combines the diagonal construction with MRD-subcode techniques, and also settles the conjecture for monotone and initially convex Ferrers diagrams.

The purpose of this paper is to investigate an analogous diagonal construction for block Ferrers diagrams. These are Ferrers diagrams whose dots are grouped into square blocks of a fixed size $m$. Block Ferrers diagrams already appear naturally in previous works on the Etzion--Silberstein conjecture. In particular, Neri and Stanojkovski studied block diagrams in the case where the block size is a prime power and related them to contractions of Ferrers diagrams over fields of suitable characteristic~\cite{bib:neri}. Further constructions of optimal Ferrers diagram rank-metric codes  for block diagrams  over sufficiently large fields  were obtained by Pratihar and Randrianarisoa using automorphisms of rational function fields~\cite{rakhi2023}. More recently, Liu obtained additional optimal codes from MRD codes, including some special results for block triangular Ferrers diagrams~\cite{bib:Liu2023} over small fields. These works show that block Ferrers diagrams form a natural and rich testing ground for the Etzion--Silberstein conjecture.

In the block setting, the diagonal MDS-code construction has a natural sum-rank analogue. Indeed, a block diagonal consists of several square matrix blocks, and the relevant metric on such a tuple of blocks is the sum-rank metric~\cite{bib:nobrega}. Therefore, the role of MDS codes is naturally replaced by that of maximum sum-rank distance codes, or MSRD codes~\cite{bib:martinez}. This leads to the notion of MSRD-constructibility for block Ferrers diagrams: an $m$-block Ferrers diagram $\mc F$ is MSRD-constructible for distance $d$ when the Etzion--Silberstein bound $\nu_{\min}(\mc F,d)$ coincides with the dimension obtained by placing MSRD codes on the block diagonals of $\mc F$.

\subsubsection*{Our contributions}

This paper starts from the observation that the diagonal construction for MDS-constructible Ferrers diagrams has a natural block analogue. In the classical setting, one places MDS codes on the diagonals of a Ferrers diagram. When the diagram is made of square blocks, however, a diagonal is no longer just a set of entries: it is a sequence of matrix blocks. The natural metric on such a sequence is the sum-rank metric, and so MDS codes should be replaced by MSRD codes.

We formalize this idea by introducing the notion of MSRD-constructibility for block Ferrers diagrams. We then show that this notion has the expected meaning: whenever the MSRD diagonal construction reaches the Etzion--Silberstein bound $\nu_{\min}(\mc F,d)$, it gives an optimal Ferrers diagram rank-metric code. This is the content of Lemma~\ref{le:dimMSRD} and Theorem~\ref{th:optMSRD}. In this sense, MSRD-constructibility is the block version of MDS-constructibility.

The next question is how this new notion is related to the old one. Every $m$-block Ferrers diagram has a contraction, obtained by replacing each full $m\times m$ block with a single dot. We prove that, when the distance is compatible with the block structure, MSRD-constructibility of the block diagram is exactly the same as MDS-constructibility of its contraction (Theorem~\ref{th:conndiv}). Thus, in this case, the block theory does not merely resemble the classical theory: it is equivalent to it after contraction.

When the distance is not aligned with the block size, the relation between a block diagram and its contraction becomes more delicate, since Singleton deletions may cut through individual blocks. We give criteria ensuring that MDS-constructibility of the contraction still lifts to MSRD-constructibility of the block expansion (Theorems~\ref{th:connection} and~\ref{th:connection2}), and use them to transfer natural families from the classical theory to the block setting. In particular, we prove that strictly $m$-monotone diagrams, and by adjunction initially $m$-convex diagrams, are MSRD-constructible for all relevant distances (Corollary~\ref{th:monoMSRD1}).

The final part of the paper addresses a second issue. The diagonal MSRD construction requires the existence of suitable MSRD codes, and therefore gives optimal Ferrers diagram codes only under a field-size assumption. To move beyond this restriction, we prove a reduction theorem inspired by the corresponding strategy for MDS-constructible diagrams. Namely, we show that the existence of optimal codes for an MSRD-constructible block Ferrers diagram can be reduced to the existence of optimal codes for a suitable block triangular diagram (Theorem~\ref{theo:tri}).

This reduction shifts the problem to block triangular diagrams, which play the same role here as triangular Ferrers diagrams in the classical MDS setting. We then construct optimal codes for block triangular diagrams in several extremal distance ranges (Theorem~\ref{th:triangular_conj}). Combining these constructions with the reduction theorem, we obtain new arbitrary-field cases of the Etzion--Silberstein conjecture for MSRD-constructible block Ferrers diagrams (Corollary~\ref{cor:triangular_construction}).

\subsubsection*{Outline}
The paper is organized as follows. In Section~\ref{sec:preliminaries} we recall the necessary background on Ferrers diagram rank-metric codes, MDS-constructible pairs, and
sum-rank-metric codes. In Section~\ref{sec:MSRD-contruction} we introduce MSRD-constructibility and
describe the MSRD diagonal construction.  Section~\ref{sec:MDSvsMSRD-constructibility} explores the relation
between MDS-constructibility of a Ferrers diagram and MSRD-constructibility of
its block expansion. Finally, in Section~\ref{sec:triangular_reduction} we prove a reduction theorem for
MSRD-constructible block diagrams and treat several families of block triangular
Ferrers diagrams, obtaining new optimal Ferrers diagram rank-metric codes.

\subsubsection*{Acknowledgments}
This research has been partially supported by the Italian National Group for Algebraic and Geometric Structures and their Applications (GNSAGA - INdAM) and by VUB-OZR project "Discrete Structures, and their applications in Data Science” (OZR3637)”.  
M. Calderini is supported by MUR-Italy via
PRIN 2022RFAZCJ ‘Algebraic methods in cryptanalysis’. A. Neri is supported by the INdAM - GNSAGA Project CUP E53C24001950001  ``Noncommutative polynomials in coding theory''. M. Messia is supported by FWO research project “Finite Geometry Applications to Coding Theory and Cryptography” (FWO grant nr. G0AJI25N).

\section{Preliminaries}\label{sec:preliminaries}

In this section, we introduce the preliminary notions and results that we will use throughout the paper. In particular, we give the necessary background on Ferrers diagrams, rank-metric codes, and sum-rank-metric codes.

\subsection{Ferrers diagram rank-metric codes}

Throughout the paper, the set of natural numbers $\mathbb N$ will be $\{1,2,3,\ldots\}$.
For a natural number $n$, we denote by $[n]$ the set $\{1,2,\ldots,n\}$ and we set $[0]=\emptyset$ by convention. We will also use $q$ to denote a prime power, and $\F_q$ is the finite field with $q$ elements. 

For every $m,n\in{\mathbb{N}}$, denote by $\mathbb{F}_q^{m\times{n}}$ the $\Fq$-vector  space of all $m\times{n}$ matrices with entries in $\mathbb{F}_q$. We can define the rank-metric on $\mathbb{F}_q^{m\times{n}}$ through the map
$$
\begin{array}{rccl}
d_{\rk}:&\mathbb{F}_q^{m\times{n}}\times \mathbb{F}_q^{m\times{n}}& \longrightarrow &\mathbb{Z}_{\ge0}, \\ & (A,B)&\longmapsto & \rk (A-B).
\end{array}
$$

\begin{definition}
An $[m\times{n},k,d]_q$ \textit{rank-metric code} is a $k$-dimensional linear subspace $\mathcal{C}$ of $\mathbb{F}_q^{m\times{n}}$ equipped with the rank distance. The parameter $d$ is the \emph{minimum rank distance} of $\mathcal{C}$, defined as $d=\min\{\,\rk(A)\,:\,A\in\mc{C},\,A\neq0\,\}$.
\end{definition}

As shown by Delsarte \cite{bib:Delsarte}, the parameters of a rank-metric code are related through the well-known Singleton-like bound:
$$
k \le \min\{m(n- d + 1), n(m- d + 1)\}. 
$$
Codes meeting the bound with equality are called \emph{maximum rank distance} (MRD) codes.

\begin{definition}
\label{def:ferr}
A (top-right justified) \textit{Ferrers diagram} $\mathcal{F}$ of order $n$ is a subset of $[n]^2$ with the following properties
\begin{enumerate}
\item if $(i,j)\in{\mathcal{F}}$ and $j<n$, then $(i,j+1)\in{\mathcal{F}}$,
\item if $(i,j)\in{\mathcal{F}}$ and $i>1$, then $(i-1,j)\in{\mathcal{F}}$.
\end{enumerate}
\end{definition}
Observe that we can graphically represent each diagram $\mathcal{F}$ as an $n\times n$ grid where we write a dot for every nonzero entry, as in the following example.

A Ferrers diagram $\mathcal{F}$ can be addressed by its columns using the notation $c_j=|\{\,(i,j)\,:\,1\le{i}\le{n},\,(i,j)\in{\mathcal{F}}\,\}|$, $j\in{[n]}$. Then, writing $\mathcal{F}=[c_1,\ldots,c_n]$, one gets a complete description of the diagram, given the number of dots in every column and the top-right aligned property.

\begin{example} The Ferrers diagram of order $5$, $\mc{F}=[1,2,3,3,5]$.
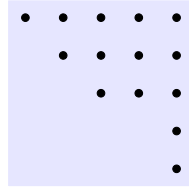
\begin{center}
  \begin{tikzpicture}[scale=0.5]
\draw[help lines, very thick, white, fill=blue!10] (0,1) -- (0,-4) -- (5,-4)--(5,1)--(0,1);

\draw[fill=black] (0.5,0.5) circle (0.1cm);
\draw[fill=black] (1.5,0.5) circle (0.1cm);
\draw[fill=black] (2.5,0.5) circle (0.1cm);
\draw[fill=black] (3.5,0.5) circle (0.1cm);
\draw[fill=black] (4.5,0.5) circle (0.1cm);

\draw[fill=black] (1.5,-0.5) circle (0.1cm);
\draw[fill=black] (2.5,-0.5) circle (0.1cm);
\draw[fill=black] (3.5,-0.5) circle (0.1cm);
\draw[fill=black] (4.5,-0.5) circle (0.1cm);

\draw[fill=black] (2.5,-1.5) circle (0.1cm);
\draw[fill=black] (3.5,-1.5) circle (0.1cm);
\draw[fill=black] (4.5,-1.5) circle (0.1cm);

\draw[fill=black] (4.5,-2.5) circle (0.1cm);
\draw[fill=black] (4.5,-3.5) circle (0.1cm);
\end{tikzpicture}    
\end{center}
 
\end{example}

For a square grid $[n]^2$, we introduce the following notation for denoting its rows, columns and diagonals.
\begin{enumerate}
    \item Rows: $ R_i:=\{\,(i,j)\in[n]^2\,:\,j\in[n]\,\},\quad i\in[n]$.
    \item Columns:
        $C_j:=\{\,(i,j)\in [n]^2\,:\,i\in[n]\,\},\quad j\in[n]$.

    \item Diagonals:
       $ D_i:=\{\,(l,n-i+l)\in [n]^2\,:\,1\leq{l}\leq{i}\,\},\quad i\in[n]$.
\end{enumerate}
Therefore, for a Ferrers diagram $\mc{F}$ of order $n$, its rows, columns and upper-diagonals are given by $R_i\cap\mc{F}$, $C_i\cap\mc{F}$ and  $D_i\cap\mc{F}$, respectively, for $i\in [n]$.
\begin{example}
The Ferrers diagram $\mc{F}=[1,2,3,4,4,6]$ of order $6$ with its rows, columns and diagonals highlighted.
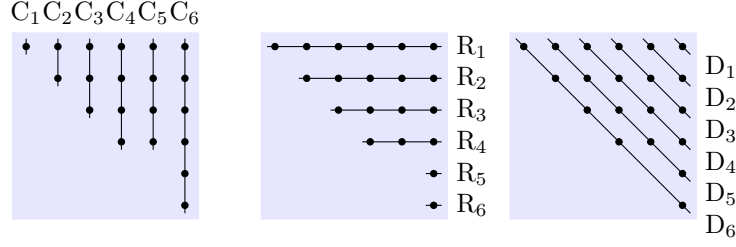
\begin{center}
\begin{tikzpicture}[scale=0.42]
\draw[help lines, very thick, white, fill=blue!10] (1,1) -- (1,-5) -- (7,-5)--(7,1)--(1,1);

\draw[black](1.5,0.75)--(1.5,0.25);
\draw[black](2.5,0.75)--(2.5,-0.75);
\draw[black](3.5,0.75)--(3.5,-1.75);
\draw[black](4.5,0.75)--(4.5,-2.75);
\draw[black](5.5,0.75)--(5.5,-2.75);
\draw[black](6.5,0.75)--(6.5,-4.75);

\draw[fill=black] (1.5,0.5) circle (0.1cm)node[black,above,yshift=5pt]{C$_1$};
\draw[fill=black] (2.5,0.5) circle (0.1cm)node[black,above,yshift=5pt]{C$_2$};
\draw[fill=black] (3.5,0.5) circle (0.1cm)node[black,above,yshift=5pt]{C$_3$};
\draw[fill=black] (4.5,0.5) circle (0.1cm)node[black,above,yshift=5pt]{C$_4$};
\draw[fill=black] (5.5,0.5) circle (0.1cm)node[black,above,yshift=5pt]{C$_5$};
\draw[fill=black] (6.5,0.5) circle (0.1cm)node[black,above,yshift=5pt]{C$_6$};

\draw[fill=black] (2.5,-0.5) circle (0.1cm);
\draw[fill=black] (3.5,-0.5) circle (0.1cm);
\draw[fill=black] (4.5,-0.5) circle (0.1cm);
\draw[fill=black] (5.5,-0.5) circle (0.1cm);
\draw[fill=black] (6.5,-0.5) circle (0.1cm);

\draw[fill=black] (3.5,-1.5) circle (0.1cm);
\draw[fill=black] (4.5,-1.5) circle (0.1cm);
\draw[fill=black] (5.5,-1.5) circle (0.1cm);
\draw[fill=black] (6.5,-1.5) circle (0.1cm);

\draw[fill=black] (4.5,-2.5) circle (0.1cm);
\draw[fill=black] (5.5,-2.5) circle (0.1cm);
\draw[fill=black] (6.5,-2.5) circle (0.1cm);

\draw[fill=black] (6.5,-3.5) circle (0.1cm);

\draw[fill=black] (6.5,-4.5) circle (0.1cm)node[white,below right,xshift=5pt]{D$_6$};
\end{tikzpicture}
\begin{tikzpicture}[scale=0.42]
\draw[help lines, very thick, white, fill=blue!10] (1,1) -- (1,-5) -- (7,-5)--(7,1)--(1,1);

\draw[black](1.25,0.5)--(6.75,0.5);
\draw[black](2.25,-0.5)--(6.75,-0.5);
\draw[black](3.25,-1.5)--(6.75,-1.5);
\draw[black](4.25,-2.5)--(6.75,-2.5);
\draw[black](6.25,-3.5)--(6.75,-3.5);
\draw[black](6.25,-4.5)--(6.75,-4.5);

\draw[fill=black] (1.5,0.5) circle (0.1cm);
\draw[fill=black] (2.5,0.5) circle (0.1cm);
\draw[fill=black] (3.5,0.5) circle (0.1cm);
\draw[fill=black] (4.5,0.5) circle (0.1cm);
\draw[fill=black] (5.5,0.5) circle (0.1cm);
\draw[fill=black] (6.5,0.5) circle (0.1cm)node[right,xshift=5pt]{R$_1$};

\draw[fill=black] (2.5,-0.5) circle (0.1cm);
\draw[fill=black] (3.5,-0.5) circle (0.1cm);
\draw[fill=black] (4.5,-0.5) circle (0.1cm);
\draw[fill=black] (5.5,-0.5) circle (0.1cm);
\draw[fill=black] (6.5,-0.5) circle (0.1cm)node[right,xshift=5pt]{R$_2$};

\draw[fill=black] (3.5,-1.5) circle (0.1cm);
\draw[fill=black] (4.5,-1.5) circle (0.1cm);
\draw[fill=black] (5.5,-1.5) circle (0.1cm);
\draw[fill=black] (6.5,-1.5) circle (0.1cm)node[right,xshift=5pt]{R$_3$};

\draw[fill=black] (4.5,-2.5) circle (0.1cm);
\draw[fill=black] (5.5,-2.5) circle (0.1cm);
\draw[fill=black] (6.5,-2.5) circle (0.1cm)node[right,xshift=5pt]{R$_4$};

\draw[fill=black] (6.5,-3.5) circle (0.1cm)node[right,xshift=5pt]{R$_5$};

\draw[fill=black] (6.5,-4.5) circle (0.1cm)node[right,xshift=5pt]{R$_6$}node[white,below right,xshift=5pt]{D$_6$};
\end{tikzpicture}
\begin{tikzpicture}[scale=0.42]
\draw[help lines, very thick, white, fill=blue!10] (1,1) -- (1,-5) -- (7,-5)--(7,1)--(1,1);

\draw[fill=black] (1.5,0.5) circle (0.1cm);
\draw[fill=black] (2.5,0.5) circle (0.1cm);
\draw[fill=black] (3.5,0.5) circle (0.1cm);
\draw[fill=black] (4.5,0.5) circle (0.1cm);
\draw[fill=black] (5.5,0.5) circle (0.1cm);
\draw[fill=black] (6.5,0.5) circle (0.1cm)node[below right,xshift=5pt]{D$_1$};

\draw[fill=black] (2.5,-0.5) circle (0.1cm);
\draw[fill=black] (3.5,-0.5) circle (0.1cm);
\draw[fill=black] (4.5,-0.5) circle (0.1cm);
\draw[fill=black] (5.5,-0.5) circle (0.1cm);
\draw[fill=black] (6.5,-0.5) circle (0.1cm)node[below right,xshift=5pt]{D$_2$};

\draw[fill=black] (3.5,-1.5) circle (0.1cm);
\draw[fill=black] (4.5,-1.5) circle (0.1cm);
\draw[fill=black] (5.5,-1.5) circle (0.1cm);
\draw[fill=black] (6.5,-1.5) circle (0.1cm)node[below right,xshift=5pt]{D$_3$};

\draw[fill=black] (4.5,-2.5) circle (0.1cm);
\draw[fill=black] (5.5,-2.5) circle (0.1cm);
\draw[fill=black] (6.5,-2.5) circle (0.1cm)node[below right,xshift=5pt]{D$_4$};

\draw[fill=black] (6.5,-3.5) circle (0.1cm)node[below right,xshift=5pt]{D$_5$};

\draw[fill=black] (6.5,-4.5) circle (0.1cm)node[below right,xshift=5pt]{D$_6$};

\draw[black] (1.25,0.75)--(6.75,-4.75);
\draw[black] (2.25,0.75)--(6.75,-3.75);
\draw[black] (3.25,0.75)--(6.75,-2.75);
\draw[black] (4.25,0.75)--(6.75,-1.75);
\draw[black] (5.25,0.75)--(6.75,-0.75);
\draw[black] (6.25,0.75)--(6.75,0.25);
\end{tikzpicture}
\end{center}

\end{example}

We now introduce the notion of adjoint of a Ferrers diagram, which corresponds to reflection across the antidiagonal.

\begin{definition}
Let $\mathcal{F}$ be a Ferrers diagram of order $n$. The \textit{adjoint} Ferrers diagram $\mathcal{F}^{\tran}$ is defined as
\begin{equation*}
\mathcal{F}^{\tran}=\{\,(n+1-j,n+1-i) \,:\,(i,j)\in{\mathcal{F}}\,\}.
\end{equation*}
\end{definition}

For a given Ferrers diagram $\mc{F} = [c_1, \ldots, c_n]$ of order $n$, we denote by $\mathbb{F}_q^{\mc{F}}$ the space of $n\times n$
matrices over $\mathbb{F}_q$ supported on $\mc{F}$, that is,
$$
\mathbb{F}_q^{\mc{F}}=\{M = (m_{i,j}) \in \mathbb{F}_q^{n\times n} \,:\, (i,j)\notin\mc{F}\text{ implies }m_{i,j} = 0\}.
$$

\begin{definition}
Let $\mathcal{F}$ be a Ferrers diagram of order $n$. An $[\mathcal{F},k,d]_q$ \textit{Ferrers diagram rank-metric code} is a $k$-dimensional linear subspace $\mathcal{C}$ of $\mathbb{F}_q^{\mathcal{F}}$, endowed with the rank distance. The parameter $d$ is the \textit{minimum rank distance} of $\mathcal{C}$, defined as
\begin{equation*}
d:=\min\{\,\rk(A) \,:\,A\in{\mathcal{C}}, A\ne{0}\,\}.
\end{equation*}
\end{definition}

\begin{definition}
Let $\mathcal{F}=[c_1,\ldots,c_n]$ be a Ferrers diagram of order $n$ and let $d\in[n]$. For $j\in{\{\,0,\ldots,d-1\,\}}$ we define
\begin{equation*}
\nu_j(\mathcal{F},d)=\sum_{i=1}^{n-j}\max\{0,\,c_i-d+1+j\}.
\end{equation*}
Furthermore, we set $\nu_{\min}(\mathcal{F},d)=\min\{\nu_0(\mc{F},d),\ldots,\nu_{d-1}(\mc{F},d)\}$.
\end{definition}

Graphically speaking, $\nu_j(\mc{F},d)$ counts the dots that remain in the Ferrers diagram after removing the first $d-j-1$ rows and the last $j$ columns, as we can see in the following small example.

\begin{example}
\label{ex:quattro}
Consider $\mc{F}=[1,2,3,4,4,6]$ and $d=4$. Here we have in red the dots that contribute to each $\nu_j(\mc{F},d)$ and the removed rows and columns are dashed.
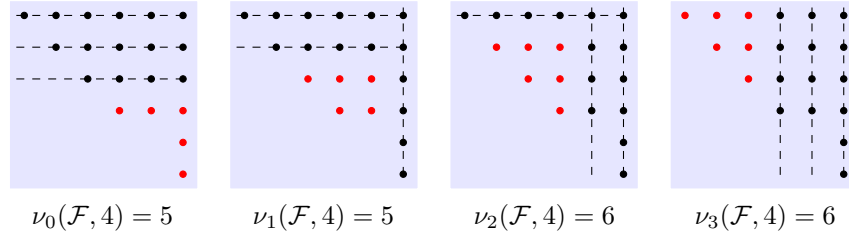
\begin{center}

\begin{tikzpicture}[scale=0.42]
\draw[help lines, very thick, white, fill=blue!10] (1,1) -- (1,-5) -- (7,-5)--(7,1)--(1,1);

% \draw[black,dashed](1.5,0.75)--(1.5,-0.75);
% \draw[black,dashed](2.5,0.75)--(2.5,-0.75);
% \draw[black,dashed](3.5,0.75)--(3.5,-1.75);
% \draw[black,dashed](4.5,0.75)--(4.5,-2.75);
% \draw[black,dashed](5.5,0.75)--(5.5,-2.75);
% \draw[black,dashed](6.5,0.75)--(6.5,-4.75);

\draw[black,dashed](1.25,0.5)--(6.75,0.5);
\draw[black,dashed](1.25,-0.5)--(6.75,-0.5);
\draw[black,dashed](1.25,-1.5)--(6.75,-1.5);
% \draw[black,dashed](4.25,-2.5)--(6.75,-2.5);
% \draw[black,dashed](6.25,-3.5)--(6.75,-3.5);
% \draw[black,dashed](6.25,-4.5)--(6.75,-4.5);

\draw[fill=black] (1.5,0.5) circle (0.1cm);
\draw[fill=black] (2.5,0.5) circle (0.1cm);
\draw[fill=black] (3.5,0.5) circle (0.1cm);
\draw[fill=black] (4.5,0.5) circle (0.1cm);
\draw[fill=black] (5.5,0.5) circle (0.1cm);
\draw[fill=black] (6.5,0.5) circle (0.1cm);

\draw[fill=black] (2.5,-0.5) circle (0.1cm);
\draw[fill=black] (3.5,-0.5) circle (0.1cm);
\draw[fill=black] (4.5,-0.5) circle (0.1cm);
\draw[fill=black] (5.5,-0.5) circle (0.1cm);
\draw[fill=black] (6.5,-0.5) circle (0.1cm);

\draw[fill=black] (3.5,-1.5) circle (0.1cm);
\draw[fill=black] (4.5,-1.5) circle (0.1cm);
\draw[fill=black] (5.5,-1.5) circle (0.1cm);
\draw[fill=black] (6.5,-1.5) circle (0.1cm);

\draw[red,fill=red] (4.5,-2.5) circle (0.1cm);
\draw[red,fill=red] (5.5,-2.5) circle (0.1cm);
\draw[red,fill=red] (6.5,-2.5) circle (0.1cm);

\draw[red,fill=red] (6.5,-3.5) circle (0.1cm);

\draw[red,fill=red] (6.5,-4.5) circle (0.1cm);
\end{tikzpicture}\quad
\begin{tikzpicture}[scale=0.42]
\draw[help lines, very thick, white, fill=blue!10] (1,1) -- (1,-5) -- (7,-5)--(7,1)--(1,1);

% \draw[black,dashed](1.5,0.75)--(1.5,-0.75);
% \draw[black,dashed](2.5,0.75)--(2.5,-0.75);
% \draw[black,dashed](3.5,0.75)--(3.5,-1.75);
% \draw[black,dashed](4.5,0.75)--(4.5,-2.75);
% \draw[black,dashed](5.5,0.75)--(5.5,-2.75);
\draw[black,dashed](6.5,0.75)--(6.5,-4.75);

\draw[black,dashed](1.25,0.5)--(6.75,0.5);
\draw[black,dashed](1.25,-0.5)--(6.75,-0.5);
% \draw[black,dashed](3.25,-1.5)--(6.75,-1.5);
% \draw[black,dashed](4.25,-2.5)--(6.75,-2.5);
% \draw[black,dashed](6.25,-3.5)--(6.75,-3.5);
% \draw[black,dashed](6.25,-4.5)--(6.75,-4.5);

\draw[fill=black] (1.5,0.5) circle (0.1cm);
\draw[fill=black] (2.5,0.5) circle (0.1cm);
\draw[fill=black] (3.5,0.5) circle (0.1cm);
\draw[fill=black] (4.5,0.5) circle (0.1cm);
\draw[fill=black] (5.5,0.5) circle (0.1cm);
\draw[fill=black] (6.5,0.5) circle (0.1cm);

\draw[fill=black] (2.5,-0.5) circle (0.1cm);
\draw[fill=black] (3.5,-0.5) circle (0.1cm);
\draw[fill=black] (4.5,-0.5) circle (0.1cm);
\draw[fill=black] (5.5,-0.5) circle (0.1cm);
\draw[fill=black] (6.5,-0.5) circle (0.1cm);

\draw[red,fill=red] (3.5,-1.5) circle (0.1cm);
\draw[red,fill=red] (4.5,-1.5) circle (0.1cm);
\draw[red,fill=red] (5.5,-1.5) circle (0.1cm);
\draw[fill=black] (6.5,-1.5) circle (0.1cm);

\draw[red,fill=red] (4.5,-2.5) circle (0.1cm);
\draw[red,fill=red] (5.5,-2.5) circle (0.1cm);
\draw[fill=black] (6.5,-2.5) circle (0.1cm);

\draw[fill=black] (6.5,-3.5) circle (0.1cm);

\draw[fill=black] (6.5,-4.5) circle (0.1cm);
\end{tikzpicture}\quad
\begin{tikzpicture}[scale=0.42]
\draw[help lines, very thick, white, fill=blue!10] (1,1) -- (1,-5) -- (7,-5)--(7,1)--(1,1);

% \draw[black,dashed](1.5,0.75)--(1.5,-0.75);
% \draw[black,dashed](2.5,0.75)--(2.5,-0.75);
% \draw[black,dashed](3.5,0.75)--(3.5,-1.75);
% \draw[black,dashed](4.5,0.75)--(4.5,-2.75);
\draw[black,dashed](5.5,0.75)--(5.5,-4.75);
\draw[black,dashed](6.5,0.75)--(6.5,-4.75);

\draw[black,dashed](1.25,0.5)--(6.75,0.5);
% \draw[black,dashed](1.25,-0.5)--(6.75,-0.5);
% \draw[black,dashed](3.25,-1.5)--(6.75,-1.5);
% \draw[black,dashed](4.25,-2.5)--(6.75,-2.5);
% \draw[black,dashed](6.25,-3.5)--(6.75,-3.5);
% \draw[black,dashed](6.25,-4.5)--(6.75,-4.5);

\draw[fill=black] (1.5,0.5) circle (0.1cm);
\draw[fill=black] (2.5,0.5) circle (0.1cm);
\draw[fill=black] (3.5,0.5) circle (0.1cm);
\draw[fill=black] (4.5,0.5) circle (0.1cm);
\draw[fill=black] (5.5,0.5) circle (0.1cm);
\draw[fill=black] (6.5,0.5) circle (0.1cm);

\draw[red,fill=red] (2.5,-0.5) circle (0.1cm);
\draw[red,fill=red] (3.5,-0.5) circle (0.1cm);
\draw[red,fill=red] (4.5,-0.5) circle (0.1cm);
\draw[fill=black] (5.5,-0.5) circle (0.1cm);
\draw[fill=black] (6.5,-0.5) circle (0.1cm);

\draw[red,fill=red] (3.5,-1.5) circle (0.1cm);
\draw[red,fill=red] (4.5,-1.5) circle (0.1cm);
\draw[fill=black] (5.5,-1.5) circle (0.1cm);
\draw[fill=black] (6.5,-1.5) circle (0.1cm);

\draw[red,fill=red] (4.5,-2.5) circle (0.1cm);
\draw[fill=black] (5.5,-2.5) circle (0.1cm);
\draw[fill=black] (6.5,-2.5) circle (0.1cm);

\draw[fill=black] (6.5,-3.5) circle (0.1cm);

\draw[fill=black] (6.5,-4.5) circle (0.1cm);
\end{tikzpicture}\quad
\begin{tikzpicture}[scale=0.42]
\draw[help lines, very thick, white, fill=blue!10] (1,1) -- (1,-5) -- (7,-5)--(7,1)--(1,1);

% \draw[black,dashed](1.5,0.75)--(1.5,-0.75);
% \draw[black,dashed](2.5,0.75)--(2.5,-0.75);
% \draw[black,dashed](3.5,0.75)--(3.5,-1.75);
\draw[black,dashed](4.5,0.75)--(4.5,-4.75);
\draw[black,dashed](5.5,0.75)--(5.5,-4.75);
\draw[black,dashed](6.5,0.75)--(6.5,-4.75);

\draw[red,fill=red] (1.5,0.5) circle (0.1cm);
\draw[red,fill=red] (2.5,0.5) circle (0.1cm);
\draw[red,fill=red] (3.5,0.5) circle (0.1cm);
\draw[fill=black] (4.5,0.5) circle (0.1cm);
\draw[fill=black] (5.5,0.5) circle (0.1cm);
\draw[fill=black] (6.5,0.5) circle (0.1cm);

\draw[red,fill=red] (2.5,-0.5) circle (0.1cm);
\draw[red,fill=red] (3.5,-0.5) circle (0.1cm);
\draw[fill=black] (4.5,-0.5) circle (0.1cm);
\draw[fill=black] (5.5,-0.5) circle (0.1cm);
\draw[fill=black] (6.5,-0.5) circle (0.1cm);

\draw[red,fill=red] (3.5,-1.5) circle (0.1cm);
\draw[fill=black] (4.5,-1.5) circle (0.1cm);
\draw[fill=black] (5.5,-1.5) circle (0.1cm);
\draw[fill=black] (6.5,-1.5) circle (0.1cm);

\draw[fill=black] (4.5,-2.5) circle (0.1cm);
\draw[fill=black] (5.5,-2.5) circle (0.1cm);
\draw[fill=black] (6.5,-2.5) circle (0.1cm);

\draw[fill=black] (6.5,-3.5) circle (0.1cm);

\draw[fill=black] (6.5,-4.5) circle (0.1cm);
\end{tikzpicture}

$\nu_0(\mc{F},4)=5 \qquad\quad \nu_1(\mc{F},4)=5\qquad\quad \nu_2(\mc{F},4)=6 \qquad\quad\nu_3(\mc{F},4)=6$
\end{center}
\end{example}

\begin{definition}
Let $\mc{F}=[c_1,\ldots,c_n]$ be a Ferrers diagram of order $n$, let $d\in[n]$ and let $j\in\{\,0,\ldots,d-1\,\}$. The pair $(\mc{F},d)$ is called $j$\textit{-Singleton} if
\begin{equation*}
    \num{F}{d}=\nu_j(\mathcal{F},d).
\end{equation*}
\end{definition}
In words, a pair $(\mc{F},d)$ is $j$-Singleton if to obtain $\num{F}{d}$ the last $j$ columns and the first $d-j-1$ rows of the diagram are removed. It is easy to see that there are pairs $(\mc{F},d)$ that are $j$-Singleton for multiple values of $j$.

The upper-triangular Ferrers diagram of order $n$ will be denoted by
\begin{equation*}
\mc{T}_n=\{\,(i,j)\,:\,i,j\in[n],i\le j\,\}.
\end{equation*}
For the diagram $\mc{T}_n$ we can note that 
$$
\nu_{\min}(\mc{T}_n,d)=\sum_{i=1}^{n-d+1}i=\binom{n-d+2}{2}=\frac{(n-d+1)(n-d+2)}{2}.
$$
Moreover, $(\mc{T}_n,d)$ is $j$-Singleton for any $j\in\{0,\ldots,d-1\}$.

\begin{definition}
Let $n\in\mathbb{N}$, $d\in[n]$ and $j\in\{\,0,\ldots,d-1\,\}$. We define the following sets
\begin{align*}
\mc{S}_{n,d,j}&=\{\,(i,l)\in [n]^2\,:\,i\in\{\,d-j,\ldots,n\,\},l\in[n-j]\,\},\\
\mc{T}_{n,d,j}&=\mc{S}_{n,d,j}\cap\mc{T}_{n},\\
\mc{L}_{n,d,j}&=[n]^2\setminus{\mc{S}_{n,d,j}}.
\end{align*}  
\end{definition}
Observe that, graphically speaking, $\mc{S}_{n,d,j}$ is the set of dots of $[n]^2$ that remain after deleting the last $j$ columns and the first $d-1-j$ rows. Therefore, if $\mc{F}$ is a Ferrers diagram of order $n$, then we have
\begin{equation*}
\nu_j(\mc{F},d)=\sum_{i=1}^{n-j}\max{\{\,0,c_i-d+1+j\,\}}=|\mc{F}\cap\mc{S}_{n,d,j}|=|\mc{F}|-|\mc{F}\cap\mc{L}_{n,d,j}|,
\end{equation*}
and hence
\begin{align*}
\num{F}{d}&=\min_{j\in\{\,0,\ldots,d-1\,\}}{\nu_j(\mc{F},d)}=\min_{j\in\{\,0,\ldots,d-1\,\}}|\mc{F}\cap\mc{S}_{n,d,j}|.
\end{align*}

The following bound on the parameters of a Ferrers diagram rank-metric code was proved by Etzion and Silberstein in \cite{bib:etzion1}.
\begin{theorem}[{see \cite[Theorem 1]{bib:etzion1}}]
\label{th:dimension}
Let $\mc{F}$ be a Ferrers diagram of order $n$ and let $\mathcal{C}$ be an $[\mathcal{F},{k},d]_q$-code. Then
\begin{equation*}
{k}\leq{\nu_{\min}(\mathcal{F},d)}.
\end{equation*}
\end{theorem}

As with rank-metric codes and Hamming metric codes, we define an optimal Ferrers diagram rank-metric code as one that meets the bound with equality.
\begin{definition}
Let $\mc{F}$ be a Ferrers diagram of order $n$. An $[\mathcal{F},{k},d]_q$-code $\mathcal{C}$ is called \emph{maximum
Ferrers diagram} (MFD) code if ${k}=\nu_{\min}(\mathcal{F},d)$.
\end{definition}

Etzion and Silberstein stated in \cite{bib:etzion1} the following conjecture.
\begin{conjecture}[{see \cite[Conjecture 1]{bib:etzion1}}]
\label{con:etzsilb}
For every Ferrers diagram $\mathcal{F}$ of order $n$, every integer $1\leq{d}\leq n$ and every finite field $\mathbb{F}_q$ there exists an  $[\mathcal{F},{\nu_{\min}(\mathcal{F},d)},d]_q$ MFD code.
\end{conjecture}

Although Conjecture \ref{con:etzsilb} remains open in its full generality, it has been verified in several particular cases; see e.g. \cite{bib:liu} for an overview. Two main approaches have been developed: constructing MFD codes as subcodes of MRD codes, and employing the concept of MDS-constructibility. A comprehensive summary of the MRD-subcode constructions can be found in \cite{bib:liu2}. Notably, the MRD-based approach completely settles the case where $d=2$. The main idea of this MRD-subcode construction relies on a reduction argument, that is, constructing MFD codes on new diagrams starting from existing MFD codes on other diagrams, via inclusion or puncturing. This idea is based on a simple observation originally stated in \cite{bib:antrobus}, which has been deeply studied in \cite{bib:beeloo}.

The second approach relies on the theory of MDS codes in the Hamming metric. When first introduced, this method enabled the proof of Conjecture \ref{con:etzsilb} over sufficiently large finite fields and for Ferrers diagrams satisfying the MDS-constructibility property (see Definition \ref{def:MDScon}). More recently, in  \cite{bib:neri}  the conjecture was proved to hold for all MDS-constructible Ferrers diagrams over arbitrary finite fields.
To eliminate the dependency on field size, the authors combined their method with MRD-subcode constructions.

\subsection{MDS-constructible Ferrers diagrams}

In this section, we give a brief overview of MDS-constructible pairs. Roughly speaking, these are pairs $(\mathcal F,d)$ such that $\mathcal F$ is a Ferrers diagram over which one can construct $[\mathcal{F},{\nu_{\min}(\mathcal{F},d)},d]_q$ MFD codes using MDS codes in the Hamming metric on the diagonals of $\mathcal F$.
Here, we give the formal definition of MDS-constructible pairs.

\begin{definition}\label{def:MDScon}
    Let $\mc{F}$ be a Ferrers diagram of order $n$ and $d\in[n]$ an integer. The pair $(\mathcal{F},d)$ is said to be \textit{MDS-constructible} if
    \begin{equation*}
        \nu_{\min}(\mathcal{F},d)=\sum_{i=1}^n\max\{\,|D_i\cap\mc{F}|-d+1,0\,\}.
    \end{equation*}
\end{definition}

\begin{definition}
A Ferrers diagram $\mc{F}=[c_1,\ldots, c_n]$ is called 
\textit{monotone} if, whenever
$i \in [n- 1]$ is such that $0 < c_i < n$, one has $c_{i+1} > c_i$. We say that $\mc{F}$ is
\textit{strictly monotone} if, whenever $i\in[n-1]$ and $c_i>0$, one has $c_{i+1}>c_i$. 
\end{definition}
\begin{definition}
A Ferrers diagram $\mc{F}=[c_1,\ldots, c_n]$ is called \textit{convex} 
if $c_{i+1}- c_i \le 1$ for every $i\in[n-1]$. We say that $\mc{F}$ is
\textit{initially convex} if it is convex and $c_1\leq 1$.
\end{definition}
\begin{remark}
    The reader can easily verify that the empty diagram $\mc{F}= \emptyset$ and the full diagram $[n]^2$ are both monotone and convex. Moreover, convex Ferrers diagrams  and monotone Ferrers diagrams are adjoint of each others. 
\end{remark}
\begin{remark}
The empty diagram is clearly both strictly monotone and initially convex; the same can be said for triangular diagrams, but the full diagram $[n]^2$ is not strictly monotone nor initially convex. Furthermore, the adjoint of a strictly monotone diagram is initially convex and vice versa.
\end{remark}

\begin{theorem}[{see \cite[Theorem 4.10]{bib:neri}}]
\label{th:incon}
   % Let $d,n$ be positive integers with $1\leq d\leq n$. Then 
    The following statements hold.
    \begin{enumerate}
        \item Conjecture \ref{con:etzsilb} holds true for strictly monotone Ferrers diagrams %of order $n$ 
        over any finite field. 
        \item Conjecture \ref{con:etzsilb} holds true for initially convex Ferrers diagrams %of order $n$ 
        over any finite field.
    \end{enumerate}
\end{theorem}

\begin{corollary}
Conjecture \ref{con:etzsilb} holds true for upper triangular matrices over any finite field.
\end{corollary}

The following is a well-known result whose proof can be found in \cite{bib:antrobus} and gives criteria for the existence of MFD codes on subdiagrams or superdiagrams.
\begin{lemma}[{see \cite[Remark II.12 and Remark II.14]{bib:antrobus}}]
\label{lem:incldots}
Let $\mathcal{F}$ be a Ferrers diagram of order $n$ and let $2\leq d \leq n$ be an integer.
\begin{enumerate}
    \item Let $\bar{\mathcal{F}}\subseteq\mathcal{F}$ be a Ferrers diagram of order $n$ such that $\nu_\text{min}(\mathcal{\mathcal{F}},d)=\nu_\text{min}(\bar{\mathcal{F}},d)$. If $\mathcal{C}$ is an $[\bar{\mathcal{F}},{\nu_{\min}(\bar{\mathcal{F}},d)},d]_q$ MFD code, then it is also an $[\mathcal{F},{\nu_{\min}(\mathcal{F},d)},d]_q$ MFD code.
     \item Let $\bar{\mathcal{F}}\supseteq\mathcal{F}$ be a Ferrers diagram of order $n$ such that
     $\nu_{\min}(\mathcal{F},d)=\nu_{\min{}}(\bar{\mathcal{F}},d)-|\bar{\mathcal{F}}\setminus{\mathcal{F}}|$. If $\mathcal{C}$ is an $[\bar{\mathcal{F}},{\nu_{\min}(\bar{\mathcal{F}},d)},d]_q$ MFD code, then $\mathcal{C}\cap\mathbb{F}_q^{\mathcal{F}}$ is an $[\mathcal{F},{\nu_{\min}(\mathcal{F},d)},d]_q$ MFD code.
\end{enumerate}
\end{lemma}

In \cite{bib:neri}, Neri and Stanojkovski proved that for  MDS-constructible pairs, the existence of an MFD code over a finite field $\mathbb{F}_q$ can be reduced to the existence of an MFD code on triangular Ferrers diagrams over the same field. Consequently, from Theorem~\ref{th:incon}, they were able to deduce the following result.

\begin{theorem}[{see \cite[Theorem 4.22]{bib:neri}}]
\label{th:optany}
Let $\mc{F}$ be a Ferrers diagram of order $n$ and let $2 \leq d \leq n$ be an integer. If $(\mc{F},d)$ is MDS-constructible, then there exists an $[\mathcal{F},{\nu_{\min}(\mathcal{F},d)},d]_q$ MFD code over any finite field $\mathbb{F}_q$.
\end{theorem}

\subsection{Sum-rank-metric codes}
\label{chsec:unosumrank}

In this section, we introduce the basic notions on sum-rank-metric codes. These codes were originally introduced in the context of multishot network coding \cite{bib:nobrega}, and in recent years they have gained popularity also for their intrinsic mathematical features; see e.g. \cite{bib:martinez,bib:byrneSRM2021,bib:nerilin}.

Let $t,m$ be positive integers and consider the $\Fq$-vector space
 \begin{equation*}
     % \mathbb{F}_q^\textbf{m}:=\bigoplus_{i=1}^t\mathbb{F}_q^{m},
(\mathbb{F}_q^{m \times m})^t:=\bigoplus_{i=1}^t\mathbb{F}_q^{m\times m}.
 \end{equation*}
 \begin{definition}
     Let $X=(X_1,\ldots,X_t)\in    (\mathbb{F}_q^{m \times m})^t$. 
     The \textit{sum-rank weight} of $X$ is the quantity
     \begin{equation*}
         \text{w}_{\srk}(X):=\sum_{i=1}^t\rk(X_i).
     \end{equation*}
     Moreover, we can endow the space $(\mathbb{F}_q^{m \times m})^t$ with the \textit{sum-rank distance}, defined as
     \begin{equation*}
     \begin{array}{rccl}
         \text{d}_{\srk} & :\,  (\mathbb{F}_q^{m \times m})^t\times   (\mathbb{F}_q^{m \times m})^t &\longrightarrow &\mathbb{Z}_{\ge 0}, \\
        & (X,Y)& \longmapsto &\text{w}_{\srk}(X-Y).
         \end{array}
     \end{equation*}
 \end{definition}
Observe that this definition is a generalization of both the Hamming distance and the rank distance. On the one hand, taking $t=1$ one has $(\mathbb{F}_q^{m \times m})^t=\mathbb{F}_q^{m\times m}$ and the sum-rank distance coincides with the rank distance. On the other hand, taking $m=1$, one obtains $(\mathbb{F}_q^{m \times m})^t=\mathbb{F}_q^{t}$, and in this case the sum-rank distance is simply the Hamming distance. In the literature, the sum-rank distance has been studied also over more general spaces, where each summand is a matrix space with possibly different size.
 
 \begin{definition}
     An $[(m\times m)^t,k,d]_{q}$ \textit{sum-rank-metric code} $\mathcal{C}$ is a linear subspace of $(\mathbb{F}_q^{m \times m})^t$ of dimension $k$ endowed with the sum-rank distance. The parameter $d$ is the \textit{minimum sum-rank distance} of $\mathcal{C}$, which  is defined as
     \begin{equation*}
         d=\text{d}_{\srk}(\mathcal{C}):=\min\{\,\text{w}_{\srk}(X)\,:\,X\in\mathcal{C},X\ne{0}\,\}.
     \end{equation*}
 \end{definition}

 \begin{remark}[see \cite{bib:gorlasumrank}]
Observe that every sum-rank-metric code can be seen as a code with the rank-metric. In particular, the linear injection $\iota\,:\,(\mathbb{F}_q^{m \times m})^t\rightarrow\mathbb{F}_q^{tm\times tm}$ given by
   \begin{equation*}
       (C_1,\ldots,C_t)\mapsto 
    \begin{pmatrix}
    C_1 &        &        \\
        & \ddots &        \\
        &        & C_t
    \end{pmatrix}
   \end{equation*}
   is distance preserving, i.e. $\text{w}_{\srk}(C)=\rk\iota(C)$ for all $C\in(\mathbb{F}_q^{m \times m})^t$. In this way a sum-rank-metric code $\mathcal{C}\subseteq(\mathbb{F}_q^{m \times m})^t$ can be identified with its image $\iota(\mathcal{C})$, which is a rank-metric code in $\mathbb{F}_q^{tm\times tm}$.
 \end{remark}
The following theorem is a special case of the Singleton-like bound for sum-rank-metric codes. A more general statement can be given without supposing all square matrices of the same size (see for instance \cite{bib:byrneSRM2021,bib:gorlaSRM2022}).
 \begin{theorem}[{see \cite[Corollary 2]{bib:Martinez2019}}]
     Let $\mathcal{C}$ be an $[(m\times m)^t,k,d]_q$ code. Then,
         $k\leq m (t m-d+1)$.
 \end{theorem}
 \begin{definition}
Let $\mathcal{C}$ be an $[(m\times m)^t,k,d]_q$ code. The code $\mc{C}$ is a \textit{maximum sum-rank distance (MSRD) code} if $k=m (t m-d+1)$.
 \end{definition}
 A construction of MSRD codes for every admissible $d, m$ and $t$ is given in \cite{bib:martinez}, under the assumption that $t\le q-1$. These codes are now known as \emph{linearized Reed-Solomon codes}.
 
 \begin{theorem}
     Let $q$ be a prime power and let $d,t,m$ be  positive integers such that $1\le d \le tm$. If $t\leq q-1$, then there exists an $[(m\times m)^t,k,d]_q$ MSRD code, that is, a sum-rank-metric code such that
        $ k= m (t m-d+1).$
     \label{th:MSRD}
 \end{theorem}

\section{MSRD-constructibility and optimality}\label{sec:MSRD-contruction}
In this section, we focus on block diagrams. 
For those diagrams, we introduce the concept of MSRD-constructible pairs, which are a generalization of the MDS-constructible ones. Moreover, we show how we can use MSRD codes to construct MFD codes for MSRD-constructible pairs. Our approach generalizes the construction of MFD codes for MDS-constructible pairs, where one uses MDS codes \cite{bib:etzion}.

\begin{definition}
Let $m \in \mathbb N$. We define the $(i,j)$-th \textit{block} of side $m$ as the set
$Q_{i,j}^{(m)}:=([im]\setminus[(i-1)m])\times ([jm]\setminus[(j-1)m]).$
\end{definition}

From this, the definition of a block Ferrers diagram follows easily by juxtaposing blocks.

\begin{definition}
    Let $m,s \in \mathbb N$. A Ferrers diagram $\mathcal{F}$ of order $sm$ is an $m$\textit{-block Ferrers diagram} if for every $i, j \in [s]$, $Q_{i,j}^{(m)}\cap \mathcal{F}\in \{\emptyset, Q_{i,j}^{(m)}  \}$. 
\end{definition}

We stress here that we will always consider square blocks of the same size.

\begin{definition}
Let $\mathcal{F}$ be an $m$-block Ferrers diagram of order $sm$. We define the \textit{block set} of $\mathcal{F}$ as
\begin{equation*}
    \mathbf{B}(\mathcal{F})=\{Q_{i,j}^{(m)}\,:\,i,j\in[s] \text{ and } Q_{i,j}^{(m)}\cap \mathcal{F}=Q_{i,j}^{(m)}\}.
\end{equation*}
\end{definition}
In every diagram we consider all blocks of the same size $m$, thus to lighten up the notation we write $Q_{i,j}$ instead of $Q_{i,j}^{(m)}$.
Let us consider block rows, block columns and block diagonals of $[sm]^2$:
\begin{enumerate}
    \item Block rows: $R_i^B:=\{Q_{i,j}\in \mathbf{B}([sm]^2)\,:\,j\in[s]\},\quad i\in[s].$
    \item Block columns: $C_j^B:=\{Q_{i,j}\in\mathbf{B}([sm]^2)\,:\,i\in[s]\},\quad j\in[s].$
    \item Block diagonals: $D^B_i:=\{Q_{l,s-i+l}\in\mathbf{B}([sm]^2)\,:\,1\leq{l}\leq{i}\},\quad i\in[s].$
\end{enumerate}
Furthermore, if $\mc{F}$ is an $m$-block Ferrers diagram of order $sm$, its block rows, block columns and block diagonals are given by $R_i^B\cap\mathbf{B}(\mc{F})$, $C_i^B\cap\mathbf{B}(\mc{F})$ and $D_i^B\cap\mathbf{B}(\mc{F})$, $i\in[s]$, respectively. Moreover, as already said, a diagram can be identified by considering only the number of elements per column and the top-right aligned property (i.e. writing $\mathcal{F}=[c_1,\ldots,c_n]$, if $\mathcal{F}$ has order $n$). In the same way, if $\mathcal{F}$ is an $m$-block Ferrers diagram with order ${sm}$, we will write $\mathcal{F}=[[h_1,\ldots,h_s]]_m$ where $h_j=|C^B_j\cap\mathbf{B}(\mc{F})|$ for every $j\in[s]$.

\begin{example}
The diagram $\mathcal{F}=[[1,1,2,2,5]]_2$ is a $2$-block Ferrers diagram of order $10$ with block set
\begin{equation*}
\mathbf{B}(\mathcal{F})=\{Q_{1,1},Q_{1,2},Q_{1,3},Q_{1,4},Q_{1,5},Q_{2,3},Q_{2,4},Q_{2,5},Q_{3,5},Q_{4,5},Q_{5,5}\}.
\end{equation*}
\begin{center} 
\begin{tikzpicture}[scale=0.4]
\draw[help lines, white, fill=teal!5] (0,1)-- (0,-9)-- (10,-9)--(10,1)--(0,1);
\draw[help lines, thick, blue] (0.1,0.9)-- (0.1,-0.9)-- (1.9,-0.9)--(1.9,0.9)--(0.1,0.9);
\draw[help lines, thick, orange] (2.1,0.9)-- (2.1,-0.9)-- (3.9,-0.9)--(3.9,0.9)--(2.1,0.9);
\draw[help lines, thick, red] (4.1,0.9)-- (4.1,-2.9)-- (5.9,-2.9)--(5.9,0.9)--(4.1,0.9);
\draw[help lines, thick, green] (6.1,0.9)-- (6.1,-2.9)-- (7.9,-2.9)--(7.9,0.9)--(6.1,0.9);
\draw[help lines,thick, magenta] (8.1,0.9)-- (8.1,-8.9)-- (9.9,-8.9)--(9.9,0.9)--(8.1,0.9);

\draw[help lines, dashed] (2,1) -- (2,-9);
\draw[help lines, dashed] (0,-1) -- (10,-1);
\draw[help lines, dashed] (4,1) -- (4,-9);
\draw[help lines, dashed] (0,-3) -- (10,-3);
\draw[help lines, dashed] (6,1) -- (6,-9);
\draw[help lines, dashed] (0,-5) -- (10,-5);
\draw[help lines, dashed] (8,1) -- (8,-9);
\draw[help lines, dashed] (0,-7) -- (10,-7);

 %prima riga
\draw[fill=black] (0.5,0.5) circle (0.1cm)node[above right,yshift=5pt]{C$^B_1$};
\draw[fill=black] (1.5,0.5) circle (0.1cm);
\draw[fill=black] (2.5,0.5) circle (0.1cm)node[above right,yshift=5pt]{C$^B_2$};
\draw[fill=black] (3.5,0.5) circle (0.1cm);
\draw[fill=black] (4.5,0.5) circle (0.1cm)node[above right,yshift=5pt]{C$^B_3$};
\draw[fill=black] (5.5,0.5) circle (0.1cm);
\draw[fill=black] (6.5,0.5) circle (0.1cm)node[above right,yshift=5pt]{C$^B_4$};
\draw[fill=black] (7.5,0.5) circle (0.1cm);
\draw[fill=black] (8.5,0.5) circle (0.1cm)node[above right,yshift=5pt]{C$^B_5$};
\draw[fill=black] (9.5,0.5) circle (0.1cm);

%seconda riga
\draw[fill=black] (0.5,-0.5) circle (0.1cm);
\draw[fill=black] (1.5,-0.5) circle (0.1cm);
\draw[fill=black] (2.5,-0.5) circle (0.1cm);
\draw[fill=black] (3.5,-0.5) circle (0.1cm);
\draw[fill=black] (4.5,-0.5) circle (0.1cm);
\draw[fill=black] (5.5,-0.5) circle (0.1cm);
\draw[fill=black] (6.5,-0.5) circle (0.1cm);
\draw[fill=black] (7.5,-0.5) circle (0.1cm);
\draw[fill=black] (8.5,-0.5) circle (0.1cm);
\draw[fill=black] (9.5,-0.5) circle (0.1cm);

%terza riga
\draw[fill=black] (4.5,-1.5) circle (0.1cm);
\draw[fill=black] (5.5,-1.5) circle (0.1cm);
\draw[fill=black] (6.5,-1.5) circle (0.1cm);
\draw[fill=black] (7.5,-1.5) circle (0.1cm);
\draw[fill=black] (8.5,-1.5) circle (0.1cm);
\draw[fill=black] (9.5,-1.5) circle (0.1cm);

%quarta riga
\draw[fill=black] (4.5,-2.5) circle (0.1cm);
\draw[fill=black] (5.5,-2.5) circle (0.1cm);
\draw[fill=black] (6.5,-2.5) circle (0.1cm);
\draw[fill=black] (7.5,-2.5) circle (0.1cm);
\draw[fill=black] (8.5,-2.5) circle (0.1cm);
\draw[fill=black] (9.5,-2.5) circle (0.1cm);

\draw[fill=black] (8.5,-3.5) circle (0.1cm);
\draw[fill=black] (9.5,-3.5) circle (0.1cm);

\draw[fill=black] (8.5,-4.5) circle (0.1cm);
\draw[fill=black] (9.5,-4.5) circle (0.1cm);

\draw[fill=black] (8.5,-5.5) circle (0.1cm);
\draw[fill=black] (9.5,-5.5) circle (0.1cm);

\draw[fill=black] (8.5,-6.5) circle (0.1cm);
\draw[fill=black] (9.5,-6.5) circle (0.1cm);

\draw[fill=black] (8.5,-7.5) circle (0.1cm);
\draw[fill=black] (9.5,-7.5) circle (0.1cm);

\draw[fill=black] (8.5,-8.5) circle (0.1cm);
\draw[fill=black] (9.5,-8.5) circle (0.1cm);
\end{tikzpicture} 
\begin{tikzpicture}[scale=0.4]
\draw[help lines, white, fill=teal!5] (0,1)-- (0,-9)-- (10,-9)--(10,1)--(0,1);
\draw[help lines, thick,blue] (0.1,0.9)-- (9.9,0.9)-- (9.9,-0.9)--(0.1,-0.9)--(0.1,0.9);
\draw[help lines, thick,orange] (4.1,-1.1)-- (9.9,-1.1)-- (9.9,-2.9)--(4.1,-2.9)--(4.1,-1.1);
\draw[help lines, thick, red] (8.1,-3.1)-- (9.9,-3.1)-- (9.9,-4.9)--(8.1,-4.9)--(8.1,-3.1);
\draw[help lines, thick, green] (8.1,-5.1)-- (9.9,-5.1)-- (9.9,-6.9)--(8.1,-6.9)--(8.1,-5.1);
\draw[help lines, thick, magenta] (8.1,-7.1)-- (9.9,-7.1)-- (9.9,-8.9)--(8.1,-8.9)--(8.1,-7.1);

\draw[help lines, dashed] (2,1) -- (2,-9);
\draw[help lines, dashed] (0,-1) -- (10,-1);
\draw[help lines, dashed] (4,1) -- (4,-9);
\draw[help lines, dashed] (0,-3) -- (10,-3);
\draw[help lines, dashed] (6,1) -- (6,-9);
\draw[help lines, dashed] (0,-5) -- (10,-5);
\draw[help lines, dashed] (8,1) -- (8,-9);
\draw[help lines, dashed] (0,-7) -- (10,-7);

 %prima riga
\draw[fill=black] (0.5,0.5) circle (0.1cm);
\draw[fill=black] (1.5,0.5) circle (0.1cm);
\draw[fill=black] (2.5,0.5) circle (0.1cm);
\draw[fill=black] (3.5,0.5) circle (0.1cm);
\draw[fill=black] (4.5,0.5) circle (0.1cm);
\draw[fill=black] (5.5,0.5) circle (0.1cm);
\draw[fill=black] (6.5,0.5) circle (0.1cm);
\draw[fill=black] (7.5,0.5) circle (0.1cm);
\draw[fill=black] (8.5,0.5) circle (0.1cm);
\draw[fill=black] (9.5,0.5) circle (0.1cm);

%seconda riga
\draw[fill=black] (0.5,-0.5) circle (0.1cm);
\draw[fill=black] (1.5,-0.5) circle (0.1cm);
\draw[fill=black] (2.5,-0.5) circle (0.1cm);
\draw[fill=black] (3.5,-0.5) circle (0.1cm);
\draw[fill=black] (4.5,-0.5) circle (0.1cm);
\draw[fill=black] (5.5,-0.5) circle (0.1cm);
\draw[fill=black] (6.5,-0.5) circle (0.1cm);
\draw[fill=black] (7.5,-0.5) circle (0.1cm);
\draw[fill=black] (8.5,-0.5) circle (0.1cm);
\draw[fill=black] (9.5,-0.5) circle (0.1cm) node[above right,xshift=5pt]{R$^B_1$};

%terza riga
\draw[fill=black] (4.5,-1.5) circle (0.1cm);
\draw[fill=black] (5.5,-1.5) circle (0.1cm);
\draw[fill=black] (6.5,-1.5) circle (0.1cm);
\draw[fill=black] (7.5,-1.5) circle (0.1cm);
\draw[fill=black] (8.5,-1.5) circle (0.1cm);
\draw[fill=black] (9.5,-1.5) circle (0.1cm);

%quarta riga
\draw[fill=black] (4.5,-2.5) circle (0.1cm);
\draw[fill=black] (5.5,-2.5) circle (0.1cm);
\draw[fill=black] (6.5,-2.5) circle (0.1cm);
\draw[fill=black] (7.5,-2.5) circle (0.1cm);
\draw[fill=black] (8.5,-2.5) circle (0.1cm);
\draw[fill=black] (9.5,-2.5) circle (0.1cm) node[above right,xshift=5pt]{R$^B_2$};

\draw[fill=black] (8.5,-3.5) circle (0.1cm);
\draw[fill=black] (9.5,-3.5) circle (0.1cm);

\draw[fill=black] (8.5,-4.5) circle (0.1cm);
\draw[fill=black] (9.5,-4.5) circle (0.1cm)node [above right,xshift=5pt]{R$^B_3$};

\draw[fill=black] (8.5,-5.5) circle (0.1cm);
\draw[fill=black] (9.5,-5.5) circle (0.1cm);

\draw[fill=black] (8.5,-6.5) circle (0.1cm);
\draw[fill=black] (9.5,-6.5) circle (0.1cm)node[above right,xshift=5pt]{R$^B_4$};

\draw[fill=black] (8.5,-7.5) circle (0.1cm);
\draw[fill=black] (9.5,-7.5) circle (0.1cm);

\draw[fill=black] (8.5,-8.5) circle (0.1cm);
\draw[fill=black] (9.5,-8.5) circle (0.1cm)node[above right,xshift=5pt]{R$^B_5$};
\end{tikzpicture} 
\begin{tikzpicture}[scale=0.4]
\draw[help lines, white, fill=teal!5] (0,1)-- (0,-9)-- (10,-9)--(10,1)--(0,1);
\draw[help lines, thick, blue] (0.1,0.9)-- (0.1,-0.9)-- (1.9,-0.9)--(1.9,0.9)--(0.1,0.9);
\draw[help lines, thick, blue] (8.1,-7.1)-- (8.1,-8.9)-- (9.9,-8.9)--(9.9,-7.1)--(8.1,-7.1);
\draw[help lines, thick, orange] (2.1,0.9)-- (2.1,-0.9)-- (3.9,-0.9)--(3.9,0.9)--(2.1,0.9);
\draw[help lines, thick, orange] (4.1,-1.1)-- (4.1,-2.9)-- (5.9,-2.9)--(5.9,-1.1)--(4.1,-1.1);
\draw[help lines, thick, orange] (8.1,-5.1)-- (8.1,-6.9)-- (9.9,-6.9)--(9.9,-5.1)--(8.1,-5.1);
\draw[help lines, thick, red] (4.1,0.9)-- (4.1,-0.9)-- (5.9,-0.9)--(5.9,0.9)--(4.1,0.9);
\draw[help lines, thick, red] (6.1,-1.1)-- (6.1,-2.9)-- (7.9,-2.9)--(7.9,-1.1)--(6.1,-1.1);
\draw[help lines, thick, red] (8.1,-3.1)-- (8.1,-4.9)-- (9.9,-4.9)--(9.9,-3.1)--(8.1,-3.1);
\draw[help lines, thick, green] (6.1,0.9)-- (6.1,-0.9)-- (7.9,-0.9)--(7.9,0.9)--(6.1,0.9);
\draw[help lines, thick, green] (8.1,-1.1)-- (8.1,-2.9)-- (9.9,-2.9)--(9.9,-1.1)--(8.1,-1.1);
\draw[help lines, thick, lime] (8.1,0.9)-- (8.1,-0.9)-- (9.9,-0.9)--(9.9,0.9)--(8.1,0.9);

\draw[help lines, dashed] (2,1) -- (2,-9);
\draw[help lines, dashed] (0,-1) -- (10,-1);
\draw[help lines, dashed] (4,1) -- (4,-9);
\draw[help lines, dashed] (0,-3) -- (10,-3);
\draw[help lines, dashed] (6,1) -- (6,-9);
\draw[help lines, dashed] (0,-5) -- (10,-5);
\draw[help lines, dashed] (8,1) -- (8,-9);
\draw[help lines, dashed] (0,-7) -- (10,-7);

 %prima riga
\draw[fill=black] (0.5,0.5) circle (0.1cm);
\draw[fill=black] (1.5,0.5) circle (0.1cm);
\draw[fill=black] (2.5,0.5) circle (0.1cm);
\draw[fill=black] (3.5,0.5) circle (0.1cm);
\draw[fill=black] (4.5,0.5) circle (0.1cm);
\draw[fill=black] (5.5,0.5) circle (0.1cm);
\draw[fill=black] (6.5,0.5) circle (0.1cm);
\draw[fill=black] (7.5,0.5) circle (0.1cm);
\draw[fill=black] (8.5,0.5) circle (0.1cm);
\draw[fill=black] (9.5,0.5) circle (0.1cm);

%seconda riga
\draw[fill=black] (0.5,-0.5) circle (0.1cm);
\draw[fill=black] (1.5,-0.5) circle (0.1cm);
\draw[fill=black] (2.5,-0.5) circle (0.1cm);
\draw[fill=black] (3.5,-0.5) circle (0.1cm);
\draw[fill=black] (4.5,-0.5) circle (0.1cm);
\draw[fill=black] (5.5,-0.5) circle (0.1cm);
\draw[fill=black] (6.5,-0.5) circle (0.1cm);
\draw[fill=black] (7.5,-0.5) circle (0.1cm);
\draw[fill=black] (8.5,-0.5) circle (0.1cm);
\draw[fill=black] (9.5,-0.5) circle (0.1cm)node[above,xshift=15pt]{$D^B_1$};

%terza riga
\draw[fill=black] (4.5,-1.5) circle (0.1cm);
\draw[fill=black] (5.5,-1.5) circle (0.1cm);
\draw[fill=black] (6.5,-1.5) circle (0.1cm);
\draw[fill=black] (7.5,-1.5) circle (0.1cm);
\draw[fill=black] (8.5,-1.5) circle (0.1cm);
\draw[fill=black] (9.5,-1.5) circle (0.1cm);

%quarta riga
\draw[fill=black] (4.5,-2.5) circle (0.1cm);
\draw[fill=black] (5.5,-2.5) circle (0.1cm);
\draw[fill=black] (6.5,-2.5) circle (0.1cm);
\draw[fill=black] (7.5,-2.5) circle (0.1cm);
\draw[fill=black] (8.5,-2.5) circle (0.1cm);
\draw[fill=black] (9.5,-2.5) circle (0.1cm)node[above,xshift=15pt]{$D^B_2$};

\draw[fill=black] (8.5,-3.5) circle (0.1cm);
\draw[fill=black] (9.5,-3.5) circle (0.1cm);

\draw[fill=black] (8.5,-4.5) circle (0.1cm);
\draw[fill=black] (9.5,-4.5) circle (0.1cm)node[above,xshift=15pt]{$D^B_3$};

\draw[fill=black] (8.5,-5.5) circle (0.1cm);
\draw[fill=black] (9.5,-5.5) circle (0.1cm);

\draw[fill=black] (8.5,-6.5) circle (0.1cm);
\draw[fill=black] (9.5,-6.5) circle (0.1cm)node[above,xshift=15pt]{$D^B_4$};

\draw[fill=black] (8.5,-7.5) circle (0.1cm);
\draw[fill=black] (9.5,-7.5) circle (0.1cm);

\draw[fill=black] (8.5,-8.5) circle (0.1cm);
\draw[fill=black] (9.5,-8.5) circle (0.1cm)node[above,xshift=15pt]{$D^B_5$};
\end{tikzpicture} 
\end{center}
Here the diagram $\mathcal{F}$ appears first with the columns highlighted, then with the
rows, and finally with the diagonals.
\end{example}

Informally speaking, the aim is to construct MFD codes on block Ferrers diagrams by inserting MSRD codes on the diagonals of the diagram. Here we formalize this idea.
\begin{construction}\label{con:MSRD}
Let $\mathcal{F}$ be an $m$-block Ferrers diagram of order $sm$, and let $d\in[sm]$. We define $D^B_{\max{}}:=\underset{i\in[s]}{\max}|D^B_i \cap\mathbf{B}(\mathcal{F})|$ and let $q$ be a prime power with
$q\ge{D^B_{\max}+1}$.
Let $A_j\subseteq (\mathbb{F}_q^{m\times m})^j$ be an $[(m\times m)^j,m(mj-d+1),d]_q$ MSRD code for each $j\in\{\lceil\frac{d}{m}\rceil,\dots,D^B_{\max}\}$ and $A_j:=\{0\}$ for all $j\in[\lceil\frac{d}{m}\rceil-1]$. Now we can define the rank-metric code $\mc{C}$
\begin{equation*}
    \mc{C}:=\big{\{}C\in\mathbb{F}_q^{\mc{F}}\,:\,\text{the i-th block diagonal of } C \text{ is a codeword of }{A_{|D^B_i \cap\mathbf{B}(\mathcal{F})|}},\forall i \in[s]\big{\}}.
\end{equation*}
\end{construction}
Using similar arguments as in \cite[ Lemma 4]{bib:etzion}, we have the following.
\begin{lemma}
\label{le:dimMSRD}
    Construction \ref{con:MSRD} provides for all $q\ge{D^B_{\max}+1}$ a linear $[\mathcal{F},k,d]_q$ Ferrers diagram rank-metric code $\mc{C}$ of dimension
    \begin{equation*}
        k=m\Bigg{(}\sum_{i=1}^s\max{}\{0,m|D^B_i \cap\mathbf{B}(\mathcal{F})|-d+1\}\Bigg{)}.
    \end{equation*}
\end{lemma}
\begin{proof}
    Since $D^B_{\max}\leq{q-1}$, we have $|D^B_i \cap\mathbf{B}(\mathcal{F})|\leq{q-1}$ for all $i\in[s]$, then Theorem \ref{th:MSRD} can be applied. Hence, for all $i\in[s]$ such that $|D^B_i \cap\mathbf{B}(\mathcal{F})|\ge \lceil\frac{d}{m}\rceil$, there exists an MSRD code over $\mathbb{F}_q$ of length $|D^B_i \cap\mathbf{B}(\mathcal{F})|$ and dimension
    \begin{align*}
        k_i=&m\Bigg{(}\sum_{j=1}^{|D^B_i \cap\mathbf{B}(\mathcal{F})|}m-d+1\Bigg{)}\\
        =&m(m|D^B_i \cap\mathbf{B}(\mathcal{F})|-d+1).
    \end{align*}
    The code $\mc{C}$ is linear since all MSRD codes are linear and the dimension of the resulting code is the sum of the MSRD codes' dimensions, then
    \begin{equation*}
        k=m\Bigg{(}\sum_{i=1}^s\max\{0,m|D^B_i \cap\mathbf{B}(\mathcal{F})|-d+1\}\Bigg{)}.
    \end{equation*}
    In order to compute the minimum distance, it is sufficient to compute the minimum rank weight of any non-zero codeword.

    Observe that
    \[\rk
    \begin{pmatrix}
        A&B\\
        0&D
    \end{pmatrix}\ge{\rk(A)+\rk(D)},
    \]
    for any three matrices $A,B$, and $D$ of suitable sizes. To apply this reasoning, cut any matrix recursively into smaller matrices, so the rank of each element of $\mc{C}$ is bounded from below by the sum-rank weight of the bottommost non-zero block diagonal. Since every non-zero block diagonal is a non-zero codeword of an MSRD code of minimum sum-rank distance $d$, then the rank of any non-zero codeword is at least $d$.
\end{proof}

\begin{definition}
    Let $\mathcal{F}$ be an $m$-block Ferrers diagram of order $sm$ and $d\in[sm]$. We say that a pair $(\mc{F},d)$ is \textit{MSRD-constructible} if
    \begin{equation*}
        \nu_{\min}(\mathcal{F},d)=m\sum_{i=1}^s\max\{m|D^B_i \cap\mathbf{B}(\mathcal{F})|-d+1,0\}.
    \end{equation*}
\end{definition}
From Lemma \ref{le:dimMSRD} we have the following optimality result for MSRD-constructible pairs.
\begin{theorem}
\label{th:optMSRD}
   Let $\mc{F}$ be an $m$-block Ferrers diagram of order $sm$ and $d\in[sm]$ such that $(\mc{F},d)$ is MSRD-constructible. Then, Construction \ref{con:MSRD} gives an $[\mathcal{F},{\nu_{\min}(\mathcal{F},d)},d]_q$ MFD code for every $q\ge D^B_{\max}+1$. 
\end{theorem}

\begin{example}

Let $\mc{F}=[[1,1,3,4]]_3$ be a $3$-block Ferrers diagram and $d=5$ the desired minimum distance. It is easy to see that $\num{F}{5}=\nu_1(\mc{F},5)=36$, whereas
\begin{equation*}
    m\sum_{i=1}^s\max\{m|D^B_i \cap\mathbf{B}(\mathcal{F})|-5+1,0\}=3(5+5+2)=36=\nu_{\min}(\mathcal{F},5).
\end{equation*}
Therefore, $(\mc{F},5)$ is MSRD-constructible.
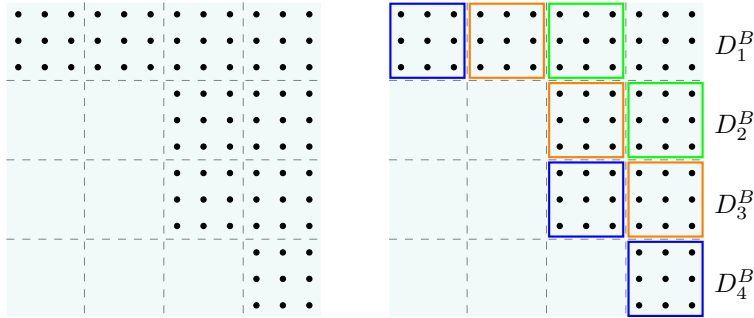
\begin{center}
\begin{tikzpicture}[scale=0.35]
\draw[help lines, very thick, white, fill=teal!5] (0,1) -- (0,-11) -- (12,-11)--(12,1)--(0,1);

\draw[help lines, dashed] (3,0.95) -- (3,-10.95);
\draw[help lines, dashed] (0.05,-2) -- (11.95,-2);
\draw[help lines, dashed] (6,0.95) -- (6,-10.95);
\draw[help lines, dashed] (9,0.95) -- (9,-10.95);
\draw[help lines, dashed] (0.05,-5) -- (11.95,-5);
\draw[help lines, dashed] (0.05,-8) -- (11.95,-8);
  
\draw[fill=black] (0.5,0.5) circle (0.1cm);
\draw[fill=black] (1.5,0.5) circle (0.1cm);
\draw[fill=black] (2.5,0.5) circle (0.1cm);
\draw[fill=black] (3.5,0.5) circle (0.1cm);
\draw[fill=black] (4.5,0.5) circle (0.1cm);
\draw[fill=black] (5.5,0.5) circle (0.1cm);
\draw[fill=black] (6.5,0.5) circle (0.1cm);
\draw[fill=black] (7.5,0.5) circle (0.1cm);
\draw[fill=black] (8.5,0.5) circle (0.1cm);
\draw[fill=black] (9.5,0.5) circle (0.1cm);
\draw[fill=black] (10.5,0.5) circle (0.1cm);
\draw[fill=black] (11.5,0.5) circle (0.1cm);

\draw[fill=black] (0.5,-0.5) circle (0.1cm);
\draw[fill=black] (1.5,-0.5) circle (0.1cm);
\draw[fill=black] (2.5,-0.5) circle (0.1cm);
\draw[fill=black] (3.5,-0.5) circle (0.1cm);
\draw[fill=black] (4.5,-0.5) circle (0.1cm);
\draw[fill=black] (5.5,-0.5) circle (0.1cm);
\draw[fill=black] (6.5,-0.5) circle (0.1cm);
\draw[fill=black] (7.5,-0.5) circle (0.1cm);
\draw[fill=black] (8.5,-0.5) circle (0.1cm);
\draw[fill=black] (9.5,-0.5) circle (0.1cm);
\draw[fill=black] (10.5,-0.5) circle (0.1cm);
\draw[fill=black] (11.5,-0.5) circle (0.1cm);

\draw[fill=black] (0.5,-1.5) circle (0.1cm);
\draw[fill=black] (1.5,-1.5) circle (0.1cm);
\draw[fill=black] (2.5,-1.5) circle (0.1cm);
\draw[fill=black] (3.5,-1.5) circle (0.1cm);
\draw[fill=black] (4.5,-1.5) circle (0.1cm);
\draw[fill=black] (5.5,-1.5) circle (0.1cm);
\draw[fill=black] (6.5,-1.5) circle (0.1cm);
\draw[fill=black] (7.5,-1.5) circle (0.1cm);
\draw[fill=black] (8.5,-1.5) circle (0.1cm);
\draw[fill=black] (9.5,-1.5) circle (0.1cm);
\draw[fill=black] (10.5,-1.5) circle (0.1cm);
\draw[fill=black] (11.5,-1.5) circle (0.1cm);

\draw[fill=black] (6.5,-2.5) circle (0.1cm);
\draw[fill=black] (7.5,-2.5) circle (0.1cm);
\draw[fill=black] (8.5,-2.5) circle (0.1cm);
\draw[fill=black] (9.5,-2.5) circle (0.1cm);
\draw[fill=black] (10.5,-2.5) circle (0.1cm);
\draw[fill=black] (11.5,-2.5) circle (0.1cm);

\draw[fill=black] (6.5,-3.5) circle (0.1cm);
\draw[fill=black] (7.5,-3.5) circle (0.1cm);
\draw[fill=black] (8.5,-3.5) circle (0.1cm);
\draw[fill=black] (9.5,-3.5) circle (0.1cm);
\draw[fill=black] (10.5,-3.5) circle (0.1cm);
\draw[fill=black] (11.5,-3.5) circle (0.1cm);

\draw[fill=black] (6.5,-4.5) circle (0.1cm);
\draw[fill=black] (7.5,-4.5) circle (0.1cm);
\draw[fill=black] (8.5,-4.5) circle (0.1cm);
\draw[fill=black] (9.5,-4.5) circle (0.1cm);
\draw[fill=black] (10.5,-4.5) circle (0.1cm);
\draw[fill=black] (11.5,-4.5) circle (0.1cm);

\draw[fill=black] (6.5,-5.5) circle (0.1cm);
\draw[fill=black] (7.5,-5.5) circle (0.1cm);
\draw[fill=black] (8.5,-5.5) circle (0.1cm);
\draw[fill=black] (9.5,-5.5) circle (0.1cm);
\draw[fill=black] (10.5,-5.5) circle (0.1cm);
\draw[fill=black] (11.5,-5.5) circle (0.1cm);

\draw[fill=black] (6.5,-6.5) circle (0.1cm);
\draw[fill=black] (7.5,-6.5) circle (0.1cm);
\draw[fill=black] (8.5,-6.5) circle (0.1cm);
\draw[fill=black] (9.5,-6.5) circle (0.1cm);
\draw[fill=black] (10.5,-6.5) circle (0.1cm);
\draw[fill=black] (11.5,-6.5) circle (0.1cm);

\draw[fill=black] (6.5,-7.5) circle (0.1cm);
\draw[fill=black] (7.5,-7.5) circle (0.1cm);
\draw[fill=black] (8.5,-7.5) circle (0.1cm);
\draw[fill=black] (9.5,-7.5) circle (0.1cm);
\draw[fill=black] (10.5,-7.5) circle (0.1cm);
\draw[fill=black] (11.5,-7.5) circle (0.1cm);

\draw[fill=black] (9.5,-8.5) circle (0.1cm);
\draw[fill=black] (10.5,-8.5) circle (0.1cm);
\draw[fill=black] (11.5,-8.5) circle (0.1cm);

\draw[fill=black] (9.5,-9.5) circle (0.1cm);
\draw[fill=black] (10.5,-9.5) circle (0.1cm);
\draw[fill=black] (11.5,-9.5) circle (0.1cm);

\draw[fill=black] (9.5,-10.5) circle (0.1cm);
\draw[fill=black] (10.5,-10.5) circle (0.1cm);
\draw[fill=black] (11.5,-10.5) circle (0.1cm);
\end{tikzpicture} \qquad 
\begin{tikzpicture}[scale=0.35]
% Sfondo e griglia
\draw[help lines, very thick, white, fill=teal!5] (0,1) -- (0,-11) -- (12,-11)--(12,1)--(0,1);
\draw[help lines, dashed] (3,0.95) -- (3,-10.95);
\draw[help lines, dashed] (6,0.95) -- (6,-10.95);
\draw[help lines, dashed] (9,0.95) -- (9,-10.95);
\draw[help lines, dashed] (0.05,-2) -- (11.95,-2);
\draw[help lines, dashed] (0.05,-5) -- (11.95,-5);
\draw[help lines, dashed] (0.05,-8) -- (11.95,-8);

% Etichette di riga
\node[right] at (12,-0.7) {$D^B_1$};
\node[right] at (12,-3.7) {$D^B_2$};
\node[right] at (12,-6.7) {$D^B_3$};
\node[right] at (12,-9.7) {$D^B_4$};

% Comando per punto
\newcommand{\dotuno}[2]{\draw[fill=black] (#1,#2) circle (0.1cm);}

% Riga 1
\foreach \x in {0.5,1.5,...,11.5} { \dotuno{\x}{0.5} }
% Riga 2
\foreach \x in {0.5,1.5,...,11.5} { \dotuno{\x}{-0.5} }
% Riga 3
\foreach \x in {0.5,1.5,...,11.5} { \dotuno{\x}{-1.5} }

% Riga 4–8 (colonne da 6.5 a 11.5)
\foreach \y in {-2.5,-3.5,-4.5,-5.5,-6.5,-7.5} {
  \foreach \x in {6.5,7.5,8.5,9.5,10.5,11.5} {
    \dotuno{\x}{\y}
  }
}

% Riga 9–11 (colonne da 9.5 a 11.5)
\foreach \y in {-8.5,-9.5,-10.5} {
  \foreach \x in {9.5,10.5,11.5} {
    \dotuno{\x}{\y}
  }
}

% Rettangoli 3x3 con margine 0.1
\draw[blue,thick] (0.1, -1.9) rectangle (2.9, 0.9);
\draw[orange,thick] (3.1, -1.9) rectangle (5.9, 0.9);
\draw[green,thick] (6.1, -1.9) rectangle (8.9, 0.9);
% \draw[red,thick] (9.1, -1.9) rectangle (11.9, 0.9);

\draw[orange,thick] (6.1, -4.9) rectangle (8.9, -2.1);
\draw[green,thick] (9.1, -4.9) rectangle (11.9, -2.1);

\draw[blue,thick] (6.1, -7.9) rectangle (8.9, -5.1);
\draw[orange,thick] (9.1, -7.9) rectangle (11.9, -5.1);

\draw[blue,thick] (9.1, -10.9) rectangle (11.9, -8.1);
\end{tikzpicture}
\end{center}
The block diagonals that contribute to $\num{F}{5}$ are $D^B_2$, $D^B_3$ and $D^B_4$. Thus, following Construction \ref{con:MSRD}, we find that $D^B_{\max}=|D^B_4|=3$ and so we can take as the field size $q\geq4$. Now place on $D^B_3$ and $D^B_4$ a $[(3\times 3)^3,15,5]_q$ MSRD code and on $D^B_2$ a $[(3\times 3)^2,6,5]_q$ MSRD code. Hence, the obtained matrix subspace is an $[\mathcal{F},{\nu_{\min}(\mathcal{F},5)},5]_q$ MFD code, since its dimension is:
\begin{equation*}
    k=15+15+6=36=\num{F}{5}.
\end{equation*}
\end{example}

\begin{remark}
    Construction~\ref{con:MSRD} can be extended to the case of non-square blocks. Let
$n_1,\ldots,n_s$ and $m$ be positive integers such that $n_i\le m$ for every
$i=1,\ldots,s$. Consider a Ferrers diagram of the form
\[
\mathcal{F}
=
[\overbrace{c_1,\ldots,c_1}^{n_1\text{ times}},
\overbrace{c_2,\ldots,c_2}^{n_2\text{ times}},
\ldots,
\overbrace{c_s,\ldots,c_s}^{n_s\text{ times}}],
\]
where $c_1\le c_2\le\cdots\le c_s\le sm$ and $m\mid c_i$ for every
$i\in [s]$. 
Writing $c_i=\ell_i m$, the $n_i$ columns of length $c_i$ are naturally partitioned into $\ell_i$ rectangular blocks of size $m\times n_i$. Thus, the Ferrers diagram is no longer decomposed into square blocks of size $m\times m$, but into rectangular blocks of size $m\times n_i$, for $i\in [s]$.

Accordingly, the construction can be extended by placing MSRD codes (e.g. linearized Reed-Solomon codes) defined
over spaces of the form
\[
\bigoplus_{i\in J}\mathbb{F}_q^{m\times n_i}, \qquad J\subseteq [s],
\] with $q$ sufficiently large (for instance, $q>s$), on the corresponding diagonal blocks of the Ferrers diagram. The notion of an MSRD-constructible pair, as well as the proof of
optimality of the construction, extends to this more general setting in a
straightforward manner.

\end{remark}

\section{MDS-constructibility and
MSRD-constructibility}\label{sec:MDSvsMSRD-constructibility}

In this section, we investigate the connection between
MSRD-constructibility of a block Ferrers diagram and the MDS-constructibility of its \emph{contraction}.

We first rigorously define the notion of \textit{expanding a dot to a block or contracting a block to a dot}. 
\begin{definition}
Let $\mathcal{P}([s]^2)$ and $\mathcal{P}([sm]^2)$ be the powerset of $[s]^2$ and $[sm]^2$, respectively. Moreover, let $\mathcal{BP}([sm]^2):=\{\,S\in \mathcal{P}([sm]^2)\,:\,S\cap Q^{(m)}_{i,j}=\emptyset \text{ or }Q^{(m)}_{i,j}\text{ for every } i,j\in[s]\,\}$, the set of all subsets of $[sm]^2$ that are unions of $m\times m$ blocks. We define the map
 \begin{align*}    
 \phi_m:\quad\mathcal{P}([s]^2)&\to\mathcal{BP}([sm]^2)\\
 S&\mapsto\bigcup_{(i,j)\in S} Q_{i,j}^{(m)} .
    \end{align*}
    If $\mathcal{F'}$ is an $m$-block Ferrers diagram of order $sm$ such that $\phi_m(\mathcal{F})=\mathcal{F'}$, then we will say that $\mathcal{F'}$ is the $m$\textit{-block version} of $\mathcal{F}$ and we will write $\mc{F}^m:=\phi_m(\mc{F})$.
\end{definition}
    The map $\phi_m$ admits as inverse the map $\phi_m^{-1}$
\begin{align*}    
\phi_m^{-1}:\quad\mathcal{BP}([sm]^2)&\to\mathcal{P}([s]^2)\\
 S&\mapsto\{\,(i,j)\,:\, S \cap Q_{i,j}^{(m)}=Q_{i,j}^{(m) }\,\}
\end{align*}
If $\mc{F}$ is an $m$-block Ferrers diagram of order $sm$, then the diagram of order $s$ given by $\phi_m^{-1}(\mc{F})$ is called the $m$\textit{-contraction} of $\mc{F}$.

\begin{remark}
In \cite{bib:neri}, the authors investigated block Ferrers diagrams, $\mc{F}$, for which $m=p^h$ for some prime number $p$ and integer $h$. The associated Ferrers diagram to the block one was called $p$-contraction, and denoted by $\mc{F}^{(p)}$. In particular, let $\mc{F}$ be an $m$-block Ferrers diagram. If $m$ is a prime power, i.e. $m=p^h$ for some $p$ prime and an integer $h>0$, then
    \begin{equation*}
        \mc{F}^{(p)}=\phi_m^{-1}(\mc{F}).
    \end{equation*}
\end{remark}

We clarify here the action of $\phi_{m}$ using a small example.

\begin{example} A Ferrers diagram $\mc{F}=[1,2,2,4]$ and its $3$-block version $\phi_3(\mc{F})=[[1,2,2,4]]_3$.
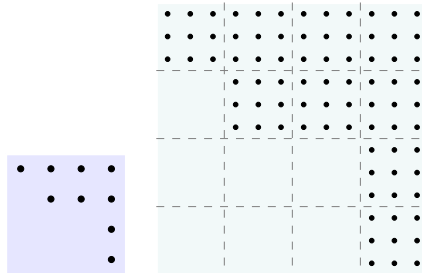
\begin{center}
\begin{tikzpicture}[scale=0.4]
\draw[help lines, very thick, white, fill=blue!10] (1,1) -- (1,-3) -- (5,-3)--(5,1)--(1,1);

\draw[fill=black] (1.5,0.5) circle (0.1cm);
\draw[fill=black] (2.5,0.5) circle (0.1cm);
\draw[fill=black] (3.5,0.5) circle (0.1cm);
\draw[fill=black] (4.5,0.5) circle (0.1cm);

\draw[fill=black] (2.5,-0.5) circle (0.1cm);
\draw[fill=black] (3.5,-0.5) circle (0.1cm);
\draw[fill=black] (4.5,-0.5) circle (0.1cm);

\draw[fill=black] (4.5,-1.5) circle (0.1cm);

\draw[fill=black] (4.5,-2.5) circle (0.1cm);
\end{tikzpicture}\quad
\begin{tikzpicture}[scale=0.3]
\draw[help lines, very thick, white, fill=teal!5] (0,1) -- (0,-11) -- (12,-11)--(12,1)--(0,1);

\draw[help lines, dashed] (3,1) -- (3,-11);
\draw[help lines, dashed] (0,-2) -- (12,-2);
\draw[help lines, dashed] (6,1) -- (6,-11);
\draw[help lines, dashed] (9,1) -- (9,-11);
\draw[help lines, dashed] (0,-5) -- (12,-5);
\draw[help lines, dashed] (0,-8) -- (12,-8);
  
\draw[fill=black] (0.5,0.5) circle (0.1cm);
\draw[fill=black] (1.5,0.5) circle (0.1cm);
\draw[fill=black] (2.5,0.5) circle (0.1cm);
\draw[fill=black] (3.5,0.5) circle (0.1cm);
\draw[fill=black] (4.5,0.5) circle (0.1cm);
\draw[fill=black] (5.5,0.5) circle (0.1cm);
\draw[fill=black] (6.5,0.5) circle (0.1cm);
\draw[fill=black] (7.5,0.5) circle (0.1cm);
\draw[fill=black] (8.5,0.5) circle (0.1cm);
\draw[fill=black] (9.5,0.5) circle (0.1cm);
\draw[fill=black] (10.5,0.5) circle (0.1cm);
\draw[fill=black] (11.5,0.5) circle (0.1cm);

\draw[fill=black] (0.5,-0.5) circle (0.1cm);
\draw[fill=black] (1.5,-0.5) circle (0.1cm);
\draw[fill=black] (2.5,-0.5) circle (0.1cm);
\draw[fill=black] (3.5,-0.5) circle (0.1cm);
\draw[fill=black] (4.5,-0.5) circle (0.1cm);
\draw[fill=black] (5.5,-0.5) circle (0.1cm);
\draw[fill=black] (6.5,-0.5) circle (0.1cm);
\draw[fill=black] (7.5,-0.5) circle (0.1cm);
\draw[fill=black] (8.5,-0.5) circle (0.1cm);
\draw[fill=black] (9.5,-0.5) circle (0.1cm);
\draw[fill=black] (10.5,-0.5) circle (0.1cm);
\draw[fill=black] (11.5,-0.5) circle (0.1cm);

\draw[fill=black] (0.5,-1.5) circle (0.1cm);
\draw[fill=black] (1.5,-1.5) circle (0.1cm);
\draw[fill=black] (2.5,-1.5) circle (0.1cm);
\draw[fill=black] (3.5,-1.5) circle (0.1cm);
\draw[fill=black] (4.5,-1.5) circle (0.1cm);
\draw[fill=black] (5.5,-1.5) circle (0.1cm);
\draw[fill=black] (6.5,-1.5) circle (0.1cm);
\draw[fill=black] (7.5,-1.5) circle (0.1cm);
\draw[fill=black] (8.5,-1.5) circle (0.1cm);
\draw[fill=black] (9.5,-1.5) circle (0.1cm);
\draw[fill=black] (10.5,-1.5) circle (0.1cm);
\draw[fill=black] (11.5,-1.5) circle (0.1cm);

\draw[fill=black] (3.5,-2.5) circle (0.1cm);
\draw[fill=black] (4.5,-2.5) circle (0.1cm);
\draw[fill=black] (5.5,-2.5) circle (0.1cm);
\draw[fill=black] (6.5,-2.5) circle (0.1cm);
\draw[fill=black] (7.5,-2.5) circle (0.1cm);
\draw[fill=black] (8.5,-2.5) circle (0.1cm);
\draw[fill=black] (9.5,-2.5) circle (0.1cm);
\draw[fill=black] (10.5,-2.5) circle (0.1cm);
\draw[fill=black] (11.5,-2.5) circle (0.1cm);

\draw[fill=black] (3.5,-3.5) circle (0.1cm);
\draw[fill=black] (4.5,-3.5) circle (0.1cm);
\draw[fill=black] (5.5,-3.5) circle (0.1cm);
\draw[fill=black] (6.5,-3.5) circle (0.1cm);
\draw[fill=black] (7.5,-3.5) circle (0.1cm);
\draw[fill=black] (8.5,-3.5) circle (0.1cm);
\draw[fill=black] (9.5,-3.5) circle (0.1cm);
\draw[fill=black] (10.5,-3.5) circle (0.1cm);
\draw[fill=black] (11.5,-3.5) circle (0.1cm);

\draw[fill=black] (3.5,-4.5) circle (0.1cm);
\draw[fill=black] (4.5,-4.5) circle (0.1cm);
\draw[fill=black] (5.5,-4.5) circle (0.1cm);
\draw[fill=black] (6.5,-4.5) circle (0.1cm);
\draw[fill=black] (7.5,-4.5) circle (0.1cm);
\draw[fill=black] (8.5,-4.5) circle (0.1cm);
\draw[fill=black] (9.5,-4.5) circle (0.1cm);
\draw[fill=black] (10.5,-4.5) circle (0.1cm);
\draw[fill=black] (11.5,-4.5) circle (0.1cm);

\draw[fill=black] (9.5,-5.5) circle (0.1cm);
\draw[fill=black] (10.5,-5.5) circle (0.1cm);
\draw[fill=black] (11.5,-5.5) circle (0.1cm);

\draw[fill=black] (9.5,-6.5) circle (0.1cm);
\draw[fill=black] (10.5,-6.5) circle (0.1cm);
\draw[fill=black] (11.5,-6.5) circle (0.1cm);

\draw[fill=black] (9.5,-7.5) circle (0.1cm);
\draw[fill=black] (10.5,-7.5) circle (0.1cm);
\draw[fill=black] (11.5,-7.5) circle (0.1cm);

\draw[fill=black] (9.5,-8.5) circle (0.1cm);
\draw[fill=black] (10.5,-8.5) circle (0.1cm);
\draw[fill=black] (11.5,-8.5) circle (0.1cm);

\draw[fill=black] (9.5,-9.5) circle (0.1cm);
\draw[fill=black] (10.5,-9.5) circle (0.1cm);
\draw[fill=black] (11.5,-9.5) circle (0.1cm);

\draw[fill=black] (9.5,-10.5) circle (0.1cm);
\draw[fill=black] (10.5,-10.5) circle (0.1cm);
\draw[fill=black] (11.5,-10.5) circle (0.1cm);
\end{tikzpicture}
\end{center}

\end{example}

\subsection{Relations between $\nu_{\min}$ of $\mc{F}$ and $\mc{F}^m$}
We begin our analysis investigating how $\nu_{\min}$ of a Ferrers diagram $\mc{F}$ is related to that of its block version $\mc{F}^m$.

\begin{lemma}
\label{le:nublock}
Let $\mc{F}$ be an $m$-block Ferrers diagram of order $sm$. Let $d\in[sm]$ and $\delta\in[s]$ be such that $d-1=m(\delta-1)+r$, with $r\in\{0,\ldots,m-1\}$. Then
    \begin{equation*}
        \nu_{\min}(\mc{F},d) =\min\{\,\nu_{km}(\mc{F},d), \nu_{km+r}(\mc{F},d)\,:\, k \in \{\,0,\ldots,\delta-1\,\}\,\}.
    \end{equation*}
\end{lemma}
\begin{proof}
To prove this, it is sufficient to see that, if we want to remove $l$ rows and columns in total from an $m$-block diagram, with $1\le l\le m$, then the best strategy is always to remove all rows or all columns. Take $i,j\in[s],a\in\{0,\ldots,l\}$ and consider removing $a$ columns from $C_j^B$ and $l-a$ rows from $R_i^B$. Without loss of generality, we can assume $|R_i^B|\geq|C_j^B|$. 

Then, the number of elements removed is
\begin{align*}
m|R_i^B|(l-a)+m|C_j^B|a-a(l-a)&\leq m|R_i^B|(l-a)+m|R_i^B|a\\
&\leq m|R_i^B|l.
\end{align*}
That is, the number of elements removed if all had been taken from $R_i^B$. Thus, we obtain the statement, since, taking $k\in\{\,0,\ldots,\delta-1\,\}$:
\begin{itemize}
    \item the value $\nu_{km}(\mc{F},d)$ counts the elements left after removing $k$ block columns and $\delta-1-k$ block rows plus $r$ rows,
    \item the value $\nu_{km+r}(\mc{F},d)$ counts the elements left after removing $k$ block columns plus $r$ columns and $\delta-1-k$ block rows.
\end{itemize}
\end{proof}

When we are dealing with an $m$-block Ferrers diagram and a pair $(\mc{F},d)$ such that $d-1$ is divisible by $m$, then the value $\nu_{\min}(\mc{F},d)$ is strictly related to its contraction $\phi_m^{-1}(\mc{F})$.

\begin{lemma}
\label{le:jsingm}
Let $\mc{F}$ be a Ferrers diagram of order $s$ and $\mc{F}^m$ be its $m$-block version. Take $\delta\in[s]$ and $d\in[sm]$ such that $d-1=m(\delta-1)$. Then, for every $j\in\{\,0,\dots,\delta-1\,\}$, $(\mc{F},\delta)$ is $j$-Singleton if and only if $(\mc{F}^m,d)$ is $mj$-Singleton. In particular, it holds
\begin{equation*}
    \nu_{\min}(\mc{F}^m,d)=m^2\num{F}{\delta}.
\end{equation*}
\end{lemma}
\begin{proof}
This result is a direct consequence of Lemma \ref{le:nublock} in case $r=0$. Thus,
\begin{equation*}
    \nu_{\min}(\mc{F}^m,d)=\underset{k\in\{\,0,\dots,\delta-1\,\}}{\min}\nu_{mk}(\mc{F}^m,d)=\underset{k\in\{\,0,\dots,\delta-1\,\}}{\min}m^2\nu_{k}(\mc{F},\delta)=m^2\num{F}{\delta},
\end{equation*}
which concludes the proof.
\end{proof}

The following lemma aims to address the relationship between $\num{F}{\delta}$ and $\nu_{\min}(\mc{F}^m,d)$ when $d-1$ is not a multiple of $m$. Observe that, compared to Lemma \ref{le:jsingm}, we need to add some hypotheses to the pair $(\mc{F},{\delta+1})$. 
\begin{lemma}
\label{le:jsing}
Let $\mc{F}$ be a Ferrers diagram of order $s$ and $\mc{F}^m$ be its $m$-block version. Let $\delta\in[s-1]$  and $d\in[sm]$ be such that $d-1=(\delta-1)m+r$, with $r\in[m-1]$. Let $j\in\{\,0,\dots,\delta-1\,\}$ be such that $(\mc{F},\delta)$ is $j$-Singleton. Then:
\begin{enumerate}
    \item if $(\mc{F},\delta+1)$ is $j$-Singleton, then $(\mc{F}^m,d)$ is $mj$-Singleton.
    
    \item If $(\mc{F},\delta+1)$ is $(j+1)$-Singleton, then $(\mc{F}^m,d)$ is $(mj+r)$-Singleton.
\end{enumerate}
Moreover, in both cases it holds
    \begin{equation*}
    \nu_{\min}(\mc{F}^m,d)=m^2\num{F}{\delta}-mr(\num{F}{\delta}-\num{F}{\delta+1}).
    \end{equation*}
\end{lemma}
\begin{proof}

Recall that from Lemma \ref{le:nublock} we have

$$
\nu_{\min}(\mc{F}^m,d) =\min\{\,\nu_{km}(\mc{F}^m,d), \nu_{km+r}(\mc{F}^m,d)\,:\, k \in \{\,0,\ldots,\delta-1\,\}\,\}.
$$
Now, note that 
$$
\begin{aligned}
\nu_{km}(\mc{F}^m,d)=&\nu_{k}(\mathcal{F},\delta+1)m^2+(\nu_{k}(\mathcal{F},\delta)-\nu_{k}(\mathcal{F},\delta+1)) (m(m-r)) \\
=& (m^2-mr)\nu_{k}(\mathcal{F},\delta)+mr\nu_{k}(\mathcal{F},\delta+1)
\end{aligned}
$$
and 
$$
\begin{aligned}
\nu_{km+r}(\mc{F}^m,d)=&\nu_{k+1}(\mathcal{F},\delta+1)m^2+(\nu_{k}(\mathcal{F},\delta)-\nu_{k+1}(\mathcal{F},\delta+1)) (m(m-r)) \\
=& (m^2-mr)\nu_{k}(\mathcal{F},\delta)+mr\nu_{k+1}(\mathcal{F},\delta+1).
\end{aligned}
$$
If both $(\mathcal{F},\delta)$ and $(\mathcal{F},\delta+1)$ are $j$-Singleton, then
$$
\begin{aligned}
\nu_{jm}(\mc{F}^m,d)=&(m^2-mr)\nu_{j}(\mathcal{F},\delta)+mr\nu_{j}(\mathcal{F},\delta+1) \\
\le&
(m^2-mr)\nu_{k}(\mathcal{F},\delta)+mr\nu_{k}(\mathcal{F},\delta+1)=\nu_{km}(\mc{F}^m,d)
\end{aligned}
$$
and
$$
\begin{aligned}
\nu_{jm}(\mc{F}^m,d)=&(m^2-mr)\nu_{j}(\mathcal{F},\delta)+mr\nu_{j}(\mathcal{F},\delta+1) \\
\le&
(m^2-mr)\nu_{k}(\mathcal{F},\delta)+mr \nu_{k+1}(\mathcal{F},\delta+1)=\nu_{km+r}(\mc{F}^m,d).
\end{aligned}
$$
Thus, $\nu_{\min}(\mc{F}^m,d)=\nu_{jm}(\mc{F}^m,d)$.

While, if $(\mathcal{F},\delta)$ is $j$-Singleton and $(\mathcal{F},\delta+1)$ is $(j+1)$-Singleton, then, in a similar way, we get
$$
\begin{aligned}
\nu_{jm+r}(\mc{F}^m,d)=&(m^2-mr)\nu_{j}(\mathcal{F},\delta)+mr\nu_{j+1}(\mathcal{F},\delta+1)\\
 \le&
(m^2-mr)\nu_{k}(\mathcal{F},\delta)+mr\nu_{k}(\mathcal{F},\delta+1)=\nu_{km}(\mc{F}^m,d)
\end{aligned}
$$
and 
$$
\begin{aligned}
\nu_{jm+r}(\mc{F}^m,d) &=(m^2-mr)\nu_{j}(\mathcal{F},\delta)+mr\nu_{j+1}(\mathcal{F},\delta+1) \\
&\le
(m^2-mr)\nu_{k}(\mathcal{F},\delta)+mr\nu_{k+1}(\mathcal{F},\delta+1)=\nu_{km+r}(\mc{F}^m,d).
\end{aligned}
$$
So, $\nu_{\min}(\mc{F}^m,d)=\nu_{jm+r}(\mc{F}^m,d)$.
\end{proof}

\begin{remark}
    Let $\mc{F}$ be a Ferrers diagram of order $s$ and $\delta\in[s-1]$, take $j\in\{\,0,\ldots,\delta-1\,\}$ such that $(\mc{F},\delta)$ is $j$-Singleton. Then $(\mc{F},\delta+1)$ is not necessarily $j$-Singleton or $(j+1)$-Singleton. As an example consider the following Ferrers diagram.

Let $\mc{F}=[0,2,2,2,4,6]$ and $\delta=3$.
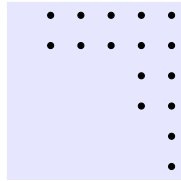
\begin{center}
\begin{tikzpicture}[scale=0.4]
\draw[help lines, very thick, white, fill=blue!10] (0,1) -- (0,-5) -- (6,-5)--(6,1)--(0,1);

\draw[fill=black] (1.5,0.5) circle (0.1cm);
\draw[fill=black] (2.5,0.5) circle (0.1cm);
\draw[fill=black] (3.5,0.5) circle (0.1cm);
\draw[fill=black] (4.5,0.5) circle (0.1cm);
\draw[fill=black] (5.5,0.5) circle (0.1cm);

\draw[fill=black] (1.5,-0.5) circle (0.1cm);
\draw[fill=black] (2.5,-0.5) circle (0.1cm);
\draw[fill=black] (3.5,-0.5) circle (0.1cm);
\draw[fill=black] (4.5,-0.5) circle (0.1cm);
\draw[fill=black] (5.5,-0.5) circle (0.1cm);

\draw[fill=black] (4.5,-1.5) circle (0.1cm);
\draw[fill=black] (5.5,-1.5) circle (0.1cm);

\draw[fill=black] (4.5,-2.5) circle (0.1cm);
\draw[fill=black] (5.5,-2.5) circle (0.1cm);

\draw[fill=black] (5.5,-3.5) circle (0.1cm);

\draw[fill=black] (5.5,-4.5) circle (0.1cm);
\end{tikzpicture}  
\end{center}
The pair $(\mc{F},3)$ is $2$-Singleton but $(\mc{F},4)$ is neither $2$-Singleton nor $3$-Singleton. 
\end{remark}

\begin{definition}
The $m$-block Ferrers diagram $\mc{T}_{s,m}^{B}$ of order $sm$ is \textit{block triangular} when
\begin{equation*}
Q_{i,j}^{(m)}\cap\mc{T}_{s,m}^B=
\begin{cases}
    Q_{i,j}^{(m)}\quad &\text{ if } i,j\in[s] \text{ and } i\leq j\\
    \emptyset &\text{otherwise.}
\end{cases}
\end{equation*}
That is, the block triangular $\mc{T}_{s,m}^{B}$ is such that $\phi_m^{-1}(\mc{T}_{s,m}^{B})=\mc{T}_s$.
\end{definition}
\begin{example}
\label{ex:tri}
The $2$-block triangular Ferrers diagram of order $8$, $\mc{T}_{4,2}^{B}$.
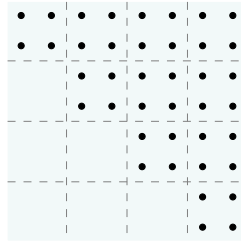
\begin{center}
\begin{tikzpicture}[scale=0.4]
\draw[help lines, very thick, white, fill=teal!5] (0,1) -- (0,-7) -- (8,-7)--(8,1)--(0,1);

\draw[help lines, dashed] (2,0.95) -- (2,-6.95);
\draw[help lines, dashed] (4,0.95) -- (4,-6.95);
\draw[help lines, dashed] (6,0.95) -- (6,-6.95);
\draw[help lines, dashed] (0.05,-1) -- (7.95,-1);
\draw[help lines, dashed] (0.05,-3) -- (7.95,-3);
\draw[help lines, dashed] (0.05,-5) -- (7.95,-5);

\draw[fill=black] (0.5,0.5) circle (0.1cm);
\draw[fill=black] (1.5,0.5) circle (0.1cm);
\draw[fill=black] (2.5,0.5) circle (0.1cm);
\draw[fill=black] (3.5,0.5) circle (0.1cm);
\draw[fill=black] (4.5,0.5) circle (0.1cm);
\draw[fill=black] (5.5,0.5) circle (0.1cm);
\draw[fill=black] (6.5,0.5) circle (0.1cm);
 \draw[fill=black] (7.5,0.5) circle (0.1cm);

\draw[fill=black] (0.5,-0.5) circle (0.1cm);
\draw[fill=black] (1.5,-0.5) circle (0.1cm);
\draw[fill=black] (2.5,-0.5) circle (0.1cm);
\draw[fill=black] (3.5,-0.5) circle (0.1cm);
\draw[fill=black] (4.5,-0.5) circle (0.1cm);
\draw[fill=black] (5.5,-0.5) circle (0.1cm);
\draw[fill=black] (6.5,-0.5) circle (0.1cm);
\draw[fill=black] (7.5,-0.5) circle (0.1cm);

\draw[fill=black] (2.5,-1.5) circle (0.1cm);
\draw[fill=black] (3.5,-1.5) circle (0.1cm);
\draw[fill=black] (4.5,-1.5) circle (0.1cm);
\draw[fill=black] (5.5,-1.5) circle (0.1cm);
\draw[fill=black] (6.5,-1.5) circle (0.1cm);
\draw[fill=black] (7.5,-1.5) circle (0.1cm);

\draw[fill=black] (2.5,-2.5) circle (0.1cm);
\draw[fill=black] (3.5,-2.5) circle (0.1cm);
\draw[fill=black] (4.5,-2.5) circle (0.1cm);
\draw[fill=black] (5.5,-2.5) circle (0.1cm);
\draw[fill=black] (6.5,-2.5) circle (0.1cm);
\draw[fill=black] (7.5,-2.5) circle (0.1cm);

\draw[fill=black] (4.5,-3.5) circle (0.1cm);
\draw[fill=black] (5.5,-3.5) circle (0.1cm);
\draw[fill=black] (6.5,-3.5) circle (0.1cm);
\draw[fill=black] (7.5,-3.5) circle (0.1cm);

\draw[fill=black] (4.5,-4.5) circle (0.1cm);
\draw[fill=black] (5.5,-4.5) circle (0.1cm);
\draw[fill=black] (6.5,-4.5) circle (0.1cm);
\draw[fill=black] (7.5,-4.5) circle (0.1cm);

\draw[fill=black] (6.5,-5.5) circle (0.1cm);
\draw[fill=black] (7.5,-5.5) circle (0.1cm);

\draw[fill=black] (6.5,-6.5) circle (0.1cm);
\draw[fill=black] (7.5,-6.5) circle (0.1cm);

\end{tikzpicture} 

\end{center}
\end{example}

\begin{definition}
Let $\mc{F}=[[c_1,\ldots,c_s]]_m$ be an $m$-block Ferrers diagram. Then we say that $\mc{F}$  is (\textit{strictly}) $m$\textit{-monotone} if its contraction $\phi_m^{-1}(\mc{F})=[c_1,\ldots,c_s]$ is (strictly) monotone. The Ferrers diagram $\mc{F}=[[c_1,\ldots,c_s]]_m$ is called (\textit{initially}) $m$\textit{-convex} if $\phi_m^{-1}(\mc{F})$ is (initially) convex.
\end{definition}

\begin{remark}
The concept of (strictly) $m$-monotone and (initially) $m$-convex diagram was already introduced in \cite{bib:neri} in the specific case of $m=p^h$ for some prime number $p$.
\end{remark}

\begin{remark}
Let us note that the adjoint diagram of a (strictly) $m$-monotone one is (initially)  $m$-convex, and vice versa.
\end{remark}

\begin{proposition}\label{prop:jsingtriang}
    Let $\mc{T}_{s,m}^{B}$ be a block triangular diagram. Let $d\in [sm]$, and let $\delta$ be such that $d-1=m(\delta-1)+r$, with $0\le r\le m-1$. Then $(\mc{T}_{s,m}^{B},d)$ is $mj$-Singleton and $(mj+r)$-Singleton for any $j\in\{0,\ldots,\delta-1\}$. In particular,
    $$
    \nu_{\min}(\mc{T}_{s,m}^{B},d)=m^2\frac{(s-\delta+1)(s-\delta+2)}{2}-mr(s-\delta+1).
    $$
\end{proposition}
\begin{proof}
    We divide the proof in two cases.

    First, assume that $\delta <s$.
    Let $\mc{T}_s=\phi_m^{-1}(\mc{T}_{s,m}^{B})$. For any $\delta\in[s]$ we have that $(\mc{T}_s,\delta)$ is $j$-Singleton, for any $j\in\{0,\ldots,\delta-1\}$. Therefore, letting $\delta$ be such that $d-1=m(\delta-1)+r$, with $0\le r\le m-1$, we have that for any $j\in\{0,\ldots,\delta-1\}$ the following two properties are satisfied:
    \begin{itemize}
        \item[(i)] $(\mc{T}_s,\delta)$ and $(\mc{T}_s,\delta+1)$ are both $j$-Singleton,
        \item[(ii)] $(\mc{T}_s,\delta)$ is $j$-Singleton and $(\mc{T}_s,\delta+1)$ is $(j+1)$-Singleton.
    \end{itemize}
    Thus, from Lemma \ref{le:jsing}, we have that $(\mc{T}_{s,m}^{B},d)$ is $mj$-Singleton and $(mj+r)$-Singleton for any $j\in\{0,\ldots,\delta-1\}$. Now, from Lemma \ref{le:jsingm}, we have 
    \begin{align*}\nu_{\min}(\mc{T}_{s,m}^{B},d)=\nu_{mj}(\mc{T}_{s,m}^{B},d)=&m^2\nu_{j}(\mc{T}_s,\delta)-mr(\nu_{j}(\mc{T}_s,\delta)-\nu_{j}(\mc{T}_s,\delta+1))\\
    =&m^2\nu_{\min}(\mc{T}_s,\delta)-mr(\nu_{\min}(\mc{T}_s,\delta)-\nu_{\min}(\mc{T}_s,\delta+1))\\
    =&m^2\frac{(s-\delta+1)(s-\delta+2)}{2}-mr(s-\delta+1).
    \end{align*}

{
    It remains to consider the case $\delta=s$. We have
$$d-1=m(s-1)+r, \qquad 0\le r\le m-1.$$
Fix $j\in\{0,\ldots,s-1\}$. After deleting the last $mj$ columns and the first $d-mj-1=m(s-j-1)+r$
rows, the only block of $\mc T^B_{s,m}$ that can contribute is the diagonal block
$Q^{(m)}_{s-j,s-j}$. In this block, exactly $r$ rows are deleted, and therefore
$$\nu_{mj}(\mc T^B_{s,m},d)=m(m-r).$$
Similarly, after deleting the last $mj+r$ columns and the first $d-(mj+r)-1=m(s-j-1)$ rows, the only block of $\mc T^B_{s,m}$ that can contribute is again
$Q^{(m)}_{s-j,s-j}$. This time, exactly $r$ columns of this block are deleted, and hence
$$\nu_{mj+r}(\mc T^B_{s,m},d)=m(m-r).$$
Thus, $(\mc T^B_{s,m},d)$ is both $mj$-Singleton and $(mj+r)$-Singleton for every
$j\in\{0,\ldots,s-1\}$. Moreover,
$$\nu_{\min}(\mc T^B_{s,m},d)=m(m-r)
= m^2\frac{(s-s+1)(s-s+2)}{2}-mr(s-s+1),$$
which is the desired formula in the case $\delta=s$.}
\end{proof}

The previous result can be extended to the more general case of $m$-monotone Ferrers diagrams.

\begin{proposition}
    Let $\mc{F}$ be an $m$-block diagram of order $sm$. Let $d\in [sm]$, and let $\delta$ be such that $d-1=m(\delta-1)+r$, with $0\le r\le m-1$. If $\mc{F}$ is $m$-monotone, then $(\mc{F},d)$ is $(m(\delta-1)+r)$-Singleton.
\end{proposition}
\begin{proof}
    First, assume $\delta <s$. Let $\mc{F}'=\phi_m^{-1}(\mc{F})$ be the contraction of $\mc{F}$. Then $\mc{F}'$ is monotone. It is easy to note (see also \cite[Remark 2.26]{bib:neri}) that $\nu_{\min}(\mc{F}',\delta)=\nu_{\delta-1}(\mc{F}',\delta)$, for any $\delta\in[s]$. Therefore, $(\mc{F}',\delta)$ is $(\delta-1)$-Singleton and $(\mc{F}',\delta+1)$ is $\delta$-Singleton, implying, by Lemma \ref{le:jsingm} (if $r=0$) and by Lemma \ref{le:jsing} (if $r\neq 0$), that $(\mc{F},d)$ is $(m(\delta-1)+r)$-Singleton.

    { It remains to consider the case $\delta=s$. In this case
\[
d-1=m(s-1)+r,
\qquad 0\le r\le m-1.
\]
We claim that $(\mc F,d)$ is $(d-1)$-Singleton. Indeed, let
$\mc F'=\phi_m^{-1}(\mc F)$ be the contraction of $\mc F$. Since $\mc F$ is
$m$-monotone, the diagram $\mc F'$ is monotone. Hence,
$$\nu_{\min}(\mc F',s)=\nu_{s-1}(\mc F',s).$$
Equivalently, for the pair
$(\mc F',s)$, the minimum is attained by deleting the last $s-1$ columns.

Passing to the $m$-block expansion, and using
$d-1=m(s-1)+r$, the deletion corresponding to the index
$d-1=m(s-1)+r$ deletes the last $d-1$ columns and no rows. This is
obtained from the block deletion corresponding to $s-1$ by deleting,
inside the remaining first block column, $r$ additional columns. Since
$\mc F'$ is monotone, the first block column of $F$ contains at most one
nonzero block, that is $Q^{(m)}_{1,1}$. Therefore, this deletion leaves
the smallest possible number of dots among all deletions of $d-1$
rows and columns. Hence
$$\nu_{\min}(\mc F,d)=\nu_{d-1}(\mc F,d),$$
that is, $(\mc F,d)$ is $(d-1)$-Singleton. Since
$d-1=m(s-1)+r=m(\delta-1)+r$, the desired conclusion follows.}
\end{proof}

\subsection{Relations between the MDS-constructibility of $\mc{F}$ and the \\ MSRD-constructibility of $\mc{F}^m$}

{
For a Ferrers diagram $\mc{F}$ of order $s$ and an integer $\delta \in [s]$, we set
$$I(\mc{F},\delta):=\{i\in[s]\,:\, |D_i\cap \mc{F}|\ge \delta\}.$$
Equivalently, $I(\mc{F},\delta)$ is the set of diagonals contributing to the
MDS-constructibility sum
$$\sum_{i=1}^s \max\{|D_i\cap \mc{F}|-\delta+1,0\}.$$
Moreover, for an $m$-block Ferrers diagram $\mc{F}$ of order $sm$ and an integer $d\in[sm]$, we set
$$I_B(\mc{F},d):=\{i\in[s] \, : \, m|D_i^B\cap \mathbf{B}(\mc{F})|-d+1>0\}.$$
Equivalently,
$$I_B(\mc{F},d)= \{i\in[s] : m|D_i^B\cap \mathbf{B}(\mc{F})|\ge d\}.$$
}

When $d$ and $\delta$ are such that $d-1=m(\delta-1)$, we can prove a necessary and sufficient condition between MDS-constructibility and MSRD-constructibility where one diagram is the $m$-block version of the other.
\begin{theorem}
\label{th:conndiv}
    Let $\mathcal{F}$ be a Ferrers diagram of order $s$ and $\mc{F}^m$ be its $m$-block version, let $\delta\in[s],d\in[sm]$ be such that $d-1=m(\delta-1)$. Then, $(\mathcal{F},\delta)$ is MDS-constructible if and only if $(\mc{F}^m,d)$ is MSRD-constructible.
\end{theorem}
\begin{proof}
As seen in Lemma \ref{le:jsingm}, with these hypotheses, for $j\in\{\,0,\dots,\delta-1\,\}$, $(\mc{F},\delta)$ is $j$-Singleton if and only if $(\mc{F}^m,d)$ is $mj$-Singleton and, moreover,
\begin{equation*}
\nu_{\min}(\mc{F}^m,d)=m^2\num{F}{\delta}.
\end{equation*}
Take $j\in\{0,\ldots,\delta-1\}$ such that $(\mc{F},\delta)$ is $j$-Singleton.
Clearly, $\nu_{\min}(\mc{F},\delta)=0$ if and only if $\nu_{\min}(\mc{F}^m,d)=0$, and in this case
there is nothing to prove. Suppose then that $\nu_{\min}(\mc{F},\delta)>0$, and set
$$I:=I(\mc{F},\delta).$$
For every $i\in[s]$, we have
$$D_i^B\cap \mathbf{B}(\mc{F}^m)=\phi_m(D_i\cap \mc{F}),$$
and therefore
$$|D_i^B\cap \mathbf{B}(\mc{F}^m)|=|D_i\cap \mc{F}|.$$
Moreover, since $d-1=m(\delta-1)$, we have
$$m|D_i^B\cap \mathbf{B}(\mc{F}^m)|-d+1=m(|D_i\cap F|-\delta+1).$$
Hence
$$I_B(\mc{F}^m,d)=I(\mc{F},\delta)=I.$$

Suppose that $(\mc{F},\delta)$ is MDS-constructible. Then
\begin{align*}
\nu_{\min}(\mc{F}^m,d)
&=m^2\nu_{\min}(\mc{F},\delta) \\
&=m^2\sum_{i=1}^s \max\{|D_i\cap \mc{F}|-\delta+1,0\} \\
&=m^2\sum_{i\in I} (|D_i\cap \mc{F}|-\delta+1) \\
&=m\sum_{i\in I} (m|D_i\cap \mc{F}|-m(\delta-1)) \\
&=m\sum_{i\in I} (m|D_i^B\cap \mathbf{B}(\mc{F}^m)|-(d-1)) \\
&=m\sum_{i=1}^s \max\{m|D_i^B\cap \mathbf{B}(\mc{F}^m)|-d+1,0\}.
\end{align*}
Thus, $(\mc{F}^m,d)$ is MSRD-constructible.

Conversely, suppose that $(\mc{F}^m,d)$ is MSRD-constructible. Then
\begin{align*}
\nu_{\min}(\mc{F},\delta)
&=\frac{1}{m^2}\nu_{\min}(\mc{F}^m,d) \\
&=\frac{1}{m}\sum_{i=1}^s
\max\{m|D_i^B\cap \mathbf{B}(\mc{F}^m)|-d+1,0\} \\
&=\frac{1}{m}\sum_{i\in I}
(m|D_i^B\cap \mathbf{B}(\mc{F}^m)|-(d-1)) \\
&=\sum_{i\in I}(|D_i\cap \mc{F}|-\delta+1) \\
&=\sum_{i=1}^s \max\{|D_i\cap \mc{F}|-\delta+1,0\}.
\end{align*}
Hence, $(\mc{F},\delta)$ is MDS-constructible.
\end{proof}

In the following, our aim is to fill the gap in the Theorem \ref{th:connection} by considering the case where $d-1=m(\delta-1)+r$, with $r\in[m-1]$. This more general case will require us to introduce some additional hypotheses on the contraction of block diagrams in order to prove MSRD-constructibility.

In order to properly set up the environment, we need to formally define the intuitive idea of a set of points to be contained in a triangular diagram of a certain size. 

\begin{definition}
Let $\mc{F}$ be a Ferrers diagram of order $s$. We define the \textit{minimum order} of $\mc{F}$ as the integer $o(\mc{F})=\max\{\,|R_1\cap\mc{F}|,|C_s\cap\mc{F}|\,\}$.
\end{definition}

\begin{remark}
If $\mc{F}$ is a Ferrers diagram of order $s$, then one can define a Ferrers diagram of order $o(\mc{F})$ with the same structure of $\mc{F}$ removing a suitable number of empty rows and empty columns.
\end{remark}
 
\begin{definition}
Let $\mc{F}$ be a Ferrers diagram of order $s$ and let $o(\mc{F})$ be its minimum order. The \textit{standard version} of $\mc{F}$ is the Ferrers diagram contained in $[o(\mc{F})]^2$ and obtained from $\mc{F}$ in the following way
\begin{equation*}
\mc{F}_{\mathrm{st}}=\{\,(i,j-(s-o(\mc{F})))\,:\,(i,j)\in\mc{F}\,\}.
\end{equation*}
\end{definition}
\begin{remark}
In particular, $o(\mc{F})=\min\{\,l\in\mathbb{N}\,:\,\exists c \in\mathbb{N} \text{ s.t. }\mc{F}\subseteq[l]\times\{\,c,\dots,c+l-1\,\}\,\}$. Thus, there exists some $c\in\mathbb{N}$ such that $\mc{F}\subseteq[o(\mc{F)}]\times\{\,c,\dots,c+o(\mc{F})-1\,\}$ and $\mc{F}_{\mathrm{st}}$ is the diagram that has the same structure as $\mc{F}$ but is contained in $[o(\mc{F})]^2$. 
\end{remark}
At this point, we can give the notion of being contained in a triangular diagram of the right order.
\begin{definition}
Let $\mc{F}$ be a Ferrers diagram of order $s$. We say that $\mc{F}$ is  \textit{triangular covered} if {$\mc F=\emptyset$ or } $\mc{F}_{\mathrm{st}}\subseteq\mc{T}_{o(\mc{F})}$.
\end{definition}
%Since the previous definitions are not very intuitive, we explain them through some examples. 
\begin{example}
Consider the following Ferrers diagrams of order $5$.

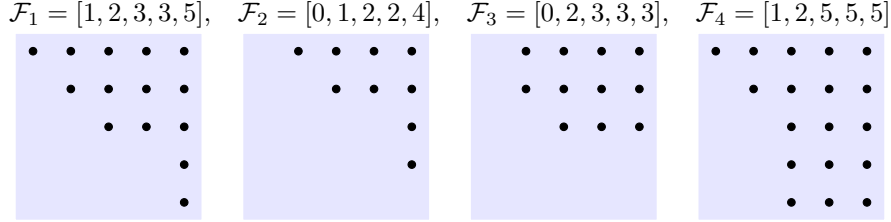
\begin{center}
$\mc{F}_1=[1,2,3,3,5]$,\quad $\mc{F}_2=[0,1,2,2,4]$,\quad $\mc{F}_3=[0,2,3,3,3]$,\quad $\mc{F}_4=[1,2,5,5,5]$

\begin{tikzpicture}[scale=0.5]
\draw[help lines, very thick, white, fill=blue!10] (0,1) -- (0,-4) -- (5,-4)--(5,1)--(0,1);

\draw[fill=black] (0.5,0.5) circle (0.1cm);
\draw[fill=black] (1.5,0.5) circle (0.1cm);
\draw[fill=black] (2.5,0.5) circle (0.1cm);
\draw[fill=black] (3.5,0.5) circle (0.1cm);
\draw[fill=black] (4.5,0.5) circle (0.1cm);

\draw[fill=black] (1.5,-0.5) circle (0.1cm);
\draw[fill=black] (2.5,-0.5) circle (0.1cm);
\draw[fill=black] (3.5,-0.5) circle (0.1cm);
\draw[fill=black] (4.5,-0.5) circle (0.1cm);

\draw[fill=black] (2.5,-1.5) circle (0.1cm);
\draw[fill=black] (3.5,-1.5) circle (0.1cm);
\draw[fill=black] (4.5,-1.5) circle (0.1cm);

\draw[fill=black] (4.5,-2.5) circle (0.1cm);
\draw[fill=black] (4.5,-3.5) circle (0.1cm);
\end{tikzpicture} \quad
\begin{tikzpicture}[scale=0.5]
\draw[help lines, very thick, white, fill=blue!10] (0,1) -- (0,-4) -- (5,-4)--(5,1)--(0,1);

\draw[fill=black] (1.5,0.5) circle (0.1cm);
\draw[fill=black] (2.5,0.5) circle (0.1cm);
\draw[fill=black] (3.5,0.5) circle (0.1cm);
\draw[fill=black] (4.5,0.5) circle (0.1cm);

\draw[fill=black] (2.5,-0.5) circle (0.1cm);
\draw[fill=black] (3.5,-0.5) circle (0.1cm);
\draw[fill=black] (4.5,-0.5) circle (0.1cm);

\draw[fill=black] (4.5,-1.5) circle (0.1cm);

\draw[fill=black] (4.5,-2.5) circle (0.1cm);
\end{tikzpicture} \quad
\begin{tikzpicture}[scale=0.5]
\draw[help lines, very thick, white, fill=blue!10] (0,1) -- (0,-4) -- (5,-4)--(5,1)--(0,1);

\draw[fill=black] (1.5,0.5) circle (0.1cm);
\draw[fill=black] (2.5,0.5) circle (0.1cm);
\draw[fill=black] (3.5,0.5) circle (0.1cm);
\draw[fill=black] (4.5,0.5) circle (0.1cm);

\draw[fill=black] (1.5,-0.5) circle (0.1cm);
\draw[fill=black] (2.5,-0.5) circle (0.1cm);
\draw[fill=black] (3.5,-0.5) circle (0.1cm);
\draw[fill=black] (4.5,-0.5) circle (0.1cm);

\draw[fill=black] (2.5,-1.5) circle (0.1cm);
\draw[fill=black] (3.5,-1.5) circle (0.1cm);
\draw[fill=black] (4.5,-1.5) circle (0.1cm);

\end{tikzpicture} \quad
\begin{tikzpicture}[scale=0.5]
\draw[help lines, very thick, white, fill=blue!10] (0,1) -- (0,-4) -- (5,-4)--(5,1)--(0,1);

\draw[fill=black] (0.5,0.5) circle (0.1cm);
\draw[fill=black] (1.5,0.5) circle (0.1cm);
\draw[fill=black] (2.5,0.5) circle (0.1cm);
\draw[fill=black] (3.5,0.5) circle (0.1cm);
\draw[fill=black] (4.5,0.5) circle (0.1cm);

\draw[fill=black] (1.5,-0.5) circle (0.1cm);
\draw[fill=black] (2.5,-0.5) circle (0.1cm);
\draw[fill=black] (3.5,-0.5) circle (0.1cm);
\draw[fill=black] (4.5,-0.5) circle (0.1cm);

\draw[fill=black] (2.5,-1.5) circle (0.1cm);
\draw[fill=black] (3.5,-1.5) circle (0.1cm);
\draw[fill=black] (4.5,-1.5) circle (0.1cm);

\draw[fill=black] (2.5,-2.5) circle (0.1cm);
\draw[fill=black] (3.5,-2.5) circle (0.1cm);
\draw[fill=black] (4.5,-2.5) circle (0.1cm);

\draw[fill=black] (2.5,-3.5) circle (0.1cm);
\draw[fill=black] (3.5,-3.5) circle (0.1cm);
\draw[fill=black] (4.5,-3.5) circle (0.1cm);
\end{tikzpicture} 
\end{center}
$\mc{F}_1$ and $\mc{F}_4$ are Ferrers diagrams of order $5$ with minimum order $o(\mc{F}_1)=o(\mc{F}_4)=5$. Hence, they coincide with their respective standard versions, so $\mc{F}_1={\mc{F}_1}_{\mathrm{st}}$ and $\mc{F}_4={\mc{F}_4}_{\mathrm{st}}$.

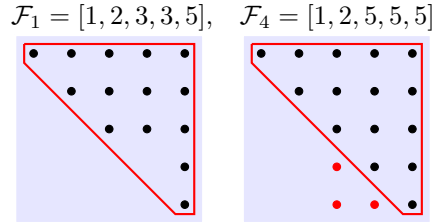
\begin{center}
$\mc{F}_1=[1,2,3,3,5]$,\quad $\mc{F}_4=[1,2,5,5,5]$

\begin{tikzpicture}[scale=0.5]
\draw[help lines, very thick, white, fill=blue!10] (0,1) -- (0,-4) -- (5,-4)--(5,1)--(0,1);
\draw[help lines, thick, red](0.25,0.75)--(4.75,0.75)--(4.75,-3.75)--(4.25,-3.75)--(0.25,0.25)--(0.25,0.75);
\draw[fill=black] (0.5,0.5) circle (0.1cm);
\draw[fill=black] (1.5,0.5) circle (0.1cm);
\draw[fill=black] (2.5,0.5) circle (0.1cm);
\draw[fill=black] (3.5,0.5) circle (0.1cm);
\draw[fill=black] (4.5,0.5) circle (0.1cm);

\draw[fill=black] (1.5,-0.5) circle (0.1cm);
\draw[fill=black] (2.5,-0.5) circle (0.1cm);
\draw[fill=black] (3.5,-0.5) circle (0.1cm);
\draw[fill=black] (4.5,-0.5) circle (0.1cm);

\draw[fill=black] (2.5,-1.5) circle (0.1cm);
\draw[fill=black] (3.5,-1.5) circle (0.1cm);
\draw[fill=black] (4.5,-1.5) circle (0.1cm);

\draw[fill=black] (4.5,-2.5) circle (0.1cm);
\draw[fill=black] (4.5,-3.5) circle (0.1cm);
\end{tikzpicture} \quad
\begin{tikzpicture}[scale=0.5]
\draw[help lines, very thick, white, fill=blue!10] (0,1) -- (0,-4) -- (5,-4)--(5,1)--(0,1);
\draw[help lines, thick, red](0.25,0.75)--(4.75,0.75)--(4.75,-3.75)--(4.25,-3.75)--(0.25,0.25)--(0.25,0.75);
\draw[fill=black] (0.5,0.5) circle (0.1cm);
\draw[fill=black] (1.5,0.5) circle (0.1cm);
\draw[fill=black] (2.5,0.5) circle (0.1cm);
\draw[fill=black] (3.5,0.5) circle (0.1cm);
\draw[fill=black] (4.5,0.5) circle (0.1cm);

\draw[fill=black] (1.5,-0.5) circle (0.1cm);
\draw[fill=black] (2.5,-0.5) circle (0.1cm);
\draw[fill=black] (3.5,-0.5) circle (0.1cm);
\draw[fill=black] (4.5,-0.5) circle (0.1cm);

\draw[fill=black] (2.5,-1.5) circle (0.1cm);
\draw[fill=black] (3.5,-1.5) circle (0.1cm);
\draw[fill=black] (4.5,-1.5) circle (0.1cm);

\draw[red,fill=red] (2.5,-2.5) circle (0.1cm);
\draw[fill=black] (3.5,-2.5) circle (0.1cm);
\draw[fill=black] (4.5,-2.5) circle (0.1cm);

\draw[red,fill=red] (2.5,-3.5) circle (0.1cm);
\draw[red,fill=red] (3.5,-3.5) circle (0.1cm);
\draw[fill=black] (4.5,-3.5) circle (0.1cm);
\end{tikzpicture}
\end{center}
Clearly, $\mc{F}_1$ is triangular covered, whereas $\mc{F}_4$ is not.

$\mc{F}_2$ and $\mc{F}_3$ are Ferrers diagrams of order $5$ with minimum order $o(\mc{F}_2)=o(\mc{F}_3)=4$. Their standard versions are

\begin{center}
${\mc{F}_2}_{\mathrm{st}}=[1,2,2,4]$,\quad ${\mc{F}_3}_{\mathrm{st}}=[2,3,3,3]$,

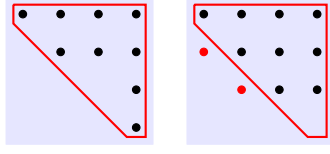
\begin{tikzpicture}[scale=0.5]
\draw[help lines, very thick, white, fill=blue!10] (1,1) -- (1,-3) -- (5,-3)--(5,1)--(1,1);
\draw[help lines, thick, red](1.25,0.75)--(4.75,0.75)--(4.75,-2.75)--(4.25,-2.75)--(1.25,0.25)--(1.25,0.75);

\draw[fill=black] (1.5,0.5) circle (0.1cm);
\draw[fill=black] (2.5,0.5) circle (0.1cm);
\draw[fill=black] (3.5,0.5) circle (0.1cm);
\draw[fill=black] (4.5,0.5) circle (0.1cm);

\draw[fill=black] (2.5,-0.5) circle (0.1cm);
\draw[fill=black] (3.5,-0.5) circle (0.1cm);
\draw[fill=black] (4.5,-0.5) circle (0.1cm);

\draw[fill=black] (4.5,-1.5) circle (0.1cm);

\draw[fill=black] (4.5,-2.5) circle (0.1cm);
\end{tikzpicture}\quad
\begin{tikzpicture}[scale=0.5]
\draw[help lines, very thick, white, fill=blue!10] (1,1) -- (1,-3) -- (5,-3)--(5,1)--(1,1);
\draw[help lines, thick, red](1.25,0.75)--(4.75,0.75)--(4.75,-2.75)--(4.25,-2.75)--(1.25,0.25)--(1.25,0.75);

\draw[fill=black] (1.5,0.5) circle (0.1cm);
\draw[fill=black] (2.5,0.5) circle (0.1cm);
\draw[fill=black] (3.5,0.5) circle (0.1cm);
\draw[fill=black] (4.5,0.5) circle (0.1cm);

\draw[red,fill=red] (1.5,-0.5) circle (0.1cm);
\draw[fill=black] (2.5,-0.5) circle (0.1cm);
\draw[fill=black] (3.5,-0.5) circle (0.1cm);
\draw[fill=black] (4.5,-0.5) circle (0.1cm);

\draw[red,fill=red] (2.5,-1.5) circle (0.1cm);
\draw[fill=black] (3.5,-1.5) circle (0.1cm);
\draw[fill=black] (4.5,-1.5) circle (0.1cm);

\end{tikzpicture} 
\end{center}
thus, $\mc{F}_2$ is triangular covered, whereas $\mc{F}_3$ is not.
\end{example}

We recall the following result from \cite{bib:neri}.

\begin{lemma}[{\cite[Lemma 4.20]{bib:neri}}]\label{lem:puntidots}
Let $\mc{F}$ be a Ferrers diagram of order $n$ and let $d\in\{\,2,\ldots,n\,\}$, $j\in\{\,0,\ldots,d-1\,\}$. Assume that $(\mc{F},d)$ is MDS-constructible and $j$-Singleton. Then the following hold
\begin{enumerate}
    \item  One has $\mc{F}\cap\mc{S}_{n,d,j}=\mc{F}\cap\mc{T}_{n,d,j}$.
    \item  One has 
%$\{\,i\in\{\,d,\ldots,n\,\}\,:\,|D_i\cap \mc{F}|\geq{d}\,\}
$I(\mc{F},d)
=\{\,i\in\{\,d,\ldots,n\,\}\,:\,D_i\cap \mc{F}\cap\mc{S}_{n,d,j}\neq\emptyset\,\}$.
    \item  If $i\in\{\,d,\ldots,n\,\}$ and $D_i\cap \mc{F}\cap\mc{S}_{n,d,j}\neq\emptyset$, then $D_i\cap\mc{F}\supseteq D_i\cap\mc{L}_{n,d,j}$.
\end{enumerate}
\end{lemma}

With the following result we aim to address the relationship between MDS-constructibility and MSRD-constructibility in the general case.

\begin{theorem}
\label{th:connection}
    Let $\mathcal{F}$ be a Ferrers diagram of order $s$ and $\mc{F}^m$ be its $m$-block version, let $\delta\in[s], d\in[sm]$  be such that $d-1=m(\delta-1)+r$, with $r\in[m-1]$. Suppose that $(\mathcal{F},\delta)$ is MDS-constructible, and 
    \begin{enumerate}
        \item either $\delta=s$,
        \item or $\delta \leq s-1$ and there exists some $j\in\{\,0,\ldots,\delta-1\,\}$ such that 
    \begin{itemize}
        \item $(\mathcal{F},\delta)$ is $j$-Singleton and $(\mc{F},\delta+1)$ is $j$-Singleton or $(j+1)$-Singleton,
              \item $\mc{F}\cap\mc{S}_{s,\delta,j}$ is triangular covered.
     
    \end{itemize}
    \end{enumerate} 
    Then, $(\mathcal{F}^m,d)$ is MSRD-constructible.
\end{theorem}
\begin{proof}
{
Assume (1) and, so, $\delta=s$. In this case, one has
$$\nu_{\min}(\mc{F},s)=\sum_{i=1}^s\max\{|D_i\cap \mc{F}|-s+1,0\}=\max\{|D_s\cap \mc{F}|-s+1,0\}\in\{0,1\}.$$
Let $j$ be such that $(\mc{F},s)$ is $j$-Singleton. If $\nu_j(\mc{F},s)=\nu_{\min}(\mc{F},s)=0$, then, also $\nu_{mj}(\mc{F}^m,d)=0=\nu_{\min}(\mc{F}^m,d)$, and thus $(\mc{F}^m,d)$ is clearly MSRD-constructible. 
If we have instead that $\nu_j(\mc{F},s)=\nu_{\min}(\mc{F},s)=1$, then $D_s\cap \mc{F}=D_s$ is full and $\mc{F}\cap\mc{S}_{s,s,j}=\{(s-j,s-j)\}$. In particular, when considering the $m$-expansion $\mc{F}^m$, we get $\nu_{mj}(\mc{F}^m,d)=m(m-r)$. Since also  $D_s^B\cap\mathbf{B}(\mc{F}^m)=D_s^B$, we get
$$m\sum_{i=1}^s
\max\{m|D_i^B\cap \mathbf{B}(\mc{F}^m)|-d+1,0\}=m(m|D_s^B\cap \mathbf{B}(\mc{F}^m)|-d+1)=m(m-r),$$
and hence
\begin{align*} m(m-r)&=m\sum_{i=1}^s
\max\{m|D_i^B\cap \mathbf{B}(\mc{F}^m)|-d+1,0\}\\
&\le \nu_{\min}(\mc{F}^m,d)\le \nu_{mj}(\mc{F}^m,d)=m(m-r),
\end{align*}
where the first inequality holds because the left-hand side is the dimension obtained by Construction~\ref{con:MSRD} and, by the Ferrers Singleton bound of Theorem \ref{th:dimension}, this dimension is at most $\nu_{\min}(\mc{F}^m,d)$.

 This forces $\nu_{\min}(\mc{F}^m,d)=m(m-r)$ and therefore implies that $(\mc{F}^m,d)$ is MSRD-constructible.

Now, we assume (2) and hence $\delta \in [s-1]$.
Suppose that there exists some $j\in\{0,\ldots,\delta-1\}$ such that
$(\mc{F},\delta)$ is $j$-Singleton, $\mc{F}\cap \mc{S}_{s,\delta,j}$ is triangular
covered, and $(\mc{F},\delta)$ is MDS-constructible. Set
$$I:=I(\mc{F},\delta).$$

First suppose that $(\mc{F},\delta+1)$ is $j$-Singleton.
Moreover, let $\mc{F}^m$ be the $m$-block version of $\mc{F}$. We are in the
hypotheses of Lemma~\ref{le:jsing} part~(1), hence $(\mc{F}^m,d)$ is $mj$-Singleton and
$$\nu_{\min}(\mc{F}^m,d)=m^2\nu_{\min}(\mc{F},\delta)-mr\bigl(\nu_j(\mc{F},\delta)-\nu_j(\mc{F},\delta+1)\bigr).$$
Observe that
$$\nu_j(\mc{F},\delta)-\nu_j(\mc{F},\delta+1)=|(\mc{F}\cap \mc{S}_{s,\delta,j})\setminus(\mc{F}\cap \mc{S}_{s,\delta+1,j})|=|\{(\delta-j,l)\in \mc{F}\,:\,l\in[s-j]\}|.$$
Set
$$k:=\nu_j(\mc{F},\delta)-\nu_j(\mc{F},\delta+1).$$
The number of dots in the first row of $\mc{F}\cap \mc{S}_{s,\delta,j}$ is
$k=\nu_j(\mc{F},\delta)-\nu_j(\mc{F},\delta+1),$
whereas the number of dots in its last column is
$\nu_j(\mc{F},\delta)-\nu_{j+1}(\mc{F},\delta+1).$
Since $(\mc{F},\delta+1)$ is also $j$-Singleton, the first row of
$\mc{F}\cap \mc{S}_{s,\delta,j}$ has length greater than or equal to its last
column. Hence
$k=o(\mc{F}\cap \mc{S}_{s,\delta,j}).$
By Lemma~\ref{lem:puntidots}, we have
$$I=\{i\in[s]\,:\,D_i\cap \mc{F}\cap \mc{S}_{s,\delta,j}\neq\emptyset\}.$$
Since $\mc{F}\cap \mc{S}_{s,\delta,j}$ is triangular covered, it follows that $|I|=k$.
Therefore, using the MDS-constructibility of $(\mc{F},\delta)$, we get
\begin{align*}
\nu_{\min}(\mc{F}^m,d)
&=m^2\nu_{\min}(\mc{F},\delta)-mr|I| \\
&=m^2\sum_{i=1}^s\max\{|D_i\cap \mc{F}|-\delta+1,0\}-mr|I| \\
&=m^2\sum_{i\in I}(|D_i\cap \mc{F}|-\delta+1)-mr|I| \\
&=m\sum_{i\in I}\bigl(m|D_i\cap \mc{F}|-m(\delta-1)-r\bigr) \\
&=m\sum_{i\in I}\bigl(m|D_i\cap \mc{F}|-(d-1)\bigr).
\end{align*}
For every $i\in[s]$ we have
$$D_i^B\cap \mathbf{B}(\mc{F}^m)=\phi_m(D_i\cap \mc{F}),$$
and hence
$$|D_i^B\cap \mathbf{B}(\mc{F}^m)|=|D_i\cap \mc{F}|.$$
Moreover, since $1\le r\le m-1$, we have
$$m|D_i^B\cap \mathbf{B}(\mc{F}^m)|-d+1>0
\quad\Longleftrightarrow\quad
|D_i\cap \mc{F}|\ge\delta.$$
Thus
$$I_B(\mc{F}^m,d)=I(\mc{F},\delta)=I.$$
Consequently,
\begin{align*}
\nu_{\min}(\mc{F}^m,d)
&=m\sum_{i\in I}\bigl(m|D_i^B\cap \mathbf{B}(\mc{F}^m)|-(d-1)\bigr) \\
&=m\sum_{i=1}^s
\max\{m|D_i^B\cap \mathbf{B}(\mc{F}^m)|-d+1,0\}.
\end{align*}
Therefore, $(\mc{F}^m,d)$ is MSRD-constructible.
}
{
Now suppose that $(\mc{F},\delta+1)$ is $(j+1)$-Singleton. Then we are in the
hypotheses of Lemma~\ref{le:jsing} part~(2), so $(\mc{F}^m,d)$ is $(mj+r)$-Singleton and
$$\nu_{\min}(\mc{F}^m,d)=
m^2\nu_{\min}(\mc{F},\delta)- mr\bigl(\nu_j(\mc{F},\delta)-\nu_{j+1}(\mc{F},\delta+1)\bigr).$$
Observe that
$$\nu_j(\mc{F},\delta)-\nu_{j+1}(\mc{F},\delta+1)=|(\mc{F}\cap \mc{S}_{s,\delta,j})\setminus(\mc{F}\cap \mc{S}_{s,\delta+1,j+1})|=|\{(i,s-j)\in F:i\in\{\delta-j,\ldots,s\}\}|.$$
Set
$$k:=\nu_j(\mc{F},\delta)-\nu_{j+1}(\mc{F},\delta+1).$$
Since $(\mc{F},\delta+1)$ is $(j+1)$-Singleton, the last column of
$\mc{F}\cap \mc{S}_{s,\delta,j}$ has length greater than or equal to its first row.
Hence, $k=o(\mc{F}\cap \mc{S}_{s,\delta,j})$.
By Lemma~\ref{lem:puntidots}, we have
$$I=\{i\in[s]:D_i\cap \mc{F}\cap \mc{S}_{s,\delta,j}\neq\emptyset\}.$$
Since $\mc{F}\cap \mc{S}_{s,\delta,j}$ is triangular covered, it follows that $|I|=k$.
Therefore, using the MDS-constructibility of $(\mc{F},\delta)$, we obtain
\begin{align*}
\nu_{\min}(\mc{F}^m,d)
&=m^2\nu_{\min}(\mc{F},\delta)-mr|I| \\
&=m^2\sum_{i=1}^s\max\{|D_i\cap \mc{F}|-\delta+1,0\}-mr|I| \\
&=m^2\sum_{i\in I}(|D_i\cap \mc{F}|-\delta+1)-mr|I| \\
&=m\sum_{i\in I}\bigl(m|D_i\cap \mc{F}|-m(\delta-1)-r\bigr) \\
&=m\sum_{i\in I}\bigl(m|D_i\cap \mc{F}|-d+1\bigr).
\end{align*}
As in the previous case, we have
$$I_B(\mc{F}^m,d)=I(\mc{F},\delta)=I.$$
Thus
\begin{align*}
\nu_{\min}(\mc{F}^m,d)
&=m\sum_{i\in I}\bigl(m|D_i^B\cap \mathbf{B}(\mc{F}^m)|-(d-1)\bigr) \\
&=m\sum_{i=1}^s
\max\{m|D_i^B\cap \mathbf{B}(\mc{F}^m)|-d+1,0\}.
\end{align*}
Therefore, $(\mc{F}^m,d)$ is MSRD-constructible. }
\end{proof}

In the following example, we show that if the triangular covered property fails, then the block Ferrers diagram may not be MSRD-constructible. 
\begin{example}
Consider $\mc{F}=[0,0,2,4,4,6,7]$, and $\mc{F}^2$ its $2$-block version. Let us fix  
 $d=6$, which corresponds to $\delta=3$. We can observe that $(\mc{F},3)$ is MDS-constructible, it is $2$-Singleton and $(\mc{F},4)$ is $3$-Singleton, therefore, to apply Theorem \ref{th:connection} we are left to check if $\mc{F}\cap\setq{S}{7}{3}{2}$ is triangular covered:
        \begin{center}
        $\qquad\mc{F}\cap\setq{S}{7}{3}{2}\quad\quad\quad\quad(\mc{F}\cap\setq{S}{7}{3}{2})_{\mathrm{st}}$
        
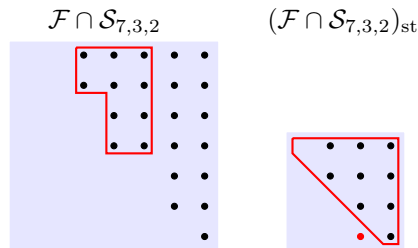
        \begin{tikzpicture}[scale=0.4]
\draw[help lines, very thick, white, fill=blue!10] (-2,1)-- (-2,-6)-- (5,-6)--(5,1)--(-2,1);

\draw[help lines, thick, red] (0.25,0.75)-- (0.25,-0.75)-- (1.25,-0.75)--(1.25,-2.75)--(1.25,-2.75)--(2.75,-2.75)--(2.75,0.75)--(0.25,0.75);

 %prima riga
\draw[fill=black] (0.5,0.5) circle (0.1cm);
\draw[fill=black] (1.5,0.5) circle (0.1cm);
\draw[fill=black] (2.5,0.5) circle (0.1cm);
\draw[fill=black] (3.5,0.5) circle (0.1cm);
\draw[fill=black] (4.5,0.5) circle (0.1cm);

%seconda riga
\draw[fill=black] (0.5,-0.5) circle (0.1cm);
\draw[fill=black] (1.5,-0.5) circle (0.1cm);
\draw[fill=black] (2.5,-0.5) circle (0.1cm);
\draw[fill=black] (3.5,-0.5) circle (0.1cm);
\draw[fill=black] (4.5,-0.5) circle (0.1cm);

%terza riga
\draw[fill=black] (1.5,-1.5) circle (0.1cm);
\draw[fill=black] (2.5,-1.5) circle (0.1cm);
\draw[fill=black] (3.5,-1.5) circle (0.1cm);
\draw[fill=black] (4.5,-1.5) circle (0.1cm);

%quarta riga
\draw[fill=black] (1.5,-2.5) circle (0.1cm);
\draw[fill=black] (2.5,-2.5) circle (0.1cm);
\draw[fill=black] (3.5,-2.5) circle (0.1cm);
\draw[fill=black] (4.5,-2.5) circle (0.1cm);

%quinta riga
\draw[fill=black] (3.5,-3.5) circle (0.1cm);
\draw[fill=black] (4.5,-3.5) circle (0.1cm);

%sesta riga
\draw[fill=black] (3.5,-4.5) circle (0.1cm);
\draw[fill=black] (4.5,-4.5) circle (0.1cm);

%settima riga
\draw[fill=black] (4.5,-5.5) circle (0.1cm);

%ottava riga
%\draw[fill=black] (4.5,-6.5) circle (0.1cm);

\end{tikzpicture}
\quad\quad\begin{tikzpicture}[scale=0.4]
\draw[help lines, very thick, white, fill=blue!10] (-1,1)-- (-1,-3)-- (3,-3)--(3,1)--(-1,1);

\draw[help lines, thick, red] (-0.75,0.75)-- (-0.75,0.25)--(2.25,-2.75)--(2.75,-2.75)--(2.75,0.75)--(-0.75,0.75);

 %prima riga
\draw[fill=black] (0.5,0.5) circle (0.1cm);
\draw[fill=black] (1.5,0.5) circle (0.1cm);
\draw[fill=black] (2.5,0.5) circle (0.1cm);

%seconda riga
\draw[fill=black] (0.5,-0.5) circle (0.1cm);
\draw[fill=black] (1.5,-0.5) circle (0.1cm);
\draw[fill=black] (2.5,-0.5) circle (0.1cm);

%terza riga
\draw[fill=black] (1.5,-1.5) circle (0.1cm);
\draw[fill=black] (2.5,-1.5) circle (0.1cm);

%quarta riga
\draw[red,fill=red] (1.5,-2.5) circle (0.1cm);
\draw[fill=black] (2.5,-2.5) circle (0.1cm);
\end{tikzpicture}
        \end{center}
        We can see that $\mc{F}\cap\setq{S}{7}{3}{2}$ is not triangular covered, therefore to verify if  $(\mc{F}^2,6)$ is MSRD-constructible we need to check it explicitly. 
        For the diagonals we have
        \begin{equation*}
            m\sum_{i=1}^s\max\{m|D^B_i \cap \mathbf{B}(\mc{F}^2)|-d+1,0\}=2(1+3+5+5+1)=30.
        \end{equation*}
        Now, from Lemma \ref{le:jsing} we have that $(\mc{F}^2,6)$ is $5$-Singleton, and hence 
        \begin{equation*}
            \nu_{\min}(\mathcal{F}^2,6)=\nu_{5}(\mc{F}^2,6)=32.
        \end{equation*}
        Thus, we conclude that $(\mc{F}^2,6)$ is not MSRD-constructible. 
\end{example}

As shown in the following propositions, the triangular covered property is related to the MDS-constructibility.

\begin{proposition}\label{prop:mdstriang}
    Let $\mc{F}$ be a Ferrers diagram of order $s$ and let $(\mc{F},\delta)$ be MDS-constructible, with $2\le \delta\le s$. Let $j\in\{0,\ldots,\delta-1\}$ be such that $(\mc{F},\delta)$ is $j$-Singleton. Then we have the following.
    \begin{enumerate}
        \item If $0<j<\delta -1$, then $\mc{F}\cap\mc{S}_{s,\delta-1,j}$ and $\mc{F}\cap\mc{S}_{s,\delta-1,j-1}$ are triangular covered.
        \item If $j=0$, then $\mc{F}\cap\mc{S}_{s,\delta-1,j}$ is triangular covered.
        \item If $j=\delta-1$, then $\mc{F}\cap\mc{S}_{s,\delta-1,j-1}$ is triangular covered.
    \end{enumerate}
\end{proposition}
\begin{proof}
    
    {
    Let us consider the case $0<j<\delta-1$, the other two cases are similar.
Set
$$I:=I(\mc{F},\delta)=\{i\in[s]: |D_i\cap \mc{F}|\ge \delta\}.$$
If $I=\emptyset$, then there is nothing to prove. Hence, assume that $I\neq\emptyset$.
By Lemma~\ref{lem:puntidots}, we have
$$I=\{i\in[s]:D_i\cap \mc{F}\cap \mc{S}_{s,\delta,j}\neq\emptyset\}.$$
Moreover, for every $i\in I$, Lemma~\ref{lem:puntidots} gives
$$D_i\cap \mc{F}\supseteq D_i\cap \mc{L}_{s,\delta,j}.$$
Set $k:=\max I$. Since $k\in I$, we have $D_k\cap \mc{F}\cap \mc{S}_{s,\delta,j}\neq\emptyset$.
By the maximality of $k$, this implies
$D_k\cap \mc{F}\cap R_{\delta-j}\neq\emptyset$ and $D_k\cap \mc{F}\cap C_{s-j+1}\neq\emptyset$.
From this we get that both
$\mc{F}\cap \mc{S}_{s,\delta-1,j}$ and $\mc{F}\cap \mc{S}_{s,\delta-1,j-1}$ are triangular covered.
    }
\end{proof}

We can reverse the previous result as follows.
\begin{proposition}\label{prop:mds_triang}
     Let $\mc{F}$ be a Ferrers diagram of order $s$ and let $2\le \delta\le s$. Suppose that:
     \begin{enumerate}
         \item $(\mc{F},\delta-1)$ is MDS-constructible;
         \item there exists $j \in \{0,\ldots,\delta-2\}$ such that:
         \begin{itemize}
             \item  $(\mc{F},\delta-1)$ is $j$-Singleton;
         \item $(\mc{F},\delta)$ is either $j$-Singleton or $(j+1)$-Singleton;
         \item $\mc{F}\cap\mc{S}_{s,\delta-1,j}$ is triangular covered.
          \end{itemize}
     \end{enumerate}
Then $(\mc{F},\delta)$ is MDS-constructible.
\end{proposition}
\begin{proof}
{
Let
$$I:=I(\mc{F},\delta-1)=\{i\in[s]: |D_i\cap \mc{F}|\ge \delta-1\}.$$
If $I=\emptyset$, then by the MDS-constructibility of $(\mc{F},\delta-1)$ we have
$\nu_{\min}(\mc{F},\delta-1)=0$. Hence also $\nu_{\min}(\mc{F},\delta)=0$, and
$$\sum_{i=1}^s\max\{|D_i\cap \mc{F}|-\delta+1,0\}=0.$$
Thus, $(\mc{F},\delta)$ is MDS-constructible.
Hence, assume that $I\neq\emptyset$.
By Lemma~\ref{lem:puntidots} applied to the MDS-constructible pair $(\mc{F},\delta-1)$, we have
$$I=\{i\in[s]:D_i\cap \mc{F}\cap \mc{S}_{s,\delta-1,j}\neq\emptyset\}.
$$

Suppose first that $(\mc{F},\delta)$ is $j$-Singleton. Let $k$ be the number of dots in the first row of
$$\mc{F}\cap \mc{S}_{s,\delta-1,j}.$$
Since $\mc{F}\cap \mc{S}_{s,\delta-1,j}$ is triangular covered and the indices in $I$ are precisely the indices of the diagonals meeting $\mc{F}\cap \mc{S}_{s,\delta-1,j}$, we have $k=|I|$.
Moreover,
$$\nu_j(\mc{F},\delta)=\nu_j(\mc{F},\delta-1)-k.$$
Therefore,
\begin{align*}
\nu_{\min}(\mc{F},\delta)
&=\nu_j(\mc{F},\delta) =\nu_j(\mc{F},\delta-1)-k \\
&=\nu_{\min}(\mc{F},\delta-1)-|I| =\sum_{i=1}^s\max\{|D_i\cap \mc{F}|-\delta+2,0\}-|I| \\
&=\sum_{i\in I}(|D_i\cap \mc{F}|-\delta+2)-|I| =\sum_{i\in I}(|D_i\cap \mc{F}|-\delta+1) \\
&=\sum_{i=1}^s\max\{|D_i\cap \mc{F}|-\delta+1,0\}.
\end{align*}
Hence, $(\mc{F},\delta)$ is MDS-constructible.

Now suppose that $(\mc{F},\delta)$ is $(j+1)$-Singleton. Let $k$ be the number of dots in the last column of
$$\mc{F}\cap \mc{S}_{s,\delta-1,j}.$$
Again, since $\mc{F}\cap \mc{S}_{s,\delta-1,j}$ is triangular covered and the indices in $I$ are precisely the indices of the diagonals meeting $\mc{F}\cap \mc{S}_{s,\delta-1,j}$, we have $k=|I|$.
Moreover,
$$\nu_{j+1}(\mc{F},\delta)=\nu_j(\mc{F},\delta-1)-k,$$
and the same computation as in the previous case gives that $(\mc{F},\delta)$ is MDS-constructible.}
\end{proof}

From Proposition \ref{prop:mdstriang} and Proposition \ref{prop:mds_triang} we have that in Theorem \ref{th:connection} we can replace the hypothesis of triangular covered with the MDS-constructibility of $(\mc{F},\delta+1)$.  

\begin{theorem}
\label{th:connection2}
Let $\mathcal{F}$ be a Ferrers diagram of order $s$ and $\mc{F}^m$ be its $m$-block version, let $\delta\in[s], d\in[sm]$  be such that $d-1=m(\delta-1)+r$, with $r\in[m-1]$. Suppose that $(\mathcal{F},\delta)$ is MDS-constructible, and 
    \begin{enumerate}
        \item either $\delta=s$,
        \item or $\delta \leq s-1$,  $(\mathcal{F},\delta+1)$ is MDS-constructible, and there exists some $j\in\{\,0,\ldots,\delta-1\,\}$ such that 
$(\mathcal{F},\delta)$ is $j$-Singleton and $(\mc{F},\delta+1)$ is $j$-Singleton or $(j+1)$-Singleton. 
    \end{enumerate} 
    Then, $(\mathcal{F}^m,d)$ is MSRD-constructible.
\end{theorem}

{
\begin{proof}
Let $d-1=m(\delta-1)+r$, with $0\le r\le m-1$. If $r=0$, the claim follows directly from Theorem \ref{th:conndiv}.

Assume now that $r\in[m-1]$. By hypothesis, the pairs $(\mc F,\delta)$ and $(\mc F,\delta+1)$ are MDS-constructible. Since $\mc F$ is monotone, one has
$$\nu_{\min}(\mc F,\delta)=\nu_{\delta-1}(\mc F,\delta)
\quad\text{and}\quad
\nu_{\min}(\mc F,\delta+1)=\nu_{\delta}(\mc F,\delta+1).$$
Moreover, by Proposition \ref{prop:mdstriang}, the diagram
$\mc F\cap \mc S_{s,\delta,\delta-1}$
is triangular covered. Therefore, all the assumptions of Theorem \ref{th:connection} are satisfied with $j=\delta-1$.
Hence $(\mc F^m,d)$ is MSRD-constructible.
\end{proof}
}

In \cite[Proposition 4.15]{bib:neri} it has been shown that, for a strictly monotone (or, by adjunction, initially convex) Ferrers diagram  $\mc{F}$ of order $s$, the pair $(\mc{F},d)$ is MDS-constructible for any $2\le d\le s$. Therefore, we get the following result.

\begin{corollary}\label{th:monoMSRD1}
 Let $\mc{F}$ be an $m$-block Ferrers diagram of order $sm$. If $\mc{F}$ is strictly $m$-monotone (or initially $m$-convex), then for any $2\le d\le sm$, the pair $(\mc{F},d)$ is MSRD-constructible. In particular, $(\mc{T}_{s,m}^B,d)$ is MSRD-constructible for every $d\in[sm]$.

\end{corollary}
\begin{proof}
{
    Assume that $\mc{F}$ is strictly $m$-monotone. 
    Let $\mc F'=\phi_m^{-1}(\mc F)$. By definition, $\mc F'$ is strictly monotone. Moreover, by \cite[Proposition 4.15]{bib:neri}  the pair $(\mc{F}',d')$ is MDS-constructible for any $2\le d'\le s$. Thus, by Theorem \ref{th:connection2}, $(\mc F,d)$ is MSRD-constructible. If $\mc F$ is initially $m$-convex, then we can consider $\mc{F}^{\tran}$, which is strictly $m$-monotone.
    
    Finally, it is immediate to verify that the pair $(\mc{T}_{s,m}^B,d)$ is MSRD-constructible also for $d=1$.}
    \end{proof}

    \begin{remark}\label{rm:monotone_triangular}
        
        {From Corollary \ref{th:monoMSRD1} we have that for a strictly $m$-monotone diagram $\mc{F}$ of order $sm$, for any $d\in[sm]$ there exists an $[\mc{F},\nu_{\min}(\mc{F},d),d]$ MFD code over $\mathbb{F}_q$ with $q>D_{\max}^B$.

        Let us note, that a similar result on the existence of an $[\mc{F},\nu_{\min}(\mc{F},d),d]$ MFD code, over a sufficiently large field, for diagrams as in Corollary \ref{th:monoMSRD1} can be obtained also from \cite[Theorem 5.7]{rakhi2023}. However, using Construction \ref{con:MSRD} the required field size is always less than or equal to the one obtained in  \cite[Theorem 5.7]{rakhi2023}. In particular, if $\mc{F}=[[c_1,\ldots,c_s]]_m$ is a strictly $m$-monotone Ferrers diagram, applying \cite[Theorem 5.7]{rakhi2023} we obtain that an $[\mc{F},\nu_{\min}(\mc{F},d),d]$ MFD code exists, for any $d\in[sm]$, over $\mathbb{F}_q$ with $q>c_s\ge D_{\max}^B$. 

        For example, considering $\mc{F}^m=[[0,1,2,4]]_m$, using Construction \ref{con:MSRD} we need $q>3$, while for applying \cite[Theorem 5.7]{rakhi2023} we need $q>4$.

        {Moreover, although having similar constraints on the field size, the construction provided in \cite{rakhi2023} does not look to be related to our Construction \ref{con:MSRD}, since it does not involve MSRD codes.}
        }
    \end{remark}

%%%%PART MOVED FROM SECTION 5 TO SECTION 4
Now, we introduce some notation to help us consider block diagonals also from the dot perspective.

Recall that, for $i\in[s]$, the set of blocks on the $i$-th diagonal is
$D^B_i=\{Q_{l,s-i+l}\in\mathbf{B}([sm]^2)\,:\,1\leq{l}\leq{i}\}$.  
Then, we define
\begin{equation*}
\Delta_i:={\bigcup_{l=1}^{i}Q_{l,s-i+l}=}\{(t,h)\in Q_{l,s-i+l}\,:\,1\leq{l}\leq{i}\},\quad i\in[s].
\end{equation*}
Informally speaking, the set $\Delta_i$ is the set of dots belonging to a block of $D_i^{B}$.

{  Note that, for an $m$-block Ferrers diagram $\mc{F}$ of order $sm$, since
$$|\Delta_i\cap \mc{F}|=m^2|D_i^B\cap \mathbf{B}(\mc{F})|,$$
we can equivalently write
$$I_B(\mc{F},d)=\{i\in[s]: |\Delta_i\cap \mc{F}|\ge md\}.$$
}

We put here a small example to visually explain the difference between $D_i^B$ and $\Delta_i$.
\begin{example}
Let $\mathcal{F}=[[1,1,2,2,5]]_2$ be a $2$-block Ferrers diagram of order $10$.
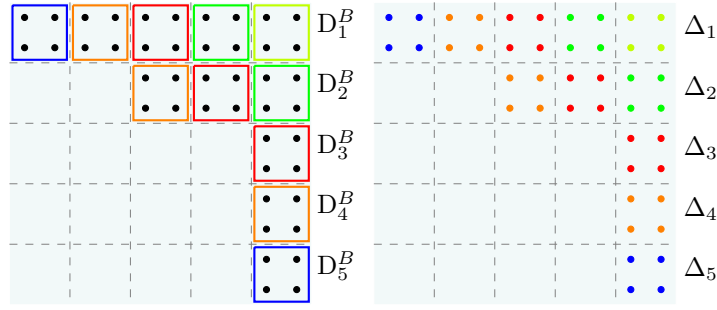
\begin{center}
\begin{tikzpicture}[scale=0.4]
\draw[help lines, white, fill=teal!5] (0,1)-- (0,-9)-- (10,-9)--(10,1)--(0,1);
\draw[help lines, thick, blue] (0.1,0.9)-- (0.1,-0.9)-- (1.9,-0.9)--(1.9,0.9)--(0.1,0.9);
\draw[help lines, thick, blue] (8.1,-7.1)-- (8.1,-8.9)-- (9.9,-8.9)--(9.9,-7.1)--(8.1,-7.1);
\draw[help lines, thick, orange] (2.1,0.9)-- (2.1,-0.9)-- (3.9,-0.9)--(3.9,0.9)--(2.1,0.9);
\draw[help lines, thick, orange] (4.1,-1.1)-- (4.1,-2.9)-- (5.9,-2.9)--(5.9,-1.1)--(4.1,-1.1);
\draw[help lines, thick, orange] (8.1,-5.1)-- (8.1,-6.9)-- (9.9,-6.9)--(9.9,-5.1)--(8.1,-5.1);
\draw[help lines, thick, red] (4.1,0.9)-- (4.1,-0.9)-- (5.9,-0.9)--(5.9,0.9)--(4.1,0.9);
\draw[help lines, thick, red] (6.1,-1.1)-- (6.1,-2.9)-- (7.9,-2.9)--(7.9,-1.1)--(6.1,-1.1);
\draw[help lines, thick, red] (8.1,-3.1)-- (8.1,-4.9)-- (9.9,-4.9)--(9.9,-3.1)--(8.1,-3.1);
\draw[help lines, thick, green] (6.1,0.9)-- (6.1,-0.9)-- (7.9,-0.9)--(7.9,0.9)--(6.1,0.9);
\draw[help lines, thick, green] (8.1,-1.1)-- (8.1,-2.9)-- (9.9,-2.9)--(9.9,-1.1)--(8.1,-1.1);
\draw[help lines, thick, lime] (8.1,0.9)-- (8.1,-0.9)-- (9.9,-0.9)--(9.9,0.9)--(8.1,0.9);

\draw[help lines, dashed] (2,1) -- (2,-9);
\draw[help lines, dashed] (0,-1) -- (10,-1);
\draw[help lines, dashed] (4,1) -- (4,-9);
\draw[help lines, dashed] (0,-3) -- (10,-3);
\draw[help lines, dashed] (6,1) -- (6,-9);
\draw[help lines, dashed] (0,-5) -- (10,-5);
\draw[help lines, dashed] (8,1) -- (8,-9);
\draw[help lines, dashed] (0,-7) -- (10,-7);

 %prima riga
\draw[fill=black] (0.5,0.5) circle (0.1cm);
\draw[fill=black] (1.5,0.5) circle (0.1cm);
\draw[fill=black] (2.5,0.5) circle (0.1cm);
\draw[fill=black] (3.5,0.5) circle (0.1cm);
\draw[fill=black] (4.5,0.5) circle (0.1cm);
\draw[fill=black] (5.5,0.5) circle (0.1cm);
\draw[fill=black] (6.5,0.5) circle (0.1cm);
\draw[fill=black] (7.5,0.5) circle (0.1cm);
\draw[fill=black] (8.5,0.5) circle (0.1cm);
\draw[fill=black] (9.5,0.5) circle (0.1cm);

%seconda riga
\draw[fill=black] (0.5,-0.5) circle (0.1cm);
\draw[fill=black] (1.5,-0.5) circle (0.1cm);
\draw[fill=black] (2.5,-0.5) circle (0.1cm);
\draw[fill=black] (3.5,-0.5) circle (0.1cm);
\draw[fill=black] (4.5,-0.5) circle (0.1cm);
\draw[fill=black] (5.5,-0.5) circle (0.1cm);
\draw[fill=black] (6.5,-0.5) circle (0.1cm);
\draw[fill=black] (7.5,-0.5) circle (0.1cm);
\draw[fill=black] (8.5,-0.5) circle (0.1cm);
\draw[fill=black] (9.5,-0.5) circle (0.1cm)node[above,xshift=15pt]{D$^B_1$};

%terza riga
\draw[fill=black] (4.5,-1.5) circle (0.1cm);
\draw[fill=black] (5.5,-1.5) circle (0.1cm);
\draw[fill=black] (6.5,-1.5) circle (0.1cm);
\draw[fill=black] (7.5,-1.5) circle (0.1cm);
\draw[fill=black] (8.5,-1.5) circle (0.1cm);
\draw[fill=black] (9.5,-1.5) circle (0.1cm);

%quarta riga
\draw[fill=black] (4.5,-2.5) circle (0.1cm);
\draw[fill=black] (5.5,-2.5) circle (0.1cm);
\draw[fill=black] (6.5,-2.5) circle (0.1cm);
\draw[fill=black] (7.5,-2.5) circle (0.1cm);
\draw[fill=black] (8.5,-2.5) circle (0.1cm);
\draw[fill=black] (9.5,-2.5) circle (0.1cm)node[above,xshift=15pt]{D$^B_2$};

\draw[fill=black] (8.5,-3.5) circle (0.1cm);
\draw[fill=black] (9.5,-3.5) circle (0.1cm);

\draw[fill=black] (8.5,-4.5) circle (0.1cm);
\draw[fill=black] (9.5,-4.5) circle (0.1cm)node[above,xshift=15pt]{D$^B_3$};

\draw[fill=black] (8.5,-5.5) circle (0.1cm);
\draw[fill=black] (9.5,-5.5) circle (0.1cm);

\draw[fill=black] (8.5,-6.5) circle (0.1cm);
\draw[fill=black] (9.5,-6.5) circle (0.1cm)node[above,xshift=15pt]{D$^B_4$};

\draw[fill=black] (8.5,-7.5) circle (0.1cm);
\draw[fill=black] (9.5,-7.5) circle (0.1cm);

\draw[fill=black] (8.5,-8.5) circle (0.1cm);
\draw[fill=black] (9.5,-8.5) circle (0.1cm)node[above,xshift=15pt]{D$^B_5$};
\end{tikzpicture}
\begin{tikzpicture}[scale=0.4]
\draw[help lines, white, fill=teal!5] (0,1)-- (0,-9)-- (10,-9)--(10,1)--(0,1);
% \draw[help lines, thick, blue] (0.1,0.9)-- (0.1,-0.9)-- (1.9,-0.9)--(1.9,0.9)--(0.1,0.9);
% \draw[help lines, thick, blue] (8.1,-7.1)-- (8.1,-8.9)-- (9.9,-8.9)--(9.9,-7.1)--(8.1,-7.1);
% \draw[help lines, thick, orange] (2.1,0.9)-- (2.1,-0.9)-- (3.9,-0.9)--(3.9,0.9)--(2.1,0.9);
% \draw[help lines, thick, orange] (4.1,-1.1)-- (4.1,-2.9)-- (5.9,-2.9)--(5.9,-1.1)--(4.1,-1.1);
% \draw[help lines, thick, orange] (8.1,-5.1)-- (8.1,-6.9)-- (9.9,-6.9)--(9.9,-5.1)--(8.1,-5.1);
% \draw[help lines, thick, red] (4.1,0.9)-- (4.1,-0.9)-- (5.9,-0.9)--(5.9,0.9)--(4.1,0.9);
% \draw[help lines, thick, red] (6.1,-1.1)-- (6.1,-2.9)-- (7.9,-2.9)--(7.9,-1.1)--(6.1,-1.1);
% \draw[help lines, thick, red] (8.1,-3.1)-- (8.1,-4.9)-- (9.9,-4.9)--(9.9,-3.1)--(8.1,-3.1);
% \draw[help lines, thick, green] (6.1,0.9)-- (6.1,-0.9)-- (7.9,-0.9)--(7.9,0.9)--(6.1,0.9);
% \draw[help lines, thick, green] (8.1,-1.1)-- (8.1,-2.9)-- (9.9,-2.9)--(9.9,-1.1)--(8.1,-1.1);
% \draw[help lines, thick, lime] (8.1,0.9)-- (8.1,-0.9)-- (9.9,-0.9)--(9.9,0.9)--(8.1,0.9);

\draw[help lines, dashed] (2,1) -- (2,-9);
\draw[help lines, dashed] (0,-1) -- (10,-1);
\draw[help lines, dashed] (4,1) -- (4,-9);
\draw[help lines, dashed] (0,-3) -- (10,-3);
\draw[help lines, dashed] (6,1) -- (6,-9);
\draw[help lines, dashed] (0,-5) -- (10,-5);
\draw[help lines, dashed] (8,1) -- (8,-9);
\draw[help lines, dashed] (0,-7) -- (10,-7);

 %prima riga
\draw[blue,fill=blue] (0.5,0.5) circle (0.1cm);
\draw[blue,fill=blue] (1.5,0.5) circle (0.1cm);
\draw[orange,fill=orange] (2.5,0.5) circle (0.1cm);
\draw[orange,fill=orange] (3.5,0.5) circle (0.1cm);
\draw[red,fill=red] (4.5,0.5) circle (0.1cm);
\draw[red,fill=red] (5.5,0.5) circle (0.1cm);
\draw[green,fill=green] (6.5,0.5) circle (0.1cm);
\draw[green,fill=green] (7.5,0.5) circle (0.1cm);
\draw[lime,fill=lime] (8.5,0.5) circle (0.1cm);
\draw[lime,fill=lime] (9.5,0.5) circle (0.1cm);

%seconda riga
\draw[blue,fill=blue] (0.5,-0.5) circle (0.1cm);
\draw[blue,fill=blue] (1.5,-0.5) circle (0.1cm);
\draw[orange,fill=orange] (2.5,-0.5) circle (0.1cm);
\draw[orange,fill=orange] (3.5,-0.5) circle (0.1cm);
\draw[red,fill=red] (4.5,-0.5) circle (0.1cm);
\draw[red,fill=red] (5.5,-0.5) circle (0.1cm);
\draw[green,fill=green] (6.5,-0.5) circle (0.1cm);
\draw[green,fill=green] (7.5,-0.5) circle (0.1cm);
\draw[lime,fill=lime] (8.5,-0.5) circle (0.1cm);
\draw[lime,fill=lime] (9.5,-0.5) circle (0.1cm)node[black,above,xshift=15pt]{$\Delta_1$};

%terza riga
\draw[orange,fill=orange] (4.5,-1.5) circle (0.1cm);
\draw[orange,fill=orange] (5.5,-1.5) circle (0.1cm);
\draw[red,fill=red] (6.5,-1.5) circle (0.1cm);
\draw[red,fill=red] (7.5,-1.5) circle (0.1cm);
\draw[green,fill=green] (8.5,-1.5) circle (0.1cm);
\draw[green,fill=green] (9.5,-1.5) circle (0.1cm);

%quarta riga
\draw[orange,fill=orange] (4.5,-2.5) circle (0.1cm);
\draw[orange,fill=orange] (5.5,-2.5) circle (0.1cm);
\draw[red,fill=red] (6.5,-2.5) circle (0.1cm);
\draw[red,fill=red] (7.5,-2.5) circle (0.1cm);
\draw[green,fill=green] (8.5,-2.5) circle (0.1cm);
\draw[green,fill=green] (9.5,-2.5) circle (0.1cm)node[black,above,xshift=15pt]{$\Delta_2$};

\draw[red,fill=red] (8.5,-3.5) circle (0.1cm);
\draw[red,fill=red] (9.5,-3.5) circle (0.1cm);

\draw[red,fill=red] (8.5,-4.5) circle (0.1cm);
\draw[red,fill=red] (9.5,-4.5) circle (0.1cm)node[black,above,xshift=15pt]{$\Delta_3$};

\draw[orange,fill=orange] (8.5,-5.5) circle (0.1cm);
\draw[orange,fill=orange] (9.5,-5.5) circle (0.1cm);

\draw[orange,fill=orange] (8.5,-6.5) circle (0.1cm);
\draw[orange,fill=orange] (9.5,-6.5) circle (0.1cm)node[black,above,xshift=15pt]{$\Delta_4$};

\draw[blue,fill=blue] (8.5,-7.5) circle (0.1cm);
\draw[blue,fill=blue] (9.5,-7.5) circle (0.1cm);

\draw[blue,fill=blue] (8.5,-8.5) circle (0.1cm);
\draw[blue,fill=blue] (9.5,-8.5) circle (0.1cm)node[black,above,xshift=15pt]{$\Delta_5$};
\end{tikzpicture} 
\end{center}
\end{example}

The following lemma refers explicitly to the block diagonal structure of the considered diagram. Thus, we adapt the proof of Lemma \ref{lem:puntidots} to block Ferrers diagrams.
\begin{lemma}\label{lem:punti}
Let $\mc{F}$ be an $m$-block Ferrers diagram of order $sm$ and let $d\in\{2,\dots,sm\}$, 
{ and let $\delta \in [s]$, $r\in \{0,\ldots,m-1\}$ such that
$d-1=m(\delta-1)+r$. Let $j\in\{km,km+r \,:\, k\in\{0,\ldots,\delta-1\}\}$ be such that $(\mc{F},d)$ is
$j$-Singleton.  Then, $(\mc{F},d)$ is MSRD-constructible if and only if the
following hold.}
\begin{enumerate}
    \item \label{lem:sin1} $\mc{F}\cap\mc{S}_{sm,d,j}=\mc{F}\cap\mc{T}_{s,m,d,j}^B$, where $\mc{T}_{s,m,d,j}^B=\mc{T}_{s,m}^B\cap\mc{S}_{sm,d,j}$.
    \item \label{lem:sin2}  
   $ I_B(\mc{F},d) = \left\{
i\in\left\{\left\lceil\frac dm\right\rceil,\ldots,s\right\}:
\Delta_i\cap \mc{F}\cap \mc{S}_{sm,d,j}\neq\emptyset
\right\}$.
    \item \label{lem:sin3} If $i\in\left\{\lceil\frac{d}{m}\rceil,\dots,s\right\}$ and $\Delta_i\cap\mc{F}\cap\mc{S}_{sm,d,j}\neq\emptyset$, then $\Delta_i\cap\mc{F}\supseteq\Delta_i\cap\mc{L}_{sm,d,j}$.
\end{enumerate}
\end{lemma}
\begin{proof}
{Assume that $(\mc{F},d)$ is MSRD-constructible. Set
$I_B:=I_B(\mc{F},d).$
Equivalently,
$$I_B=
\left\{i\in[s]:|\Delta_i\cap \mc{F}|\ge md\right\}.$$
{
We first observe that, for every
$i\in\left\{\left\lceil\frac dm\right\rceil,\ldots,s\right\}$,
one has
\[
|\Delta_i\cap \mc{L}_{sm,d,j}|=m(d-1).
\]
Indeed, write $d-1=m(\delta-1)+r$.
If $j=km$, then $\mc{L}_{sm,d,j}$ removes $k$ full block columns,
$\delta-1-k$ full block rows, and $r$ additional rows in the next
block row. Hence, on each block diagonal $\Delta_i$ with
$i\ge \lceil d/m\rceil$, this removes $(\delta-1)m^2+rm=m(d-1)$
dots. If $j=km+r$, then $\mc{L}_{sm,d,j}$ removes $k$ full block columns, $\delta-1-k$ full block rows, and $r$ additional columns in the next
block column. Again, on each such $\Delta_i$, this removes
$(\delta-1)m^2+rm=m(d-1)$
dots.}

Moreover, if $i<\left\lceil d/m\right\rceil$, then
$|\Delta_i\cap \mc{F}|\leq |\Delta_i|=im^2<md$, so no such $i$ belongs to $I_B$. Hence
$I_B\subseteq \left\{\left\lceil d/m\right\rceil,\ldots,s\right\}.$
Now, set
$$Y:=\left\{ i\in\left\{\left\lceil d/m\right\rceil,\ldots,s\right\} \,:\,
\Delta_i\cap \mc{F}\cap \mc{S}_{sm,d,j}\neq\emptyset
\right\}.$$
We claim that $I_B\subseteq Y$.}
{Indeed, let $i\in I_B$. Observe that
$\Delta_i=(\Delta_i\cap \mc{S}_{sm,d,j})\cup(\Delta_i\cap \mc{L}_{sm,d,j})$,
and hence
$$|\Delta_i\cap \mc{S}_{sm,d,j}|=|\Delta_i|-|\Delta_i\cap \mc{L}_{sm,d,j}|=
im^2-m(d-1).$$
If $\Delta_i\cap \mc{F}\cap \mc{S}_{sm,d,j}=\emptyset$, then
\begin{align*}
|\Delta_i\cap \mc{F}\cup(\Delta_i\cap \mc{S}_{sm,d,j})|
&= |\Delta_i\cap \mc{F}|+|\Delta_i\cap \mc{S}_{sm,d,j}| \\
&\ge md+im^2-m(d-1) \\
&=im^2+m \\
&>|\Delta_i|,
\end{align*}
a contradiction. Therefore
$\Delta_i\cap \mc{F}\cap \mc{S}_{sm,d,j}\neq\emptyset$,
and so $i\in Y$. Thus $I_B\subseteq Y$.

To prove \ref{lem:sin1}, observe that, since $(\mathcal{F},d)$ is $j$-Singleton, by definition
\begin{equation*}
\num{F}{d}=|\mc{F}\cap\set{S}|.
\end{equation*}
}

{
Moreover, since $(\mc{F},d)$ is MSRD-constructible and
$|\Delta_i\cap \mc{F}|=m^2|D_i^B\cap \mathbf{B}(\mc{F})|$,
we have
\begin{align*}
\nu_{\min}(\mc{F},d)
&=
m\sum_{i=1}^s
\max\{m|D_i^B\cap \mathbf{B}(\mc{F})|-(d-1),0\} \\
&=
\sum_{i\in I_B}(|\Delta_i\cap \mc{F}|-m(d-1)) =
\sum_{i\in I_B}(|\Delta_i\cap \mc{F}|-|\Delta_i\cap \mc{L}_{sm,d,j}|) \\
&\le
\sum_{i\in I_B}(|\Delta_i\cap \mc{F}|-|\Delta_i\cap \mc{F}\cap \mc{L}_{sm,d,j}|) =
\sum_{i\in I_B}|\Delta_i\cap \mc{F}\cap \mc{S}_{sm,d,j}| \\
&\le
\sum_{i\in Y}|\Delta_i\cap \mc{F}\cap \mc{S}_{sm,d,j}| =
|\mc{F}\cap \mc{T}^B_{s,m,d,j}|.
\end{align*}
Hence, $|\mc{F}\cap\set{S}|\leq|\mc{F}\cap\mc{T}_{s,m,d,j}^B|$ and, clearly, $\mc{F}\cap\mc{T}_{s,m,d,j}^B\subseteq\mc{F}\cap\set{S}$, implying $\mc{F}\cap\set{S}=\mc{F}\cap\mc{T}_{s,m,d,j}^B$.

}

{
To prove \ref{lem:sin2} and \ref{lem:sin3}, observe that all inequalities above must be equalities.
In particular, since $I_B\subseteq Y$, equality in
$$\sum_{i\in I_B}|\Delta_i\cap \mc{F}\cap \mc{S}_{sm,d,j}| \leq \sum_{i\in Y}|\Delta_i\cap \mc{F}\cap \mc{S}_{sm,d,j}|$$
implies $I_B=Y$, because every summand indexed by $Y$ is positive by definition of $Y$.
This proves \ref{lem:sin2}.
Moreover, equality in
$$\sum_{i\in I_B}(|\Delta_i\cap \mc{F}|-|\Delta_i\cap \mc{L}_{sm,d,j}|) \leq \sum_{i\in I_B}(|\Delta_i\cap \mc{F}|-|\Delta_i\cap \mc{F}\cap \mc{L}_{sm,d,j}|)$$
implies that, for every $i\in I_B=Y$,
$\Delta_i\cap \mc{L}_{sm,d,j}=\Delta_i\cap \mc{F}\cap \mc{L}_{sm,d,j}$.
Hence
$\Delta_i\cap \mc{L}_{sm,d,j}\subseteq \Delta_i\cap \mc{F}$, or, equivalently,
$$\Delta_i\cap \mc{F}\supseteq \Delta_i\cap \mc{L}_{sm,d,j},$$
which proves \ref{lem:sin3}.

Conversely, assume that the three conditions hold. By condition \ref{lem:sin2}, we have
$$I_B=\left\{ i\in\left\{\left\lceil\frac{d}{m}\right\rceil,\ldots,s\right\}:
\Delta_i\cap \mc{F}\cap \mc{S}_{sm,d,j}\neq\emptyset\right\}.$$
Using condition \ref{lem:sin1}, we get
\begin{align*}
\nu_{\min}(\mc{F},d)
&=|\mc{F}\cap \mc{S}_{sm,d,j}| =|\mc{F}\cap \mc{T}^B_{s,m,d,j}| = \sum_{i\in I_B}|\Delta_i\cap \mc{F}\cap \mc{S}_{sm,d,j}|.
\end{align*}
By condition \ref{lem:sin3}, for every $i\in I_B$,
$\Delta_i\cap \mc{L}_{sm,d,j}\subseteq \Delta_i\cap \mc{F},$
and hence
$$
|\Delta_i\cap \mc{F}\cap \mc{S}_{sm,d,j}| =|\Delta_i\cap \mc{F}|-|\Delta_i\cap \mc{L}_{sm,d,j}|.$$
Therefore,
\begin{align*}
\nu_{\min}(\mc{F},d)
&=
\sum_{i\in I_B}
\left(|\Delta_i\cap \mc{F}|-|\Delta_i\cap \mc{L}_{sm,d,j}|\right) \\
&= \sum_{i\in I_B}
\left(|\Delta_i\cap \mc{F}|-m(d-1)\right) \\
&= m\sum_{i=1}^s
\max\{m|D_i^B\cap \mathbf{B}(\mc{F})|-d+1,0\}.
\end{align*}
Thus, $(\mc{F},d)$ is MSRD-constructible.}
\end{proof}

%%END OF PART MOVED

{We can now reverse Theorem \ref{th:connection2} as follows.
\begin{theorem}
{
Let $\mathcal{F}$ be a Ferrers diagram of order $s$ and $\mc{F}^m$ be its $m$-block version. Let $\delta\in[s], d\in[sm]$  such that $d-1=m(\delta-1)+r$, with $r\in[m-1]$. Suppose that $(\mathcal{F}^m,d)$ is MSRD-constructible. The following statements hold.
    \begin{enumerate}
        \item If $\delta=s$, then $(\mathcal{F}^m,m(\delta-1)+r'+1)$ is MSRD-constructible for any $r'\in\{0,\ldots,m-1\}$, and, in particular,
$(\mathcal{F},\delta)$ MDS-constructible.
        \item If $\delta \leq s-1$ and there exists some $j\in\{\,0,\ldots,\delta-1\,\}$ such that 
  $(\mathcal{F},\delta)$ is $j$-Singleton and $(\mc{F},\delta+1)$ is $j$-Singleton or $(j+1)$-Singleton, then $(\mathcal{F}^m,m(\delta-1)+r'+1)$ is MSRD-constructible for any $r'\in\{0,\ldots,m\}$, and, in particular,
$(\mathcal{F},\delta)$ and $(\mathcal{F},\delta+1)$ are MDS-constructible.

    \end{enumerate} }
\end{theorem}
\begin{proof}
Let us consider (1). So, $d=(s-1)m+r+1$ for some $r \in [m-1]$. From Lemma \ref{le:nublock}, there exists $j\in \{0,\ldots,s-1\}$ such that 
$\nu_{\min}(\mc{F}^m,d)=\min\{\nu_{mj}(\mc{F}^m,d),\nu_{mj+r}(\mc{F}^m,d)\}$.

Moreover,
$$
\begin{aligned}
\nu_{\min}(\mc{F}^m,d)&= m\sum_{i=1}^s
\max\{m|D_i^B\cap \mathbf{B}(\mc{F}^m)|-d+1,0\}\\&= \sum_{i=1}^s
\max\{|\Delta_i\cap \mc{F}^m|-m^2s+m(m-r),0\}\in\{0,m(m-r)\}.
\end{aligned}
$$

Suppose $\nu_{\min}(\mc{F}^m,d)=0$. Note that from the block structure we have that $\nu_{mj}(\mc{F}^m,d)=0$ implies $\nu_{mj+r}(\mc{F}^m,d)=0$, and vice versa. So, $(\mc{F}^m,d)$ is $mj$-Singleton and $\nu_{mj}(\mc{F}^m,d)=0$. Always for the block structure, we get that $\nu_{mj}(\mc{F}^m,m(s-1)+r'+1)=0$ for any $r'\in\{0,   \ldots,m-1\}$. So, $(\mc{F}^m,m(s-1)+r'+1)$ is clearly MSRD-constructible for any $r'\in\{0,\ldots,m-1\}$.

If $\nu_{\min}(\mc{F}^m,d)=m(m-r)$, then $\Delta_s\cap\mc{F}^m=\Delta_s$ is full and thus, for any $r'\in\{0,\ldots,m-1\}$, $\mc{F}^m\cap\mc{S}_{sm,m(s-1)+r'+1,mj}\subseteq Q_{s-j,s-j}$, implying $\nu_{mj}(\mc{F}^m,m(s-1)+r'+1)=m(m-r')$.

Moreover,
$$
\begin{aligned}
m(m-r')&= \sum_{i=1}^s
\max\{|\Delta_i\cap \mc{F}^m|-m^2s+m(m-r')),0\}\\&\le \nu_{\min}(\mc{F}^m,m(s-1)+r'+1)\le \nu_{mj}(\mc{F}^m,m(s-1)+r'+1)=m(m-r').
\end{aligned}
$$
Therefore, $(\mc{F}^m,m(s-1)+r'+1)$ is MSRD-constructible, for any $r'\in\{0,\ldots,m-1\}$. From Theorem \ref{th:conndiv} we conclude that $(\mathcal{F},\delta)$ is MDS-constructible.

{Let us, now, consider case (2).}
    We can suppose that $(\mathcal{F},\delta)$ and
 $(\mc{F},\delta+1)$ are both $j$-Singleton (the other case is similar).
 Now, from Lemma \ref{le:jsingm} and Lemma \ref{le:jsing} we have that $(\mathcal{F}^m,m(\delta-1)+r'+1)$ is $mj$-Singleton for any $r'\in\{0,\ldots,m\}$, so $\nu_{\min}(\mathcal{F}^m,m(\delta-1)+r'+1)=|\mc{F}^m\cap\mc{S}_{sm,m(\delta-1)+r'+1,mj}|$. 

 Since $(\mathcal{F}^m,d)$ is MSRD-constructible, from Lemma \ref{lem:punti} we get
 $\mc{F}^m\cap\mc{S}_{sm,d,mj}=\mc{F}^m\cap\mc{T}_{s,m,d,mj}^B$, and due to the block structure of $\mc{F}^m$, denoting by $d'=m(\delta-1)+r'+1$, it must hold $\mc{F}^m\cap\mc{S}_{sm,d',mj}=\mc{F}^m\cap\mc{T}_{s,m,d',mj}^B$, for any $r'\in\{0,\ldots,m\}$, since we are always adding or deleting rows from the block-row $R^B_{\delta-j}$.

Let $I_B(\mc{F}^m,d)=\{i\in\{\lceil\frac{d}{m}\rceil,\dots,s\}\,:\,|\Delta_i\cap\mc{F}^m|\geq{dm}\}$. We can note that for any $r'\in\{0,\ldots,m\}$ we have 
$I_B(\mc{F}^m,d')\subseteq I_B(\mc{F}^m,d)$.
Indeed, we have $\left\lceil\frac{d}{m}\right\rceil=\delta\le\left\lceil\frac{d'}{m}\right\rceil$ for any $r'$. So, for $r'\ge r$ we have trivially $\left\{i\in\left\{\left\lceil\frac{d'}{m}\right\rceil,\dots,s\right\}:|\Delta_i\cap\mc{F}^m|\ge md'\right\}\subseteq I_B(\mc{F}^m,d)$.
While for $r'\le r$, due to the block structure we have that if $|\Delta_i\cap\mc{F}^m|\ge md'=m^2(\delta-1)+m(r'+1)$, since $m^2$ divides $|\Delta_i\cap\mc{F}^m|$, then $|\Delta_i\cap\mc{F}^m|\ge m^2\delta\ge md$.

Now, always from the block structure, it follows that if $i\in I_B(\mc{F}^m,d)$, then $\Delta_i\cap\mc{F}^m\supseteq\Delta_i\cap\mc{L}_{sm,d',mj}$, for any $r'\in\{0,\ldots,m\}$.

This implies also
$$\left\{i\in\left\{\left\lceil\frac{d'}{m}\right\rceil,\dots,s\right\}:\Delta_i\cap\mc{F}^m\cap\mc{S}_{sm,d',mj}\neq\emptyset\right\}=I_B(\mc{F}^m,d')$$
for any $r'\in\{0,\ldots,m\}$.

 So, from Lemma \ref{lem:punti} we can conclude that $(\mathcal{F}^m,d')$ is MSRD-constructible for any $r'\in\{0,\ldots,m\}$.

 Finally, from Theorem \ref{th:conndiv} we obtain also that $(\mathcal{F},\delta)$ and $(\mathcal{F},\delta+1)$ are MDS-constructible.
\end{proof}
}

\begin{remark}
    For the case of ($m$-)monotone (and thus convex) Ferrers diagrams, we have seen that for any $\delta\in [s-1]$ there always exists  $j\in\{0,\dots,\delta-1\}$ such that $(\mc{F},\delta)$ is $j$-Singleton and $(\mc{F},\delta+1)$ is $j$-Singleton or $(j+1)$-Singleton. Meaning that, if we want to obtain $\num{F}{\delta+1}$, we just need to consider all $j$ for which $\num{F}{\delta}=\nu_j(\mc{F},\delta)$, delete another column or row, and check which one gives us the minimum. That is, for ($m$-)monotone Ferrers diagrams we have 
    $$
\num{F}{\delta+1}=\min\{\nu_j(\mc{F},\delta+1),\nu_{j+1}(\mc{F},\delta+1)\,:\, (\mc{F},\delta) \text{ is $j$-Singleton}\}.
$$
    It would be interesting to verify if this holds in general for any diagrams.
    If this is the case, {the MDS-constructibility  of $(\mathcal{F},\delta)$ and $(\mathcal{F},\delta+1)$ would be equivalent to the MSRD-constructibility of $(\mathcal{F}^m,m(\delta-1)+r+1)$ for any $r\in\{0,\ldots,m\}$.} 
\end{remark}

\section{Reducing the problem to triangular diagrams}\label{sec:triangular_reduction}

Recall that Theorem \ref{th:optMSRD} gives a construction of MFD codes for MSRD-constructible pairs over fields of size larger than a certain constant. In this section, we take some steps further in removing that restriction on the field size.

First, we recall a result from \cite{bib:neri} adapted to the block notation.

\begin{lemma}[Lemma 4.19, \cite{bib:neri}]
\label{le:subdiblock}
Let $\mc{F}$ be an $m$-block Ferrers diagram of order $sm$, let $d\in\{2,\dots,sm\}$ and let $j\in\{0,\dots,d-1\}$ be such that $(\mc{F},d)$ is $j$-Singleton. Let $\mc{F}'\subseteq\mc{F}$ be a Ferrers diagram of order $sm$ with the property that $\mc{F}\cap\mc{L}_{sm,d,j}=\mc{F'}\cap\mc{L}_{sm,d,j}$. Then $(\mc{F}',d)$ is $j$-Singleton.
\end{lemma}

{
\begin{lemma}\label{lem:diagonals_blocks}
    Let $\mc{F}$ be an $m$-block Ferrers diagram of order $sm$, and let
$d\in\{2,\ldots,sm\}$. Assume that $I_B(\mc{F},d)\neq\emptyset$, and set
$v:=\min I_B(\mc{F},d)$ and $l:=\max I_B(\mc{F},d).$
Then
$$I_B(\mc{F},d)=\{v,v+1,\ldots,l\},$$
and the block diagonal $D_v^B$ is full, i.e. $D_v^B\subseteq \mathbf{B}(\mc{F}).$
Consequently, 
$\Delta_i\subseteq \mc{F}$ for every $i\in[v]$. 
\end{lemma}

\begin{proof}
For every $i \in [s]$ write
$$X_i:=\{a\in[i]\,:\, Q_{a,s-i+a}\in\mathbf{B}(\mc{F})\}$$
and
$$b_i:=|X_i|=|D_i^B\cap \mathbf{B}(\mc{F})|.$$
We first observe that the set $I_B(\mc{F},d)$ is an interval. Indeed, if
$i<k$ and $i,k\in I_B(\mc{F},d)$, then every block diagonal between $D_i^B$ and $D_k^B$ contains at least $\min\{b_i,b_k\}$ blocks of $\mc{F}$, by the
Ferrers diagram property. Hence, every intermediate index also belongs to $I_B(\mc{F},d)$.

Now, let $v:=\min I_B(\mc{F},d)$. We claim that $D_v^B$ is full. If $v=1$, then this is clearly true. Hence, Suppose, by contradiction, that $v>1$ and $D_v^B$ is not full. Then, $X_v\subsetneq[v]$, $b_v<v$ and there exists $a \in [v]\setminus X_v$. Since $\mc{F}$ is a
Ferrers diagram, every block of $D_v^B\cap \mathbf{B}(\mc{F})$, except possibly the
first one, forces a block of $D_{v-1}^B\cap \mathbf{B}(\mc{F})$ by moving one step right in the same block row and a block of $D_{v-1}^B\cap \mathbf{B}(\mc{F})$ by moving one step up in the same block column. 
In particular, 
\begin{enumerate}
    \item[(r)] if $i \in X_v\cap[a-1]$, then $i\in X_{v-1}$;
    \item[(u)] if $i \in X_v\cap\{a+1,\ldots,v\}$, then  $i-1 \in X_{v-1}$.
\end{enumerate}
This gives an injection from $X_v$ to $X_{v-1}$ and hence
$b_{v-1}\ge b_v.$
Therefore,
$mb_{v-1}\ge mb_v\ge d$,
so $v-1\in I_B(\mc{F},d)$, contradicting the minimality of $v$. Thus, $D_v^B$ is full.

Finally, if $D_v^B$ is full, then the Ferrers property implies that all
block diagonals $D_i^B$ with $1\le i\le v$ are full. Equivalently, for every $i\in[v]$,
$\Delta_i\subseteq \mc{F}.$
This concludes the proof.
\end{proof}
}

Now, we are able to give a result that is the analog of Theorem \ref{th:optany} for MSRD-constructible pairs. Observe that, in our case, we have to suppose the existence of MFD codes on block triangular diagrams.
\begin{theorem}
\label{theo:tri}
Let $\mc{F}$ be an $m$-block Ferrers diagram of order $sm$ and $d\in\{2,\dots,sm\}$. % and $\mathbb{F}_q$ be a finite field.
Assume that:
\begin{enumerate}
 \item $(\mc{F},d)$ is MSRD-constructible;
    \item { if $I_B(\mc{F},d)\neq\emptyset$, then there exists a
$[\mc{T}^B_{l,m},\nu_{\min}(\mc{T}^B_{l,m},d),d]_q$
MFD code, where
$l:=\max I_B(\mc{F},d)$.} 

\end{enumerate}
Then, there exists an $[\mathcal{F},{\nu_{\min}(\mathcal{F},d)},d]_q$ MFD code.
\end{theorem}
\begin{proof}
Let $(\mc{F},d)$ be an MSRD-constructible pair and let $j\in\{0, \ldots, d-1\}$ be such that $(\mc{F},d)$ is $j$-Singleton.
By Lemma~\ref{le:nublock}, since $\mc{F}$ is an $m$-block Ferrers diagram, we may choose $j$
in the set $j\in\{km, km+r\,:\, k\in\{0,\ldots,\delta-1\}\}$, where
$$d-1=m(\delta-1)+r,\qquad 0\le r\le m-1.$$

{
Set
$$Y:=\left\{ i\in \left\{ \left\lceil \frac{d}{m} \right\rceil,\ldots,s \right\} \, :\,  \Delta_i\cap \mc{F}\cap \mc{S}_{sm,d,j}\neq\emptyset
\right\}.$$
By Lemma~\ref{lem:punti}, we have $Y=I_B(\mc{F},d)$.

If $I_B(\mc{F},d)=\emptyset$, then $\nu_{\min}(\mc{F},d)=0$, and the zero code is
an $[\mc{F},0,d]_q$ MFD code. Hence suppose that $I_B(\mc{F},d)\neq\emptyset$
and set $l:=\max I_B(\mc{F},d)$ and $v:=\min I_B(\mc{F},d)$. Define
$$\mc{F}' := (\mc{F}\cap \mc{S}_{sm,d,j}) \cup
\left(
\bigcup_{i=1}^l
(\Delta_i\cap \mc{L}_{sm,d,j})
\right).$$
By Lemma~\ref{lem:punti} part \ref{lem:sin3}, we have $\Delta_i\cap \mc{L}_{sm,d,j}\subseteq \mc{F}$ for every $i \in \{v,\ldots,l\}$. Moreover, by Lemma \ref{lem:diagonals_blocks}, $\Delta_i\cap \mc{L}_{sm,d,j}\subseteq \Delta_i\subseteq \mc{F}$ for every $i \in [v]$. Thus, we deduce
$\mc{F}'\subseteq \mc{F}$. Furthermore, since
$l=\max I_B(\mc{F},d)=\max Y$, we have $\mc{F}'\subseteq \mc{T}^B_{l,m}$.

With the choice of $j$ made at the beginning, Proposition~\ref{prop:jsingtriang} implies that the pair $(\mc{T}^B_{l,m},d)$ is $j$-Singleton. Let $\mc{C_T}$ be the MFD code on $\mc{T}^B_{l,m}$, which exists by hypothesis.
Since
$$\mc{F}'\cap \mc{L}_{sm,d,j}=\mc{T}^B_{l,m}\cap \mc{L}_{sm,d,j},$$
Lemma~\ref{le:subdiblock} gives that $(\mc{F}',d)$ is also $j$-Singleton. Hence
$$\nu_{\min}(\mc{T}^B_{l,m},d)=\nu_{\min}(\mc{F}',d)+|\mc{T}^B_{l,m}\setminus \mc{F}'|.$$
By Lemma~\ref{lem:incldots} part~(2),
${\mc{C}}' := \mc{C_T}\cap \mathbb F_q^{\mc{F}'}$
is an $[\mc{F}',\nu_{\min}(\mc{F}',d),d]_q$ MFD code.

Finally, since $\mc{F}'\cap \mc{S}_{sm,d,j}=\mc{F}\cap \mc{S}_{sm,d,j}$, we have
$$\nu_{\min}(\mc{F}',d)=\nu_{\min}(\mc{F},d).$$
Since $\mc{F}'\subseteq \mc{F}$, Lemma~\ref{lem:incldots} part~(1) gives an
$[\mc{F},\nu_{\min}(\mc{F},d),d]_q$ MFD code.}
\end{proof}

\begin{remark}
From Theorem \ref{theo:tri}, we have that, in order to prove Conjecture \ref{con:etzsilb} for the case of MSRD-constructible pairs, it is enough to prove it for the case of block triangular Ferrers diagrams.
\end{remark}

\begin{remark}
\label{rem:ph}
In the case that $m=p^h$ for some prime $p$ and some integer $h>0$, from \cite[Theorem 4.9]{bib:neri} we have that there exists  a $[\mc{T}_{l,m}^B,\nu_{\min}(\mc{T}_{l,m}^B,d),d]_q$ MFD code on any finite field of characteristic $p$ for every $l\ge  1$ and $d\in\{2,\ldots,lm\}$.
\end{remark}
The following corollary is a direct consequence of Theorem \ref{theo:tri} in the particular case of Remark \ref{rem:ph}.
\begin{corollary}
\label{cor:tri2}
Let $p$ be a prime and $h>0$ be an integer, let $\mc{F}$ be a $p^h$-block Ferrers diagram of order $sp^h$, $d\in\{2,\dots,sp^h\}$ and let $\mathbb{F}_q$ be a finite field of characteristic $p$. If $(\mc{F},d)$ is MSRD-constructible, then there exists an $[\mathcal{F},{\nu_{\min}(\mathcal{F},d)},d]_q$ MFD code.  
\end{corollary}

\subsection{MFD codes over block triangular Ferrers diagrams}

From the previous results, we have that solving Conjecture \ref{con:etzsilb} for MSRD-constructible pairs can be reduced to solving it for the case of block triangular Ferrers diagrams. The results of Neri and Stanojkovski \cite{bib:neri} partially solve this problem when $m$ is a prime power, showing the existence of an MFD code over any field with characteristic that divides $m$. 
Another result for block triangular Ferrers diagrams has been obtained in \cite{bib:Liu2023}. In particular, the author showed that there exists a $[\mc{T}_{s,2}^B,\nu_{\min}(\mc{T}_{s,2}^B,4),4]_q$ MFD code over any field $\mathbb{F}_q$ and any positive integer $s$. 
{Moreover, as observed in Remark \ref{rm:monotone_triangular}, using Construction \ref{con:MSRD} -- or, equivalently, \cite[Theorem 5.7]{rakhi2023} -- for any $q>s$, we can obtain an $[\mc{T}_{s,m}^B,\nu_{\min}(\mc{T}_{s,m}^B,d),d]_q$ MFD code for any $d\in[sm]$.}

To the best of our knowledge, there are no other general results concerning block triangular Ferrers diagrams. 

We can note that, for an $m$-block Ferrers diagram $\mc{F}$, when we consider the parameter $d$ we can restrict to the case $d\ge m+2$. Indeed, if $d\le m+1$, then we have always that -- up to considering its standard version $\mc{F}_{\mathrm{st}}$ -- either the last $d-1$ columns or the first $d-1$ rows contain $sm$ elements each one. So, from  \cite[Theorem 2]{bib:etzion1}, there exists an $[\mathcal{F},{\nu_{\min}(\mathcal{F},d)},d]_q$ MFD code over any field $\mathbb{F}_q$.

In this section, we prove Conjecture \ref{con:etzsilb} for MSRD-constructible pairs in the case of $d=sm,sm-1$, with no restrictions on $s,m$ and $q$.

\begin{lemma}\label{lem:TsmMFD}
    Let $m,s\ge 2$. Then, for any prime power $q$ there exists a $[\mc{T}_{s,m}^B,\nu_{\min}(\mc{T}_{s,m}^B,sm),sm]_q$ MFD code.
\end{lemma}
\begin{proof}
    Let $\{1,\gamma,\ldots,\gamma^{s-1}\}$ be a basis of $\mathbb{F}_{q^{sm}}$ over $\mathbb{F}_{q^{m}}$. That is, an element $x$ in $\mathbb{F}_{q^{sm}}$ can be represented as $x= x_0+\gamma x_1+\ldots+\gamma^{s-1}x_{s-1}$.
    
    Then, for any $\alpha\in\mathbb{F}_{q^m}$ we can consider the $\mathbb{F}_q$-linear mapping
    $$\begin{array}{rccl}\psi_{\alpha}:&\F_{q^{sm}}&\rightarrow &\F_{q^{sm}}\\
& x& \longmapsto & \alpha x\end{array}.$$
Now, let $\{\delta_1,\ldots,\delta_m\}$ be a basis of $\mathbb{F}_{q^m}$ over $\mathbb{F}_q$, then $\{\gamma^i\delta_1,\ldots,\gamma^i\delta_m\,:\, i=0,\ldots,s-1\}$ 
is a basis of $\mathbb{F}_{q^{sm}}$ over $\mathbb{F}_q$. Since $\psi_\alpha$ maps $\gamma^i\mathbb{F}_{q^m}$ into itself, we have that $\psi_\alpha$, when represented with respect to the previous basis, is a block triangular matrix with support in $\mc{T}_{s,m}^B$, in particular it is a block diagonal matrix. It is easy to see that such a matrix has full rank, whenever $\alpha\ne 0$, and the set $M:=\{\psi_\alpha\,:\,\alpha\in \mathbb{F}_{q^m}\}$ is an $m$-dimensional $\mathbb{F}_q$-vector space. Therefore, $M$ is a $[\mc{T}_{s,m}^B,\nu_{\min}(\mc{T}_{s,m}^B,sm),sm]_q$ MFD code since $\nu_{\min}(\mc{T}_{s,m}^B,sm)=m$.
\end{proof}

\begin{lemma}\label{lem:Tsm-1MFD}
    Let $m,s\ge 2$. Then, for any prime power $q$ there exists a $[\mc{T}_{s,m}^B,\nu_{\min}(\mc{T}_{s,m}^B,sm-1),sm-1]_q$ MFD code.
\end{lemma}
\begin{proof}
    Let us consider the case $s=2$. Then, let $\gamma\in\mathbb{F}_{q^{2m}}$ such that $\mathbb{F}_{q^{2m}}=\mathbb{F}_{q^{m}}+\gamma\mathbb{F}_{q^{m}}$.
Let $\zeta\in \Fqm$ be such that $\mathrm{Tr}_{q^m/q}(\zeta)\ne 0$.

For $\alpha,\beta \in \Fqm$ we define the $\Fq$-linear map
$$\begin{array}{rccl}\psi_{\alpha, \beta}:&\F_{q^{2m}}&\rightarrow &\F_{q^{2m}}\\
& x+\gamma y& \longmapsto & (\beta x^q-\alpha x)+\gamma(\beta y^q-\alpha y -\alpha \zeta x)\end{array}$$
Let $(\delta_1,\ldots,\delta_m)$ be a basis of $\Fqm$ over $\Fq$, and  consider the basis of $\F_{q^{2m}}$ over $\Fq$ given by $(\delta_1,\ldots,\delta_m,$ $\gamma\delta_1,\ldots,\gamma\delta_m)$.  
Since $\psi_{\alpha,\beta}$ restricted to $\Fqm$ has image in $\Fqm$, then the associated matrix is a block triangular matrix.

Now, we want to show that $\mathrm{ker}(\psi_{\alpha,\beta})$ has at most dimension $1$ over $\Fq$. First note that for any $\lambda\in\Fqm^*$ the maps $\psi_{0,\lambda}$ and $\psi_{\lambda,0}$ are invertible. Thus, we can consider $\alpha,\beta\in\Fqm^*$. Moreover, without loss of generality, we can assume that $\alpha=1$.
Under these assumptions,  we have that 
$$x+\gamma y \in \ker(\psi_{1,\beta})\Longleftrightarrow \left\{\begin{array}{lcl} \beta x^q-x&=&0 \\
\beta y^q-y&=&\zeta x
\end{array}\right. .$$

The linear map $\beta x^q-x$ (over $\Fqm$) is invertible if and only if $N_{q^m/q}(\beta)\ne 1$.
%So, if $N_{q^m/q}(\beta)\ne 1$ we have that $\psi_{1,\beta}$ is invertible.
Hence, suppose that $N_{q^m/q}(\beta)=1$. Thus, there exists $\bar \beta\in\Fqm$ such that $\beta=\bar \beta^q/\bar \beta$.
Hence, $\beta x^q-x=0$ implies that $\bar \beta x\in \Fq$. 
Now, from $\beta y^q-y=\zeta x$ we have that $\bar \beta^q y^q-\bar \beta y=\bar \beta\zeta x$, which implies $\mathrm{Tr}_{q^m/q}(\bar \beta\zeta x)=0$. Since $\bar \beta x\in \Fq$ we get $\mathrm{Tr}_{q^m/q}(\bar \beta\zeta x)=\bar \beta x \mathrm{Tr}_{q^m/q}(\zeta)$=0, and this happens if and only if $x=0$.
Therefore, the dimension of $\ker(\psi_{1,\beta})$ is at most $1$.

Now, let $s>2$ and $\{1,\gamma,\ldots,\gamma^{s-1}\}$ be a basis of $\mathbb{F}_{q^{sm}}$ over $\mathbb{F}_{q^{m}}$. Considering $\zeta$ as above, we define the mapping
$$\begin{array}{rccl}\psi_{\alpha, \beta}:&\F_{q^{sm}}&\rightarrow &\F_{q^{sm}}\\
& \sum\limits_{i=0}^{s-1}\gamma^ix_i& \longmapsto & (\beta x_0^q-\alpha x_0)+\sum\limits_{i=1}^{s-1}\gamma^i(\beta x_i^q-\alpha x_i -\alpha \zeta x_{i-1}).
\end{array}$$

As before, we can note that such linear map is given by a block triangular matrix with support in $\mc{T}_{s,m}^B$. 
Also in this case, we need to investigate the kernel of the mappings $\psi_{1,\beta}$'s. So, we have

$$x_0+\gamma x_1+\ldots+\gamma^{s-1}x_{s-1} \in \ker(\psi_{1,\beta})\Longleftrightarrow \left\{\begin{array}{lcl} \beta x_0^q-x_0&=&0 \\
\beta x_1^q-x_1&=&\zeta x_0\\
\vdots&&\\
\beta x_{s-1}^q-x_{s-1}&=&\zeta x_{s-2}
\end{array}\right.$$
From the above arguments, when $N_{q^m/q}(\beta)=1$, we iteratively obtain that $x_0+\gamma x_1+\ldots+\gamma^{s-1}x_{s-1} \in \ker(\psi_{1,\beta})$ if and only if $x_0=x_1=\ldots=x_{s-2}=0$ and $x_{s-1}\in\ker(\beta x^q-x)$. Therefore, $\ker(\psi_{1,\beta})$ has dimension $1$.

Thus, the set $M:=\{\psi_{\alpha,\beta}\,:\,\alpha,\beta\in\Fqm\}$ is an $\Fq$-vector space of dimension $2m$ and any nonzero matrix in this space has rank at least $sm-1$. Since $\nu_{\min}(\mc{T}_{s,m}^B,sm-1)=2m$, we have that $M$ is a $[\mc{T}_{s,m}^B,\nu_{\min}(\mc{T}_{s,m}^B,sm-1),sm-1]_q$ MFD code.
\end{proof}

Lemmas \ref{lem:TsmMFD} and \ref{lem:Tsm-1MFD} can be summarized in the following theorem. 

\begin{theorem}\label{th:triangular_conj}
Let $m,s\ge 2$. Then, for any prime power $q$ there exists a $[\mc{T}_{s,m}^B,\nu_{\min}(\mc{T}_{s,m}^B,d),d]_q$ MFD code for $d\in\{1,\ldots,m+1\}\cup\{sm-1,sm\}$.
\end{theorem}

As a consequence, from Theorem \ref{theo:tri} and Theorem \ref{th:triangular_conj} we have the following.
\begin{corollary}\label{cor:triangular_construction}
   Let $m,s\ge 2$ and $\mc{F}$ be an $m$-block Ferrers diagram of order $sm$. If $(\mc{F},d)$ is MSRD-constructible, with $d\in\{1,\ldots,m+1\}\cup\{sm-1, sm\}$, then there exists an  $[\mc{F},\nu_{\min}(\mc{F},d), d]_q$ MFD code over any
finite field $\Fq$. 
\end{corollary}

\bibliographystyle{abbrv}
\bibliography{biblio.bib}

\end{document}